\documentclass[sn-mathphys,Numbered]{sn-jnl}

\usepackage{graphicx}
\usepackage{pythonhighlight}
\usepackage{amsmath,amssymb,amsthm}
\usepackage{mathtools}
\usepackage{pythonhighlight}
\usepackage{subfigure}
\usepackage{float}
\usepackage{algorithm}
\usepackage{algorithmicx}
\usepackage{algpseudocode}
\usepackage{enumerate}

\usepackage{multirow}%
\usepackage{amsfonts}%
\usepackage{mathrsfs}%
\usepackage[title]{appendix}%
\usepackage{xcolor}%
\usepackage{textcomp}%
\usepackage{manyfoot}%
\usepackage{booktabs}%
\usepackage{listings}%

\theoremstyle{thmstyleone}%
\newtheorem{theorem}{Theorem}[section]
\newtheorem{proposition}{Proposition}[section]%
\newtheorem{lemma}{Lemma}[section]%
\newtheorem{assumption}{Assumption}[section]%

\theoremstyle{thmstyletwo}%
\newtheorem{remark}{Remark}%

\theoremstyle{thmstylethree}%
\newtheorem{definition}{Definition}[section]%

\renewcommand{\proofname}{\indent{Proof}}

\numberwithin{equation}{section}

\begin{document}

\title[Regularized Proximal Newton-Type Methods for Manifold Optimization]{An Adaptive Regularized Proximal Newton-Type Methods for Composite Optimization over the Stiefel Manifold}


\author[1]{\fnm{Qinsi} \sur{Wang}}\email{qinsiwang20@fudan.edu.cn}

\author*[1]{\fnm{Wei Hong} \sur{Yang}}\email{whyang@fudan.edu.cn}
%

\affil*[1]{\orgdiv{School of Mathematical Sciences}, \orgname{Fudan University}, \orgaddress{\street{220 Handan Street}, \city{Shanghai}, \postcode{200433},\country{China}}}
%
%


\abstract{
Recently, the proximal Newton-type method and its variants have been generalized to solve composite optimization problems over the Stiefel manifold whose objective function is the summation of a smooth function and a nonsmooth function.
In this paper, we propose an adaptive quadratically regularized proximal quasi-Newton method, named ARPQN, to solve this class of problems.
Under some mild assumptions, the global convergence, the local linear convergence rate and the iteration complexity of ARPQN are established. 
Numerical experiments and comparisons with other state-of-the-art methods indicate that ARPQN is very promising.
We also propose an adaptive quadratically regularized proximal Newton method, named ARPN.
It is shown the ARPN method has a local superlinear convergence rate under certain reasonable assumptions, which demonstrates attractive convergence properties of regularized proximal Newton methods.
}





\maketitle

\section{Introduction}
\label{sec:intro}


Composite optimization problems over Riemannian manifolds have received increasing attention in many application fields.
The objective function of such problems is the summation of a smooth function and a nonsmooth function.
In this paper, we consider the following composite optimization problem over the Stiefel manifold $\mathrm{St}(n,r):=\{X\in\mathbb{R}^{n\times r}:X^TX=I_r\}$, which can be formulated as
\begin{equation}
\label{eq:prob}
\mathop{\min}\limits_{X\in\mathrm{St}(n,r)} F(X):=f(X)+h(X),
\end{equation}
where $f:\mathbb{R}^{n\times r}\to\mathbb{R}$ is smooth and $h:\mathbb{R}^{n\times r}\to\mathbb{R}$ is convex but nonsmooth.
Problem \eqref{eq:prob} has arisen in various applications, such as compressed modes problems \cite{cm2013}, sparse principal component analysis \cite{spca2006}, clustering problems \cite{clustering2022} and so on. For more details about applications of composite optimization over the Stiefel manifold, the reader is referred to \cite{absil2008,absil2017,mashiqian2020,dingchao2022}.

Composite optimization with manifold constraints has been extensively studied in recent years and the corresponding methods can be divided into four parts: subgradient methods, operator splitting methods, augmented Lagrangian methods and proximal-type methods.
Here, we review these methods as follows and refer the reader to \cite{mashiqian2020,dingchao2022} for more details.
In \cite{subgradient1998}, Ferreira and Oliveira generalize the subgradient method from Euclidean space to Riemannian manifolds.
Borckmans et al. \cite{subgradient2014} consider nonsmooth optimization problems with equality and inequality constraints.
They show that the nonlinear equality constraint can be handled in the framework of Riemannian
manifolds and develop a feasible subgradient descent algorithm.
In \cite{subgradient2015}, Gorhs and Hosseini propose an $\varepsilon$-subgradient method for minimizing a locally Lipschitz function over Riemmanian manifolds and prove its global convergence.

Operator splitting methods split \eqref{eq:prob} into several terms, each of which
is easier to solve than the original problem.
The most common framework of operator splitting methods is based on the alternating direction method of multipliers (ADMM).
Lai et al. \cite{soc2014} propose a three-block ADMM to solve \eqref{eq:prob}.
In \cite{madmm2016}, Kovnatsky et al. propose a method which uses a two-block ADMM to solve the composite optimization problems.
Augmented Lagrangian (AL) algorithms are well-known numerical methods for constrained optimization.
There are several research works which use AL algorithms to solve \eqref{eq:prob}.
In \cite{yyx2019}, Gao et al. propose a parallelized proximal linearized AL algorithm.
Zhou et al. \cite{dingchao2022} design a manifold-based AL method for \eqref{eq:prob},
in which the AL subproblem is solved by a globalized semismooth Newton method.
It is worth mentioning that the AL method proposed in \cite{cxj2017} can be used for solving \eqref{eq:prob}.

Another classical approach to composite optimization problems is the proximal gradient method.
In \cite{huang2021riemannian}, a Riemannian proximal gradient (RPG) algorithm is proposed by Huang and Wei.
They further propose an inexact version of RPG in \cite{irpg2023}, which solves the Riemannian proximal mapping inexactly.
In \cite{mashiqian2020}, Chen et al. present a retraction-based proximal gradient method, named ManPG,
which can be viewed as an inexact RPG method.
Wang and Yang \cite{manpqn2023} propose a proximal quasi-Newton method, which can accelerate the ManPG method.
In \cite{rpn2023}, a Riemannian proximal Newton method is proposed by Si et al.
and a local superlinear convergence rate is established for their method.

Recently, there are several works which use regularized Newton-type methods to solve unconstrained composite optimization problems.
Ghanbari and Scheinberg \cite{reg2018} propose a general inexact regularized proximal quasi-Newton method and prove its global convergence.
In \cite{reg2019}, Grapiglia and Nesterov study an accelerated regularized Newton method for composite optimization which requires the smooth part has a H{\" o}lder-continuous Hessian and analyze the iteration complexity.
Their method is a natural extension of the algorithm proposed in \cite{reg2017},
whose origin can be traced back to the work on the cubic regularization method \cite{cubic2006}.
In \cite{reg2022}, Aravkin et al. propose a proximal quasi-Newton trust-region method for composite optimization,
which can also be viewed as a regularized proximal quasi-Newton method.

In each iteration,
compared with traditional Newton-type methods,
the subproblems of regularized Newton-type methods can be solved in fewer steps due to the regularization term.
When performing a line search along the descent direction,
regularized Newton-type methods usually need less number of trials during the backtracking procedure.
These advantages make regularized Newton-type methods appropriate for manifold optimization problems.
Recognizing this, Hu et al. \cite{wzw2018} generalize the regularized Newton method for smooth optimization from Euclidean space to Riemannian manifolds.
Numerical experiments indicate that the regularized Riemannian Newton method is very promising.
To solve composite optimization problems over the Stiefel manifold,
we propose an adaptive quadratically regularized proximal quasi-Newton method, named ARPQN.
Specifically, we construct a quadratic subproblem with a regularization term to approximate \eqref{eq:prob} at the iterate $X_k\in\mathcal{M}$,
which can be written as:
\begin{equation}
\label{eq:sub_prob}
\mathop{\min}_{V\in {\rm T}_{X_k}\mathcal{M}}
\phi_k(V):=\langle\nabla f(X_k),V\rangle
+\frac{1}{2}\langle (\mathcal{B}_k+\sigma_kI)[V],V\rangle
+h(X_k+V),
\end{equation}
where $\mathcal{M}$ is used to denote $\mathrm{St}(n,r)$, and $\mathcal{B}_k$ is an approximate operator of the Hessian of $f$ at $X_k$.
The solution of \eqref{eq:sub_prob} is employed as the search direction.
During the iteration process, the regularization parameter $\sigma_k$ is adjusted adaptively to accelerate the convergence of ARPQN.
Numerical results demonstrate that the ARPQN method runs faster than the proximal quasi-Newton algorithm proposed in \cite{manpqn2023}.
We also propose an adaptive quadratically regularized proximal Newton method, named ARPN,
whose subproblem is formed just by replacing the term $h(X_k+V)$ by $h(\mathbf{R}_{X_k}(V))$ in \eqref{eq:sub_prob}.
Under some reasonable assumptions, the local superlinear convergence rate of ARPN is established.

The rest of this paper is organized as follows. Notations and preliminaries are introduced in section \ref{sec:notation}.
The adaptive quadratically regularized proximal quasi-Newton algorithm
(ARPQN) together with its convergence and complexity analysis is presented in section \ref{sec:alg} in details.
Next, we propose an adaptive quadratically regularized proximal Newton method (ARPN) and establish its local superlinear convergence rate in section \ref{sec:new_algo}.
Numerical results of ARPQN compared with other algorithms on test problems are shown in section \ref{sec:5}.
Finally, we end with a brief discussion in section \ref{sec:6}.

\section{Notations and Preliminaries}
\label{sec:notation}

In the following, we briefly introduce the notations, definitions and preliminary concepts about manifold optimization which will be used throughout this paper.

For a manifold $\mathcal{M}$, the tangent space to $\mathcal{M}$ at $X$ is denoted as ${\rm T}_X\mathcal{M}$, which is the set of all tangent vectors to $\mathcal{M}$ at $X$.
The tangent bundle ${\rm T}\mathcal{M}:=\cup_{X\in\mathcal{M}}{\rm T}_X\mathcal{M}$ consists of all tangent vectors to $\mathcal{M}$. The manifold $\mathcal{M}$ is called a Riemannian manifold if its tangent spaces are endowed with a smoothly varying inner product $\langle\xi,\eta\rangle_X$, where $\xi,\eta\in{\rm T}_X\mathcal{M}$.
The induced norm is $\|\xi\|_X=\langle\xi,\xi\rangle_X^{1/2}$.
For ease of notation, we use $\|\xi\|$ and $\langle\xi,\eta\rangle$ instead of $\|\xi\|_X$ and $\langle\xi,\eta\rangle_X$, respectively, if no ambiguity arises.

In \cite[p.46]{absil2008}, the gradient of a smooth function $f$ at $X\in\mathcal{M}$, denoted by $\mathrm{grad}f(X)$, is defined as the unique element of $\mathrm{T}_X\mathcal{M}$ satisfying $\langle \mathrm{grad} f(X),\xi\rangle = \mathrm{D} f(X)[\xi]$ for all~$\xi\in{\rm T}_X\mathcal{M}$,
where $\mathrm{D}f(X)[\xi]$ is the directional derivative of $f$ at $X$ along $\xi$.

\vskip 2mm

\begin{definition}
(Retraction \cite[Definition 4.1.1]{absil2008})
\label{retraction}
A retraction on a manifold $\mathcal{M}$ is a smooth mapping ${\bf R}$ from the tangent bundle ${\rm T}\mathcal{M}$ onto $\mathcal{M}$ with the following properties. Let ${\bf R}_X$ denote the restriction of ${\bf R}$ to ${\rm T}_X\mathcal{M}$.
\begin{enumerate}
\item[(1)] ${\bf R}_X(0_X)=X$, where $0_X$ denotes the zero element of ${\rm T}_X\mathcal{M}$.

\item[(2)] With the canonical identification ${\rm T}_{0_X}({\rm T}_X\mathcal{M})\simeq{\rm T}_X\mathcal{M}$, ${\bf R}_X$ satisfies
\begin{eqnarray}
\nonumber
{\rm D}{\bf R}_X(0_X)={\bf id}_{{\rm T}_X\mathcal{M}},
\end{eqnarray}
where ${\rm D}{\bf R}_X(0_X)$ denotes the differential of the retraction ${\bf R}_X$ at the zero element $0_X\in{\rm T}_X\mathcal{M}$ and ${\bf id}_{{\rm T}_X\mathcal{M}}$ denotes the identity mapping on ${\rm T}_X\mathcal{M}.$
\end{enumerate}
\end{definition}

\begin{remark}
If $\mathcal{M}$ is an embedded submanifold of a Euclidean space $E$,
we can extend $\mathbf{R}_X$ to a smooth mapping $R(X,\xi)$ from $\mathcal{M}\times E$ to $\mathcal{M}$,
which satisfies
$R(X,\xi)=\mathbf{R}_X(\xi)$ for all $\xi\in\mathrm{T}_X\mathcal{M}$.
If no confusion, we also use $\mathbf{R}_X(\xi)$ to denote $R(X,\xi)$ in this paper.
The Euclidean differential of $\mathbf{R}_X$ at $\xi\in E$ is denoted as $\mathbb{D}\mathbf{R}_X(\xi)$.
\end{remark}

\vskip 2mm

\begin{proposition}
\label{prop2_1}
(\cite{2order_boudness2016})
Suppose $\mathcal{M}$ is a compact embedded submanifold of a Euclidean space $E$,
and ${\bf R}$ is a retraction.
Then there exist positive constants $M_1$ and $M_2$ such that for all $X\in\mathcal{M}$ and for all $\xi\in{\rm T}_X\mathcal{M}$,
\begin{eqnarray}
\|{\bf R}_X(\xi)-X\| &\leq& M_1\|\xi\| ,\label{eq5}\\
\|{\bf R}_X(\xi)-X-\xi\| &\leq& M_2\|\xi\|^2.\label{eq6}
\end{eqnarray}
\end{proposition}

For quasi-Newton methods, we need to consider moving a tangent vector along a curve from one tangent space ${\rm T}_X\mathcal{M}$ to another one ${\rm T}_Y\mathcal{M}$. Then we introduce the definition of vector transport.

\vskip 2mm

\begin{definition} (Vector Transport \cite[Definition 8.1.1]{absil2008})
A vector transport associated with a retraction ${\bf R}$ is defined as a continuous function $\mathcal{T}:{\rm T}\mathcal{M}\times {\rm T}\mathcal{M}\rightarrow {\rm T}\mathcal{M}, (\eta_X,\xi_X)\mapsto{\rm T}_{\eta_X}(\xi_X)$, which satisfies the following conditions:

(i)~$\mathcal{T}_{\eta_X}:{\rm T}_X\mathcal{M}\rightarrow {\rm T}_{{\bf R}_X(\eta_X)}\mathcal{M}$ is a linear invertible map,

(ii)~$\mathcal{T}_{0_X}(\xi_X)=\xi_X$.
\end{definition}

\vskip 2mm

Denote $Y:={\bf R}_X(\eta_X)$ where $\eta_X\in {\rm T}_{X}\mathcal{M}$.
For simplicity of notation, we define
$\mathcal{T}_{X,Y}(\xi_X):=\mathcal{T}_{\eta_X}(\xi_X)$ where $\xi_X\in {\rm T}_{X}\mathcal{M}$.
In the rest of the paper, we use $\mathcal{M}$ to denote the Stiefel manifold $\mathrm{St}(n,r)$.
For $\xi,\eta\in  {\rm T}_X\mathcal{M}$,
the inner product is defined by $\langle\xi,\eta\rangle:=tr(\xi^T\eta)$, which is inherited from the embedding Euclidean space $\mathbb{R}^{n\times r}$. The induced norm $\|\xi\|=\langle\xi,\xi\rangle^{1/2}$ is just the Frobenius norm $\|\cdot\|_F$.
According to \cite[p.42]{absil2008}, the tangent space of the Stiefel manifold at $X$ can be written as
 \[
 {\rm T}_X\mathcal{M}=\{V : V^\top X+X^\top V=0\}.
 \]
By \cite[(3.35)]{absil2008}, the orthogonal projection of $V\in\mathbb{R}^{n\times r}$ onto the tangent space ${\rm T}_X\mathcal{M}$ can be formulated as
\begin{equation}
\label{eq:proj}
{\rm Proj}_{{\rm T}_X\mathcal{M}} V = V-\frac{1}{2}X(X^\top V+V^\top X).
\end{equation}
From \cite[(3.37)]{absil2008}, we know that the Riemannian gradient of $f$ at $X$ is equal to the orthogonal projection of $\nabla f(X)$ onto ${\rm T}_X\mathcal{M}$, where $\nabla f(X)$ denotes the Euclidean gradient of $f$ at $X$,
that is
\begin{equation}
\nonumber
{\rm grad} f(X) = {\rm Proj}_{{\rm T}_X\mathcal{M}} \nabla f(X).
\end{equation}
Let $X^*$ be a local optimal solution of \eqref{eq:prob}. By \cite[Theorem 5.1]{ywh2013} (or \cite{hosseini2011}), we can obtain the first-order necessary condition of the problem \eqref{eq:prob}:
\begin{equation}\label{eq:2_opt}
0\in{\rm grad}f(X^*)+{\rm Proj}_{{\rm T}_{X^*}\mathcal{M}} (\partial h(X^*)).
\end{equation}

\section{An Adaptive Regularized Proximal Quasi-Newton Method}\label{sec:alg}

In this section, inspired by the methods proposed in \cite{reg2018} and \cite{wzw2018}, we design an adaptive regularized proximal quasi-Newton method, named ARPQN, for the composite optimization problem \eqref{eq:prob}.
The ARPQN algorithm is introduced in subsection \ref{subsec:algo1} and its convergence and iteration complexity are analyzed in subsections \ref{sec:32} and \ref{sec:33}, respectively.

\subsection{The Algorithmic Framework}
\label{subsec:algo1}

As stated in the section \ref{sec:intro}, at the $k$-th iterate $X_k$, we construct the subproblem \eqref{eq:sub_prob}, in which the approximate Hessian operator $\mathcal{B}_k$ is updated by a damped LBFGS method.
In the following, we introduce the method briefly. For more details, we refer to \cite[section 3.2]{manpqn2023}.
For ease of notation, denote $\mathcal{T}_{k,k+1}:=\mathcal{T}_{X_{k},X_{k+1}}$ and $g_k={\rm grad}f(X_k)$.
Let
\begin{eqnarray}
\label{eq:skyk}
S_{k}&:=& \mathcal{T}_{k,k+1}({\bf R}_{X_{k}}^{-1}(X_{k+1})), \ Y_{k} := g_{k+1}-\mathcal{T}_{k,k+1}(g_{k}).
\end{eqnarray}
For any $h\in{\rm T}_{X_k}\mathcal{M}$, we use $h^{\flat}$ to denote the linear function on ${\rm T}_{X_k}\mathcal{M}$ induced by $h^{\flat}\eta:=\langle h,\eta\rangle$ for all $\eta\in{\rm T}_{X_k}\mathcal{M}$.
Given an initial estimate $\mathcal{B}_{k,0}$, the LBFGS strategy for updating $\mathcal{B}_{k,i}$ can be formulated as
\begin{eqnarray}
\label{eq3_bk}
\mathcal{B}_{k,i}
&=&
\widetilde{\mathcal{B}}_{k,i-1}
- \frac{\widetilde{\mathcal{B}}_{k,i-1}S_j(\widetilde{\mathcal{B}}_{k,i-1}S_j)^{\flat}}{\mathrm{tr}(S_j^{\flat}\widetilde{\mathcal{B}}_{k,i-1}S_j)}
+ \rho_jY_jY_j^{\flat},\\
\nonumber j
&=& k-(p-i+1),\ i=1,\dots,p,
\end{eqnarray}
where
\[
\widetilde{\mathcal{B}}_{k,i-1}=\mathcal{T}_{j,j+1}\circ\mathcal{B}_{k,i-1}\circ\mathcal{T}_{j,j+1}^{-1} \ {\rm and} \ \rho_j=\frac{1}{\mathrm{tr}(Y_j^{\flat}S_j)},
\]
and $p$ is the memory size for LBFGS method.
Then we set the approximated operator $\mathcal{B}_k:=\mathcal{B}_{k,p}$.

To reduce the computational cost both for updating $\mathcal{B}_k$ and for solving the subproblem \eqref{eq:sub_prob}, we use a easily computed $\mathbf{B}_k$ to approximate $\mathcal{B}_k$ where
\begin{eqnarray}\label{hanwu}
\mathbf{B}_{k}[V] ={\rm Proj}_{{\rm T}_{X_k}\mathcal{M}}\big(({\rm diag}B_k)V\big),
\end{eqnarray}
in which $V\in {\rm T}_{X_k}\mathcal{M}$ and
$B_k\in\mathbb{R}^{n\times n}$ is a positive definite symmetric matrix so that the solution $V_k$ of \eqref{eq:sub_prob} is the descent direction of $F$ at $X_k$.
Then, it holds that
\begin{equation}
\label{eq:3_bkdiag}
{\rm tr}(V^T\mathbf{B}_k[V])={\rm tr}(V^T({\rm diag}B_k)V),\quad\forall~ V\in{\rm T}_{X_k}\mathcal{M}.
\end{equation}
Next we show how to construct $B_k$.

Since it is expensive to calculate \eqref{eq3_bk}, we use the Euclidean difference $s_k:=X_{k+1}-X_k\in\mathbb{R}^{n\times r}$ and $y_k:=g_{k+1}-g_k\in\mathbb{R}^{n\times r}$ to replace $S_k$ and $Y_k$ in \eqref{eq:skyk}. The damped technique introduced in \cite{damp1978,damp2017} is employed to guarantee the positive definiteness of $B_k$.
Specifically, define $\overline{y}_{k-1} = \beta_{k-1}{y}_{k-1} + (1-\beta_{k-1})H_{k,0}^{-1}{s}_{k-1}$, where the initial estimate $H_{k,0}$ is set to be $(1/\vartheta_k)I$ for some $\vartheta_k>0$, and
\begin{eqnarray*}
\beta_{k-1}=
\left\{
\begin{array}{ll}
\frac{0.75{\rm tr}(s_{k-1}^TH_{k,0}^{-1}s_{k-1})}{{\rm tr}(s_{k-1}^TH_{k,0}^{-1}s_{k-1}) - {\rm tr}(s_{k-1}^Ty_{k-1})},
&\textrm{if $ {\rm tr}(s_{k-1}^Ty_{k-1}) <  0.25 {\rm tr}(s_{k-1}^TH_{k,0}^{-1}s_{k-1})$};\\
1,&\textrm{otherwise}.
\end{array}
\right.
\end{eqnarray*}
Then $B_k:=B_{k,p}$ is computed by
\begin{eqnarray}\label{eq3_eu_bk}
\left\{\begin{array}{l}
B_{k,0}=\vartheta_k I_n,\\
B_{k,i}=B_{k,i-1}-\frac{B_{k,i-1}s_{j}s_{j}^TB_{k,i-1}}{{\rm tr}(s_{j}^TB_{k,i-1}s_{j})}
+\frac{\overline{y}_{j}\overline{y}_{j}^T}{{\rm tr}(s_{j}^T\overline{y}_{j})},\\
j=k-(p-i+1),i=1,\dots,p.
\end{array}
\right.
\end{eqnarray}

Denote $\|V\|_{\mathbf{B}_{k}}^2:={\rm tr}(V^T\mathbf{B}_k[V])$. Next lemma shows that $\|V\|_{\mathbf{B}_{k}}^2=O(\|V\|^2)$.

\vskip 2mm

\begin{lemma}
\label{lm_bk}
(\cite[Lemma 3.1]{manpqn2023})
Suppose that $\nabla f$ is Lipschitz continuous.
Suppose $\mathbf{B}_{k}$ and $B_k$ are defined by \eqref{hanwu} and \eqref{eq3_eu_bk} respectively.
Then for all $k$, there exist $0<\kappa_1<\kappa_2$ such that
\begin{eqnarray}\label{kan}
\kappa_1\|V\|^2
\leq\|V\|_{\mathbf{B}_{k}}^2\leq \kappa_2\|V\|^2,
\quad\forall~V\in {\rm T}_{X_k}\mathcal{M}.
\end{eqnarray}
\end{lemma}

We now give a brief description of our algorithm.
Denote $V_k$ as the exact solution of the subproblem \eqref{eq:sub_prob}. After $V_k$ is obtained,
we apply a retraction-based nonmonotone backtracking line-search technique, which is introduced in \cite{gll1986}, to determine the stepsize $\alpha_k$. The stepsize is set to be $\alpha_k=\gamma^{N_k}$ where $N_k$ is the smallest integer such that
\begin{equation}\label{eq:line_search}
F({\bf R}_{X_k}(\alpha_k V_k))
\leq
\max_{\max\{0,k-m\}\leq j\leq k}F(X_j)
-\frac{1}{2}\sigma\alpha_k\|V_k\|_{\mathcal{B}_k}^2,
\end{equation}
in which $\sigma,\gamma\in(0,1)$ are parameters for line-search strategy.
To simplify the notation, we denote
\begin{equation}
\label{eq:3_nls_index}
l(k):=\mathop{\arg\max}\limits_{\max\{0,k-m\}\leq j\leq k}F(X_j),
\end{equation}
and then we have $F(X_{l(k)})=\max_{\max\{0,k-m\}\leq j\leq k}F(X_j)$.

Then, we define the following ratio
\begin{equation}
\label{eq:3_rhok}
\rho_k:=\frac{F({\bf R}_{X_k}(\alpha_k V_k))-F(X_{l(k)})}{\phi_k(\alpha_kV_k)-\phi_k(0)},
\end{equation}
which will be calculated at each iteration.
The ratio $\rho_k$ describes how well the model \eqref{eq:sub_prob} approximates the problem \eqref{eq:prob} at the current iterate $X_k$.
If $V_k$ is nonzero, $\rho_k$ is always positive since \eqref{eq:line_search} guarantees a strict reduction of $F$ and $\phi_k(\alpha_kV_k)<\phi_k(0)$,
which will be proved in subsection \ref{sec:32}.
When $\rho_k$ is positive but close to $0$, 
there is not good agreement between \eqref{eq:prob} and \eqref{eq:sub_prob} over this step. 
Then, we will enlarge the regularization parameter $\sigma_k$ and compute $V_k$, $\alpha_k$ and $\rho_k$ repeatedly until $\rho_k$ is sufficiently large.
Otherwise,
we can obtain the new iterate $X_{k+1}={\bf R}_{X_k}(\alpha_k V_k)$ and when $\rho_k$ is too large, we will shrink $\sigma_k$.
We describe our method in Algorithm \ref{alg:ada_manpqn}.

\begin{algorithm}
\caption{An adaptive regularized proximal quasi-Newton algorithm for Riemannian composite optimization (ARPQN)}
\label{alg:ada_manpqn}
\begin{algorithmic}[1]
\Require Initial point $X_0\in\mathcal{M}$, initial regularization parameter $\sigma_0>0$, line-search parameters $\sigma,\gamma\in\left(0,1\right)$, $0<\eta_1<\eta_2<1$ and $0<\gamma_1<1<\gamma_2$.
\For{$k = 0,1,\dots$}
	\If{$k\geq1$}
	\State Update $\mathcal{B}_k$ by the quasi-Newton method;
	\Else
	\State Set $\mathcal{B}_k=I$;
	\EndIf
	\While{$1$}
        		\State Solve the subproblem \eqref{eq:sub_prob} to obtain the search direction $V_k$;
        		\State Set the initial stepsize $\alpha_k=1$;
                \While{\eqref{eq:line_search} is not satisfied}
                		\State $\alpha_k\gets\gamma\alpha_k$;
                \EndWhile
                \State Set $Z_{k}={\bf R}_{X_k}(\alpha_k V_k)$;
                \State Calculate the ratio $\rho_k$ by \eqref{eq:3_rhok};
                \If{$\rho_k\geq\eta_1$}
                		\If{$\rho_k\geq\eta_2$}
				\State Update $\sigma_{k}\gets\gamma_1\sigma_k$;
			\EndIf
                		\State break;
		\Else
		\State Update $\sigma_{k}\gets\gamma_2\sigma_k$;
		\EndIf
        \EndWhile
        \State Set $X_{k+1}=Z_k$ and $\sigma_{k+1}=\sigma_k$;
\EndFor
\end{algorithmic}
\end{algorithm}

\subsection{Global Convergence and Convergence Rate Analysis}
\label{sec:32}

In this subsection, we prove the global convergence of Algorithm \ref{alg:ada_manpqn}.
Moreover, a local linear convergence rate of Algorithm \ref{alg:ada_manpqn} is established under some conditions.
First, we state some standard assumptions on the problem which will be used in the rest of this paper.

\vskip 2mm

\begin{assumption}
\label{ass:1}
Let $\{X_k\}_k$ be the sequence generated by Algorithm \ref{alg:ada_manpqn}.

\begin{enumerate}
\item[\rm(A.1)]
$f:\mathbb{R}^{n\times r}\to\mathbb{R}$ is a continuously differentiable function, and $\nabla f$ is Lipschitz continuous with Lipschitz constant $L$.
\item[\rm(A.2)]
$h:\mathbb{R}^{n\times r}\to\mathbb{R}$ is a convex but nonsmooth function, and $h$ is Lipschitz continuous with Lipschitz constant $L_h$.
\item[\rm(A.3)]
There exist $0<\kappa_1<\kappa_2$ such that for all $k\geq0$,
\begin{equation}\label{eq:bk_norm}
\kappa_1\|V\|^2
\leq\langle \mathcal{B}_k[V],V\rangle
\leq \kappa_2\|V\|^2,
\quad\forall~V\in {\rm T}_{X_k}\mathcal{M}.
\end{equation}
\item[\rm(A.4)]
The optimal solution $V_k$ of \eqref{eq:sub_prob} satisfies $\|V_k\|\neq0$ for any $k\geq0$.
\end{enumerate}
\end{assumption}

\vskip 2mm

\begin{remark}
In Assumption \ref{ass:1}, \rm{(A.1)} and \rm{(A.2)} are standard assumptions for convergence analysis of composite optimization.
In our implementation, we use $\mathbf{B}_k$ to replace $\mathcal{B}_k$ where $\mathbf{B}_k$ is defined in \eqref{hanwu}. By Lemma \ref{lm_bk}, we know that $\mathbf{B}_k$ satisfies \rm{(A.3)}.
For \rm{(A.4)}, if $V_k=0$ then $X_k$ satisfies the first-order necessary condition \eqref{eq:2_opt}.
\end{remark}

By Assumption \ref{ass:1}, $\phi_k$ in \eqref{eq:sub_prob} is strongly convex,
and therefore \eqref{eq:sub_prob} has a unique solution,
which is denoted by $V_k$. Then, by \eqref{eq:2_opt}, we can deduce that
\begin{equation}
\label{eq:4_sub_opt}
0\in{\rm Proj}_{{\rm T}_{X_k}\mathcal{M}}\partial\phi_k(V_k)={\rm grad} f(X_k) +(\mathcal{B}_k+\sigma_kI)[V_k]+{\rm Proj}_{{\rm T}_{X_k}\mathcal{M}}\partial h(X_k+V_k).
\end{equation}
Thus, $V_k=0$ is equivalent to that $X_k$ satisfies \eqref{eq:2_opt} and therefore $X_k$ is a stationary point of \eqref{eq:prob};
If $V_k\neq0$, similar to the proof of \cite[Lemma 5.1]{mashiqian2020} (or \cite[Lemma 4.1]{manpqn2023}),
we can show that $V_k$ provides sufficient decrease in $\phi_k$. For completeness, we give a proof here.

\vskip 2mm

\begin{lemma}
\label{lm:41}
Suppose Assumption \ref{ass:1} holds. For any $\alpha\in \left[0,1\right]$, it holds that
\begin{equation}\label{eq:phi_alpha}
\phi_k(\alpha V_k)-\phi_k(0)
\leq
\frac{\alpha(\alpha-2)}{2}\langle(\mathcal{B}_k+\sigma_kI)[V_k],V_k\rangle.
\end{equation}
\end{lemma}

\begin{proof}
By \eqref{eq:4_sub_opt}, there exists $\xi\in\partial h(X_k+V_k)$ such that
\[
{\rm grad}f(X_k) +(\mathcal{B}_k+\sigma_kI)[V_k]+{\rm Proj}_{{\rm T}_{X_k}\mathcal{M}}\xi=0.
\]
From $\xi+\nabla f(X_k)\in\partial\left(\phi_k-\frac{1}{2}\langle(\mathcal{B}_k+\sigma_kI)[\cdot],\cdot\rangle\right)(V_k)$,
it follows that
\begin{eqnarray}
\nonumber
&&
\phi_k(0) -\phi_k(V_k)
\geq  \langle\nabla f(X_k)+\xi,-V_k\rangle-\frac{1}{2}\langle(B_k+\sigma_kI)[V_k],V_k\rangle
\\
&=& \langle{\rm grad}f(X_k)+(\mathcal{B}_k+\sigma_kI)[V_k]
+{\rm Proj}_{{\rm T}_{X_k}\mathcal{M}}\xi,-V_k\rangle
+\frac{1}{2}\langle(\mathcal{B}_k+\sigma_kI)[V_k],V_k\rangle \nonumber\\
&=& \frac{1}{2}\langle(\mathcal{B}_k+\sigma_kI)[V_k],V_k\rangle.\label{eq33_1}
\end{eqnarray}
Since $h$ is a convex function, for all $0\leq\alpha\leq1$, we have
\begin{equation}\label{eq33_3}
h(X_k+\alpha V_k)-h(X_k)
\leq 
\alpha\left(h(X_k+V_k)-h(X_k)\right).
\end{equation}
Combining \eqref{eq33_1} and \eqref{eq33_3} yields
\begin{eqnarray}
\nonumber
\phi_k(\alpha V_k) -\phi_k(0)
&=&
\langle\nabla f(X_k),\alpha V_k\rangle
+ \frac{1}{2}\langle(\mathcal{B}_k+\sigma_kI)[\alpha V_k],\alpha V_k\rangle + h(X_k+\alpha V_k)-h(X_k)
\\
\nonumber
&\leq&
\alpha\left(\langle\nabla f(X_k),V_k\rangle
+\frac{\alpha}{2}\langle(\mathcal{B}_k+\sigma_kI)[V_k],V_k\rangle
+h(X_k+V_k)-h(X_k)
\right)
\\
\nonumber
&=&
\alpha\left(\phi_k(V_k)-\phi_k(0)+\frac{\alpha-1}{2}\langle(\mathcal{B}_k+\sigma_kI)[V_k],V_k\rangle\right)\\
\nonumber
&\leq&  \frac{\alpha(\alpha-2)}{2}\langle(\mathcal{B}_k+\sigma_kI)[V_k],V_k\rangle.
\end{eqnarray}
The assertion holds.
\end{proof}

An important part of convergence analysis of regularized Newton-type methods is to prove boundedness of $\sigma_k$.
We begin with a preparatory lemma.
By the procedure of Algorithm \ref{alg:ada_manpqn},
if $\rho_k<\eta_1$, then $\sigma_k$ will increase by $\gamma_2$ times (see step 21).
In the following result, we prove that when $\sigma_k$ is sufficiently large,
it holds that $\rho_k\geq\eta_1$. Thus, the inner loop (steps 7--23) of Algorithm \ref{alg:ada_manpqn} will terminate in finite steps.
Since $\mathcal{M}$ is compact, we can define
$$
\varrho:=\sup_{X\in\mathcal{M}}\|\nabla f(X)\|.
$$
Recall that $M_1$ and $M_2$ are parameters defined by \eqref{eq5} and \eqref{eq6}. In the rest of the paper,
we use the notation
\begin{eqnarray}
c_1:=\varrho M_2+\frac{1}{2}LM_1^2,\quad~~~c_2:=c_1+L_hM_2.\label{c2}
\end{eqnarray}

\begin{lemma}
\label{lm:42}
Suppose Assumption \ref{ass:1} holds.
If
\begin{equation}\label{eq:over_sigma}
\sigma_k\geq\overline{\sigma}:=
-\kappa_1+\frac{2c_2}{2-\eta_1},
\end{equation}
where $\kappa_1$ is introduced in \eqref{eq:bk_norm},
then $\rho_k\geq\eta_1$.
\end{lemma}

\begin{proof}
Since $\nabla f$ is Lipschitz continuous with constant $L$, for any $\alpha>0$, we have
\begin{eqnarray}
\nonumber
&&
f({\bf R}_{X_k}(\alpha V_k))
\\
\nonumber
&\leq&
f(X_k)+\langle\nabla f(X_k),{\bf R}_{X_k}(\alpha V_k)-X_k\rangle+\frac{L}{2}\|{\bf R}_{X_k}(\alpha V_k)-X_k\|^2
\\
\nonumber
&\leq&
f(X_k)+\langle\nabla f(X_k),{\bf R}_{X_k}(\alpha V_k)-X_k-\alpha V_k\rangle+\langle\nabla f(X_k),\alpha V_k\rangle+\frac{1}{2}LM_1^2\|\alpha V_k\|^2,
\\
\nonumber
&\leq&
f(X_k)+\langle\nabla f(X_k),\alpha V_k\rangle+(\varrho M_2+\frac{1}{2}LM_1^2)\|\alpha V_k\|^2
\\
\label{eq34_1}
&=&
f(X_k)+\langle\nabla f(X_k),\alpha V_k\rangle+c_1\|\alpha V_k\|^2,
\end{eqnarray}
where the second and the last inequality use \eqref{eq5} and \eqref{eq6}.
Since $h$ is Lipschitz continuous with constant $L_h$, taking into account \eqref{eq6}, we can obtain
\begin{equation}
 \label{eq34_1_h}
 h({\bf R}_{X_k}(\alpha V_k))-h(X_k+\alpha V_k)
\leq
L_h\|{\bf R}_{X_k}(\alpha V_k)-X_k-\alpha V_k\|
\\
\leq
L_hM_2\|\alpha V_k\|^2.
\end{equation}
From \eqref{eq34_1} and \eqref{eq34_1_h}, it follows that for any $\alpha>0$,
\begin{eqnarray}
\nonumber
&&
F({\bf R}_{X_k}(\alpha V_k))
\\
\nonumber
&\leq&
f(X_k)+\langle\nabla f(X_k),\alpha V_k\rangle+(c_1+L_hM_2)\|\alpha V_k\|^2 +h(X_k+\alpha V_k)
\\
\label{eq:4_10}
&= &
f(X_k)+\phi_k(\alpha V_k) +c_2\|\alpha V_k\|^2 - \frac{1}{2}\langle(\mathcal{B}_k+\sigma_kI)[\alpha V_k],\alpha V_k\rangle
\\
\label{eq:decrease_f}
&\leq&
F(X_{l(k)})+\phi_k(\alpha V_k)-\phi_k(0)
+(c_2-\frac{1}{2}(\kappa_1+\sigma_k))\|\alpha V_k\|^2.
\end{eqnarray}
Assume that \eqref{eq:over_sigma} holds.
By \eqref{eq:phi_alpha}, \eqref{eq:decrease_f} and $\alpha_k\in\left(0,1\right]$,
we can deduce that
\begin{eqnarray}
\nonumber
1-\rho_k
&=&
\frac{F({\bf R}_{X_k}(\alpha_k V_k))-F(X_{l(k)})-\phi_k(\alpha_kV_k)+\phi_k(0)}{-\phi_k(\alpha_kV_k)+\phi_k(0)}
\\
\nonumber
&\leq&
\frac{(c_2-\frac{1}{2}(\kappa_1+\sigma_k))\|\alpha_k V_k\|^2}{\frac{\alpha_k(2-\alpha_k)}{2}\langle(\mathcal{B}_k+\sigma_kI)[V_k],V_k\rangle}
\leq
\frac{(2c_2-\kappa_1-\sigma_k)\alpha_k}{(2-\alpha_k)(\kappa_1+\sigma_k)}
\leq
\frac{2c_2-\kappa_1-\sigma_k}{\kappa_1+\sigma_k}
\\
\nonumber
&\leq&
1-\eta_1,
\end{eqnarray}
where the second inequality uses \eqref{eq:bk_norm} and the last inequality is due to \eqref{eq:over_sigma}.
Thus, the assertion holds.
\end{proof}

\vskip 2mm

With the help of Lemma \ref{lm:42}, we can prove that the sequence $\{\sigma_k\}_k$ is bounded.

\vskip 2mm

\begin{lemma}
\label{lm:43}
Suppose that Assumption \ref{ass:1} holds. Then,
\begin{eqnarray}\label{si}
\sigma_k\leq\max\{\sigma_0,\gamma_2\overline{\sigma}\}, \quad\forall~k\geq0,
\end{eqnarray}
where $\overline{\sigma}$ is defined by \eqref{eq:over_sigma}.
\end{lemma}

\begin{proof}
The proof is by induction. It is obvious that \eqref{si} holds for $k=0$.
Assume that it holds for $k=j$.
We consider two cases of the value of $\sigma_j$. (1) If $\sigma_j<\overline{\sigma}$,
by the procedure of Algorithm \ref{alg:ada_manpqn} (see steps 15--22) and Lemma \ref{lm:42}, we have $\sigma_{j+1}<\gamma_2\overline{\sigma}$.
(2) If $\sigma_j\geq\overline{\sigma}$,
from Lemma \ref{lm:42}, it follows that $\rho_j\geq\eta_1$.
Thus, by steps 15--19 of Algorithm \ref{alg:ada_manpqn}, we have
\[
\sigma_{j+1}
\leq\sigma_{j}
\leq\max\{\sigma_{0}, \gamma_2\overline{\sigma}\}.
\]
Then, \eqref{si} holds for $k=j+1$. The assertion holds.
\end{proof}

Next, we prove the global convergence of Algorithm \ref{alg:ada_manpqn}.

\begin{theorem}
\label{thm:32_1}
Suppose Assumption \ref{ass:1} holds.
Then $\alpha_k\geq\gamma\overline{\alpha}$, where $\gamma$ is stated in the step 11 of Algorithm \ref{alg:ada_manpqn},
and
\begin{equation}
\label{eq:over_alpha}
\overline{\alpha}:=\min\{1,\frac{(2-\sigma)\kappa_1}{2c_2}\},
\end{equation}
where $\sigma$ is the parameter in \eqref{eq:line_search} and $c_2$ is defined in \eqref{c2}.
Thus the backtracking line search procedure (steps 10--12 of Algorithm \ref{alg:ada_manpqn}) will terminate in finite steps.
Moreover, we have $\lim_{k\to\infty}\|V_k\|=0$ and all accumulation points of $\{X_k\}$ are stationary points of problem \eqref{eq:prob}.
\end{theorem}

\begin{proof}
From $\sigma<1$ and \eqref{eq:over_alpha}, it follows that $\overline{\alpha}>0$.
For any $0<\alpha\leq 1$, by \eqref{eq:4_10} and taking into account \eqref{eq:phi_alpha}, we can deduce that
\begin{eqnarray}
\nonumber
F({\bf R}_{X_k}(\alpha V_k))
&\leq&
F(X_k)+\phi_k(\alpha_kV_k)-\phi_k(0)+(c_2-\frac{1}{2}\sigma_k)\alpha^2\|V_k\|^2-\frac{1}{2}\alpha^2\|V_k\|^2_{\mathcal{B}_k}
\\
\nonumber
&\leq&
F(X_{l(k)})+(c_2\alpha-\sigma_k)\alpha\|V_k\|^2-\alpha\|V_k\|^2_{\mathcal{B}_k},
\end{eqnarray}
where $l(k)$ is defined in \eqref{eq:3_nls_index}.
By the above inequality and \eqref{eq:bk_norm}, we can see that if $0<\alpha\leq\overline{\alpha}$,
\begin{eqnarray}
\nonumber
F({\bf R}_{X_k}(\alpha V_k))
&\leq&
F(X_{l(k)})+\frac{(2-\sigma)\kappa_1}{2}\alpha\|V_k\|^2-\alpha\|V_k\|^2_{\mathcal{B}_k}
\\
\nonumber
&\leq&
F(X_{l(k)})+(\frac{(2-\sigma)}{2}-1)\alpha\|V_k\|^2_{\mathcal{B}_k}
=
F(X_{l(k)})-\frac{1}{2}\sigma\alpha\|V_k\|^2_{\mathcal{B}_k},
\end{eqnarray}
which implies that \eqref{eq:line_search} holds for any $\alpha\in(0,\overline{\alpha}]$.
Thus, from steps 10--12 of Algorithm \ref{alg:ada_manpqn}, we must have $\alpha_k\geq\gamma\overline{\alpha}$.

Note that $X_k$ is updated only when $\rho_k\geq\eta_1$.
Thus, by \eqref{eq:3_nls_index}, \eqref{eq:3_rhok} and \eqref{eq:phi_alpha}, we have
\begin{eqnarray}
\nonumber
F(X_{l(k+1)})
&\leq&
\max\{F(X_{k+1}),F(X_{l(k)})\}
\\
\nonumber
&\leq&
\max\{F(X_{l(k)})-\eta_1(\phi_k(0)-\phi_k(\alpha_kV_k)), F(X_{l(k)})\}=F(X_{l(k)}).
\end{eqnarray}
Then $\{F(X_{l(k)})\}_k$ is a non-increasing sequence.
Using the same argument as that in \cite[Theorem 4.1]{manpqn2023},
we can prove that $\lim_{k\to\infty}\|V_k\|=0$.

Let $X^*$ be an accumulation point of sequence $\{X_k\}$.
By \eqref{eq:4_sub_opt} and $\lim_{k\to\infty}\|V_k\|=0$, we know that $X^*$ satisfies \eqref{eq:2_opt}. The proof is complete.
\end{proof}

\vskip 2mm

For composite optimization in Euclidean space, a linear convergence rate is obtained for the regularized proximal quasi-Newton methods under the condition that $f$ is strongly convex in \cite[Theorem 3]{reg2018}.
For composite optimization over the Stiefel manifold,
we can only obtain a local linear convergence rate for ARPQN under the following assumption, which is used in \cite{manpqn2023} to prove the local linear convergence of the proximal quasi-Newton algorithm.

\vskip 2mm

\begin{assumption}
\label{ass:2}
The function $f$ is twice continuously differentiable. The sequence $\{X_k\}$ has an accumulation point $X^*$ such that
\begin{equation}\label{eq:4_hess}
\lambda_{\min}({\rm Hess} (f\circ{\bf R}_{X^*})(0_{X^*}))\geq\delta,
\end{equation}
where $\delta>5L_hM_2$.
\end{assumption}

\vskip 2mm

By Lemma \ref{lm:43} and Assumption \ref{ass:1}, we have
$\kappa_1\|V\|^2
\leq\langle (\mathcal{B}_k+\sigma_kI)[V],V\rangle
\leq (\kappa_2+\max\{\sigma_0,\gamma_2\overline{\sigma}\})\|V\|^2$ for all $V\in {\rm T}_{X_k}\mathcal{M}$.
Thus, all conditions of \cite[Theorem 4.3]{manpqn2023} are satisfied.
By Theorems 4.2 and 4.3 in \cite{manpqn2023}, we have the following result.

\vskip 2mm

\begin{theorem}
\label{thm:423}
Suppose Assumptions \ref{ass:1} and \ref{ass:2} hold, and $X^*$ is the accumulation point of $\{X_k\}$ which satisfies \eqref{eq:4_hess}.
Then, $X_k$ converges to $X^*$ and there exist $K>0$ and $\tau\in(0,1)$ such that
\begin{equation}\label{eq:4_linear}
F(X_k)-F(X^*)\leq \tau^{k-K}\big(F(X_{l(K)})-F(X^*)\big),\quad \forall~k\geq K.
\end{equation}
\end{theorem}

\subsection{Complexity Analysis}
\label{sec:33}
In this subsection, we analyze the iteration complexity of Algorithm \ref{alg:ada_manpqn}.
Note that the complexity result in this subsection does not need Assumption \ref{ass:2}.

\vskip 2mm

\begin{definition}[$\epsilon$-stationary point \cite{mashiqian2020}]
Given $\epsilon>0$ and a point $X_k$ generated by Algorithm \ref{alg:ada_manpqn},
$X_k$ is an $\epsilon$-stationary point of \eqref{eq:prob} if the optimal solution $V_k$ of \eqref{eq:sub_prob} satisfies $\|V_k\|\leq\epsilon$.
\end{definition}

\vskip 2mm

In the following results, we give an upper bound for the number of the outer loops to reach an $\epsilon$-stationary point of \eqref{eq:prob};
we also provide an upper bound for the total numbers of the inner ``while" loops (steps 7-23).
This is necessary for estimating the computational cost of Algorithm \ref{alg:ada_manpqn}.
In fact, at the $k$-th outer loop, if $\rho_k<\eta_1$, then $\sigma_k$ will increase by $\gamma_2$ times,
and $V_k$ will be computed by the adaptive regularized semismooth Newton (ASSN) method again in the next inner loop.
In our analysis, we use $r(k)$ to denote the number of the times of calling the ASSN method at the $k$-th outer loop.

Recall that $\overline{\sigma}$ and $\overline{\alpha}$ are defined in \eqref{eq:over_sigma} and \eqref{eq:over_alpha} respectively.
Scalars $\gamma$, $\gamma_1$ and $\eta_1$ are parameters of Algorithm \ref{alg:ada_manpqn}.
For $\epsilon>0$, denote
\begin{eqnarray}\label{theta}
\Theta&:=&\frac{2(F(X_0)-F^*)}{\eta_1\kappa_1\gamma\overline{\alpha}(2-\gamma\overline{\alpha})\epsilon^2},
\end{eqnarray}
where $F^*$ is the optimal value of \eqref{eq:prob}.

\vskip 2mm

\begin{theorem}
\label{thm:43}
Suppose Assumptions \ref{ass:1} holds.
An $\epsilon$-stationary point of Algorithm \ref{alg:ada_manpqn} will be found in at most $(m+1)\lceil\Theta\rceil$ outer loops,
where $\Theta$ is defined by \eqref{theta} and $\lceil\cdot\rceil$ denotes rounding up to the next integer.
Moreover, we have
\begin{equation}
\nonumber
\sum_{i=0}^{(m+1)\lceil\Theta\rceil-1}r(i)
\leq
(m+1)\left\lceil\Theta\right\rceil\log_{\gamma_2}\left(\frac{\gamma_2}{\gamma_1}\right)
+\log_{\gamma_2}\left(\frac{\max\{\sigma_0,\gamma_2\overline{\sigma}\}}{\sigma_0}\right)
+1.
\end{equation}
\end{theorem}

\begin{proof}
Similar to the proof of Theorem 3.2 in \cite{dyh2002}, by Lemma \ref{lm:41}, we can deduce that
\begin{equation}
\label{eq:3_nls_decrease}
F(X_{l((j+1)(m+1))})-F(X_{l(j(m+1))})
\leq
\max_{1\leq i\leq m}
\left\{-\frac{1}{2}\gamma\overline{\alpha}(2-\gamma\overline{\alpha})\eta_1\kappa_1\|V_{j(m+1)+i}\|^2\right\}.
\end{equation}
Assume that Algorithm \ref{alg:ada_manpqn} does not terminate after $K(m+1)$ outer loops,
that is $\|V_i\|\geq\epsilon$ for $i=0,1,\dots,K(m+1)-1$.
By \eqref{eq:3_nls_decrease}, we have
\[
F(X_0)-F^*
\geq
F(X_{l(0)})-F(X_{l(K(m+1))})
\geq
\frac{1}{2}\gamma\overline{\alpha}(2-\gamma\overline{\alpha})\eta_1\kappa_1\epsilon^2K.
\]
Then we have
\begin{eqnarray}\label{eq:out_comp}
K\leq \frac{2(F(X_0)-F^*)}{\eta_1\kappa_1\gamma\overline{\alpha}(2-\gamma\overline{\alpha})\epsilon^2}
=\Theta,
\end{eqnarray}
which implies the first assertion.

For each $i\geq0$, by the procedure of Algorithm \ref{alg:ada_manpqn} (see steps 15--22), taking into account the definition of $r(i)$,
we have $\sigma_{i+1}/\sigma_i\geq\gamma_1\gamma_2^{r(i)-1}$.
Thus, by \eqref{si}, we have
\begin{eqnarray}
\nonumber
\max\{\sigma_0,\gamma_2\overline{\sigma}\}
&\geq&
\sigma_{K(m+1)}
=
\sigma_0\prod_{i=0}^{K(m+1)-1}\frac{\sigma_{i+1}}{\sigma_i}
\geq
\sigma_0\prod_{i=0}^{K(m+1)-1}(\gamma_1\gamma_2^{r(i)-1})
\\
\nonumber
&=&
\sigma_0(\frac{\gamma_1}{\gamma_2})^{K(m+1)}\gamma_2^{\sum_{i=0}^{K(m+1)-1}r(i)}.
\end{eqnarray}
The above inequality can be rewritten as
\begin{eqnarray}
\nonumber
\sum_{i=0}^{K(m+1)-1}r(i)
&\leq&
\log_{\gamma_2}\left(\frac{\max\{\sigma_0,\gamma_2\overline{\sigma}\}}{\sigma_0}
(\frac{\gamma_2}{\gamma_1})^{K(m+1)}\right)
\\
\nonumber
&=&
K(m+1)\log_{\gamma_2}\left(\frac{\gamma_2}{\gamma_1}\right)
+\log_{\gamma_2}\left(\frac{\max\{\sigma_0,\gamma_2\overline{\sigma}\}}{\sigma_0}\right)
+1.
\end{eqnarray}
Combining it with \eqref{eq:out_comp} yields
\[
\nonumber
\sum_{i=0}^{(m+1)\lceil\Theta\rceil-1}r(i)
\leq
(m+1)\left\lceil\Theta\right\rceil\log_{\gamma_2}\left(\frac{\gamma_2}{\gamma_1}\right)
+\log_{\gamma_2}\left(\frac{\max\{\sigma_0,\gamma_2\overline{\sigma}\}}{\sigma_0}\right)
+1.
\]
Thus the assertion holds.
\end{proof}

\section{Adaptive Regularized Proximal Newton-Type Methods with Superlinear Convergence Rate}
\label{sec:new_algo}

In this section, we propose an adaptive regularized proximal Newton-type method, named ARPN, to solve \eqref{eq:prob}.
At the $k$-th iterate $X_k$, the subproblem of ARPN is defined as
\begin{eqnarray}
\mathop{\min}_{V\in {\rm T}_{X_k}\mathcal{M}}
\varphi_{k}(V)
&:=&
\label{sub_prob1}
\langle g_k,V\rangle
+\frac{1}{2}\langle \mathbb{H}_k[V],V\rangle
+h(\mathbf{R}_{X_k}(V)),
\end{eqnarray}
where $g_k:=\mathrm{grad}f(X_k)$ and $\mathbb{H}_k$ is a linear operator on ${\rm T}_{X_k}\mathcal{M}$.
When $\mathbb{H}_k=\frac{\widetilde{L}}{2}I$,
where $\widetilde{L}>0$, the subproblem \eqref{sub_prob1} has been considered by Huang and Wei in \cite{huang2021riemannian},
and the local linear convergence rate of their algorithm is established.
In the ARPN method, $\mathbb{H}_k$ is set to be $H_k+\sigma_kI$,
where $\sigma_k>0$ is the regularization parameter and $H_k$ is the Hessian of $f$ at $X_k$ or generated by the quasi-Newton method.
The global convergence and the local superlinear convergence rate of ARPN are established for both two cases of $H_k$.
The ARPN method is stated in Algorithm \ref{algo2}.

\begin{algorithm}
\caption{An Adaptive regularized proximal Newton-type algorithm for Riemannian composite optimization (ARPN) }
\label{algo2}
\begin{algorithmic}[1]
	\Require Initial point $X_0\in\mathcal{M}$, line-search parameters $\sigma\in(0,1/4), \gamma \in (0,1)$, initial regularization parameter $\sigma_0>0$, $0<\eta_1 <\eta_2 <1$ and $0<\gamma_1 <1<\gamma_2$,
	Lipschitz constant $L_h$ of function $h$ (refer to (A.2) in Assumption \ref{ass:1}),
	parameter $M_2$ for retraction mapping $\mathbf{R}$ (see \eqref{eq_th_ass4} in Assumption \ref{ass:th1}).
	\For{k=0,1,\dots}
	\State Update $H_k$;
	\State Set $\kappa_3=3L_hM_2$;
	\While{$\lambda_{\min}(H_k)+\frac{1}{2}\sigma_k\leq \kappa_3$}
		\State $\sigma_k\gets\gamma_2\sigma_k$;
	\EndWhile
	\State Calculate $\mathbb{H}_k=H_k+\sigma_kI$;
	\While{$1$}
	\State Solve the subproblem \eqref{sub_prob1} to obtain the search direction $V_k$;
	\State Set the initial stepsize $\alpha_k=1$;
	\While{$F(\mathbf{R}_{X_k}(\alpha_kV_k))\leq F(X_k)-\frac{1}{2}\sigma\alpha_k
	\|V_k\|_{\mathbb{H}_k}^2$ is not satisfied}
		\State $\alpha_k\gets\gamma\alpha_k$;
	\EndWhile
	\State Set $Z_k=\mathbf{R}_{X_k}(\alpha_kV_k)$;
	\State Calculate the ratio
	\begin{equation}
\label{eq:th_new_rhok}
\rho_k:=\frac{F({\bf R}_{X_k}(\alpha_k V_k))-F(X_{k})}{\varphi_k(\alpha_kV_k)-\varphi_k(0)},
\end{equation}
	\State Execute steps 15--22 of Algorithm \ref{alg:ada_manpqn};
		\EndWhile
		\State  Set $X_{k+1}=Z_k$ and $\sigma_{k+1}=\sigma_k$;
	\EndFor
\end{algorithmic}
\end{algorithm}

\vskip 2mm

\subsection{Global Convergence for the Case of $H_k=\mathrm{Hess}f(X_k)$}

In this subsection, we consider the case of $H_k=\mathrm{Hess}f(X_k)$ and prove the global convergence of ARPN.
We need the following assumption.

\vskip 2mm

\begin{assumption}
\label{ass:th1}
Assume the following:
\begin{enumerate}
\item[\rm(B.1)]
Conditions $\mathrm{(A.1)}$, $\mathrm{(A.2)}$ and $\mathrm{(A.4)}$ of Assumption \ref{ass:1} hold.
\item[\rm(B.2)]
$\mathbf{R}_{X}(\xi)$ is a second-order retraction
(for the definition of second-order retraction, we refer to \cite[Proposition 5.5.5]{absil2008}).
\item[\rm(B.3)]
$\mathbf{R}_X(\xi)$ can be extended to a mapping from $\mathcal{M}\times\mathbb{R}^{n\times r}$ to $\mathcal{M}$,
which is a continuous differentiable mapping of $(X,\xi)$ and is denoted by $\mathbf{R}_X(\xi)$ also.
Moreover, $\mathbb{D}\mathbf{R}_X(0_X)=\mathbf{id}$ for all~$X\in\mathcal{M}$, where $\mathbb{D}\mathbf{R}_X(\xi)$ is the Euclidean differential of $\mathbf{R}_X$ at $0_X$, $\mathbf{id}$ is the identity operator on $\mathbb{R}^{n\times r}$,
and there exists $M_2>0$ such that
\begin{equation}
\label{eq_th_ass4}
\|\mathbb{D}\mathbf{R}_X(\xi)-\mathbb{D}\mathbf{R}_X(\eta)\|
\leq
2M_2\|\xi-\eta\|,
\quad\forall~X\in\mathcal{M}, \quad\forall~\xi,\eta\in\mathrm{T}_X\mathcal{M}.
\end{equation}
\item[\rm(B.4)] $f$ and $\mathbf{R}_X$ are twice differentiable, and
$\mathrm{Hess}f(\mathbf{R}_X(V))$ is continuous with respect to $(X,V)$.
\end{enumerate}
\end{assumption}

\vskip 2mm

\begin{remark}
\label{rm_retr}
The retraction based on the polar decomposition is given by (see \cite[Example 4.1.3 ]{absil2008})
\[
\mathbf{R}_X(\xi):=(X+\xi)(I_r+\xi^T\xi)^{-1/2},\quad\forall~X\in\mathcal{M}, \quad\forall~\xi\in\mathrm{T}_X\mathcal{M}.
\]
By \cite[Proposition 7]{AbsMal2012}, $\mathbf{R}_X(\xi)$ is just the projection of $X+\xi$ onto $\mathcal{M}$,
which is also the retraction based on the singular value decomposition (SVD).
Moreover, by \cite[Example 23]{AbsMal2012}, $\mathbf{R}_X$ is a second-order retraction.
It is easy to verify that $\mathbf{R}_X$ satisfies conditions $\rm(B.2)$, $\rm(B.3)$ and $\rm(B.4)$ of Assumption \ref{ass:th1}.

If \eqref{eq_th_ass4} holds, for all $\xi,~\eta\in \mathrm{T}_X\mathcal{M}$, we can deduce that
\begin{eqnarray}
\label{eq_th_ass3}
\|\mathbf{R}_X(\xi)-\mathbf{R}_X(\eta)-\mathbb{D}\mathbf{R}_X(\eta)[\xi-\eta]\|
&\leq&
M_2\|\xi-\eta\|^2,
\\
\label{eq_th_ass1}
\|\mathbf{R}_X(\xi)-X-\xi\|
&\leq&
M_2\|\xi\|^2.
\end{eqnarray}
\end{remark}

In the following results, the notation $T^*$ refers to the adjoint operator of $T$ \cite[p.191]{absil2008},
where $T:E_1\to E_2$ is a linear operator and $E_1,E_2$ are two Euclidean spaces.
The linear operator $T^*$ satisfies $\langle T[x],y\rangle=\langle x,T^*[y]\rangle$ for all~$x\in E_1$ and $y\in E_2$.

\vskip 2mm

For any $k\geq0$, $h\circ\mathbf{R}_{X_k}$ may be nonconvex,
and therefore $\arg\min_{V\in\mathrm{T}_{X_k}\mathcal{M}}\varphi_k(V)$ may not be a singleton.
At the $k$-th iteration, we select a $V_k\in\arg\min_{V\in\mathrm{T}_{X_k}\mathcal{M}}\varphi_k(V)$.
The following lemma shows that $V_k$ provides sufficient decrease in $\varphi_{k}$.

\vskip 2mm

\begin{lemma}
\label{lm_th_phi}
Suppose Assumption \ref{ass:th1} holds.
Then we have
\begin{equation}
\label{eq_th_1}
\varphi_{k}(\alpha V_k)-\varphi_{k}(0)
\leq
\frac{1}{2}(L_hM_2+\frac{1}{2}\sigma_k)(\alpha-2)\alpha\|V_k\|^2,
~\forall~\alpha\in[0,1],
\end{equation}
where $M_2$ is given as in \eqref{eq_th_ass4} and $L_h$ is the Lipschitz constant of $h$.
\end{lemma}

\begin{proof}
By $V_k\in\arg\min_{V\in\mathrm{T}_{X_k}\mathcal{M}}\varphi_k(V)$, we have
\begin{equation}
\label{eq_th_sub_opt}
0\in g_k
+\mathbb{H}_k[V_k]
+\mathrm{Proj}_{\mathrm{T}_{X_k}\mathcal{M}}\mathbb{D}\mathbf{R}_{X_k}(V_k)^*[\partial h(\mathbf{R}_{X_k}(V_k))].
\end{equation}
Then for all $\alpha\in[0,1]$, by definition of $\varphi_k$ and \eqref{eq_th_sub_opt}, we have
\begin{eqnarray}
\nonumber
&&
\varphi_{k}(\alpha V_k)-\varphi_{k}(0)
=
\langle g_k, \alpha V_k\rangle
+\frac{1}{2}\langle\mathbb{H}_k[ \alpha V_k], \alpha V_k\rangle
+h(\mathbf{R}_{X_k}( \alpha V_k))
-h(X_k)
\\
\nonumber
&=&
\frac{1}{2}(\alpha-2)\alpha\langle\mathbb{H}_k[V_k], V_k\rangle
-\langle\mathrm{Proj}_{\mathrm{T}_{X_k}\mathcal{M}}\mathbb{D}\mathbf{R}_{X_k}(V_k)^*[\eta],\alpha V_k\rangle
+h(\mathbf{R}_{X_k}(\alpha V_k))-h(X_k)
\\
\nonumber
&\leq&
\frac{1}{2}(\alpha-2)\alpha\|V_k\|_{\mathbb{H}_k}^2
+\langle \eta, \mathbf{R}_{X_k}(\alpha V_k)-X_k-\mathbb{D}\mathbf{R}_{X_k}(V_k)[\alpha V_k]\rangle
\\
\nonumber
&\leq&
\frac{1}{2}(\alpha-2)\alpha\|V_k\|_{\mathbb{H}_k}^2
+L_h\cdot(M_2\alpha^2\|V_k\|^2
+\|\mathbb{D}\mathbf{R}_{X_k}(\alpha V_k)-\mathbb{D}\mathbf{R}_{X_k}(V_k)\|\cdot\alpha\|V_k\|)
\\
\label{eq_th_phi1}
&\leq&
\frac{1}{2}(\alpha-2)\alpha
\|V_k\|_{\mathbb{H}_k}^2
+L_hM_2(2-\alpha)\alpha\|V_k\|^2,
\end{eqnarray}
where $\eta\in\partial h(\mathbf{R}_{X_k}(V_k))$,
the first inequality follows from the convexity of $h$,
the second inequality follows from \eqref{eq_th_ass3}
and the last inequality \eqref{eq_th_phi1} follows from \eqref{eq_th_ass4}.

By steps 3--6 of Algorithm \ref{algo2}, we can obtain that
$
\lambda_{\min}(\mathbb{H}_k)
=\lambda_{\min}(H_k)+\sigma_k>\kappa_3+\sigma_k/2
=3L_hM_2+\sigma_k/2
$.
Substituting it into \eqref{eq_th_phi1} yields
\begin{eqnarray}
\nonumber
\varphi_{k}(\alpha V_k)-\varphi_{k}(0)
&\leq&
\frac{1}{2}(\alpha-2)\alpha
(3L_hM_2+\frac{1}{2}\sigma_k)
\|V_k\|^2
-L_hM_2(\alpha-2)\alpha\|V_k\|^2
\\
\nonumber
&=&
\frac{1}{2}(L_hM_2+\frac{1}{2}\sigma_k)(\alpha-2)\alpha\|V_k\|^2.
\end{eqnarray}
Then \eqref{eq_th_1} holds.
\end{proof}

\vskip 2mm

In the following result, 
we show that sufficient reduction of $F$ can be achieved along $V_k$ when the regularization parameter $\sigma_k$ is sufficiently large.

\vskip 2mm

\begin{lemma}
\label{lm_th_sigma_lb}
Suppose Assumption \ref{ass:th1} holds.
If
\begin{equation}
\label{eq_th_sigma_lb}
\sigma_k\geq\widetilde{\sigma}:=
\frac{4c_1-2\kappa_3-2(1-\eta_1)L_hM_2}{2-\eta_1},
\end{equation}
then $\rho_k\geq\eta_1$,
where $c_1:=\varrho M_2+\frac{1}{2}LM_1^2$ is given as in \eqref{c2},
$\kappa_3=3L_hM_2$ and $\eta_1\in(0,1)$ are parameters of Algorithm \ref{algo2}.
\end{lemma}

\begin{proof}
Since $\nabla f$ is Lipschitz continuous with Lipschitz constant $L$, by \eqref{eq34_1}, for any $\alpha\in(0,1]$, we have
\begin{eqnarray}
\nonumber
F(\mathbf{R}_{X_k}(\alpha V_k))
&\leq&
f(X_k)+\langle\nabla f(X_k),\alpha V_k\rangle
+c_1\|\alpha V_k\|^2
+h(\mathbf{R}_{X_k}(\alpha V_k))
\\
\label{eq_th_lm_sigma1}
&=&
F(X_k)
+\varphi_k(\alpha V_k)-\varphi_k(0)
+c_1\|\alpha V_k\|^2
-\frac{1}{2}\langle\alpha V_k,\mathbb{H}_k[\alpha V_k]\rangle
\\
\label{eq_th_lm_sigma2}
&<&
F(X_k)
+\varphi_k(\alpha V_k)-\varphi_k(0)
+(c_1-\frac{1}{2}(\kappa_3+\frac{1}{2}\sigma_k))\|\alpha V_k\|^2,
\end{eqnarray}
where \eqref{eq_th_lm_sigma2} follows from steps 3--6 of Algorithm \ref{algo2}.

Assume \eqref{eq_th_sigma_lb} holds. Substituting \eqref{lm_th_phi} and \eqref{eq_th_lm_sigma2} into the definition of $\rho_k$ yields
\begin{eqnarray}
\nonumber
1-\rho_k
&=&
\frac{F(\mathbf{R}_{X_k}(\alpha_k V_k))-F(X_k)-\varphi_k(\alpha_k V_k)+\varphi_k(0)}{-\varphi_k(\alpha_k V_k)+\varphi_k(0)}
\leq
\frac{(c_1-\frac{1}{2}\kappa_3-\frac{1}{4}\sigma_k)\|\alpha_k V_k\|^2}{\frac{1}{2}(L_hM_2+\frac{1}{2}\sigma_k)(2-\alpha_k)\alpha_k\| V_k\|^2}
\\
\nonumber
&=&
\frac{(4c_1-\sigma_k-2\kappa_3)\alpha_k}{(2L_hM_2+\sigma_k)(2-\alpha_k)}
\leq
\frac{4c_1-\sigma_k-2\kappa_3}{2L_hM_2+\sigma_k}
\leq
1-\eta_1,
\end{eqnarray}
that is $\rho_k\geq\eta_1$, which completes the proof.
\end{proof}

\vskip 2mm

Since $f$ is twice continuously differentiable and $\mathcal{M}$ is compact, we can define
\begin{equation}
\label{eq_th_tilde_lambda}
\widetilde{\lambda}_k:=
\lambda_{\min}(\mathrm{Hess}f(X_k)),
~\widetilde{\lambda}:=
\inf_{k\geq 0}\widetilde{\lambda}_k.
\end{equation}
The following result shows that the regularization parameter $\sigma_k$ has an upper bound.

\vskip 2mm

\begin{lemma}
\label{lm_th_sigma_ub}
Suppose Assumption \ref{ass:th1} holds.
Then,
\begin{equation}
\label{eq_th_sigma_ub}
\sigma_k
\leq
\max\{\sigma_0,\gamma_2\widetilde{\sigma},2\gamma_2(\kappa_3-\widetilde{\lambda})\},
~\forall~k\geq0,
\end{equation}
where $\widetilde{\sigma}$ is defined in \eqref{eq_th_sigma_lb}, $\widetilde{\lambda}$ is given in \eqref{eq_th_tilde_lambda},
$\sigma_0,~\gamma_2$ and $\kappa_3$ are parameters of Algorithm \ref{algo2}.
\end{lemma}

\begin{proof}
The proof is by induction.
We can see that \eqref{eq_th_sigma_ub} holds trivially for $k=0$.
Assuming now that it is true for some $k=j$, we show that it holds for $k=j+1$.
We consider the cases of $\sigma_j<\max\{\widetilde{\sigma},2(\kappa_3-\widetilde{\lambda})\}$ and $\sigma_j\geq\max\{\widetilde{\sigma},2(\kappa_3-\widetilde{\lambda})\}$ separately:
(1) For the former case,
by Lemma \ref{lm_th_sigma_lb}, steps 4--6 and step 16 of Algorithm \ref{algo2}, we have $\sigma_{j+1}<\gamma_2\max\{\widetilde{\sigma},2(\kappa_3-\widetilde{\lambda})\}$.
(2) For the case $\sigma_j\geq
\max\{\widetilde{\sigma},2(\kappa_3-\widetilde{\lambda})\}$,
from Lemma \ref{lm_th_sigma_lb}, it holds that $\rho_j\geq\eta_1$.
By steps 4-6 and step 16 of Algorithm \ref{algo2}, we have
$\sigma_{j+1}
\leq
\sigma_j
\leq
\max\{\sigma_0,
\gamma_2\widetilde{\sigma},
2\gamma_2(\kappa_3-\widetilde{\lambda})\}$.
Thus, \eqref{eq_th_sigma_ub} holds for $k=j+1$ as well.
The proof is complete.
\end{proof}

\vskip 2mm

The following theorem shows that $\alpha_k$ has a uniformly lower bound and establish the global convergence of Algorithm \ref{algo2}.

\vskip 2mm

\begin{theorem}
\label{lm_th_alpha_lb}
Suppose Assumption \ref{ass:th1} holds. Then the following statements hold:
\begin{enumerate}
\item[\rm(i)] $\alpha_k\geq\gamma\widetilde{\alpha}$, where
\begin{equation}
\label{eq:tilde_alpha}
\widetilde{\alpha}:=\min\{1,\frac{(2-3\sigma)\kappa_3}{6c_1}\},
\end{equation}
where $c_1$ is defined in \eqref{c2},
$\gamma\in(0,1),~\kappa_3=3L_hM_2,~\sigma\in(0,1/4)$ are parameters of Algorithm \ref{algo2}.
\item[\rm(ii)] The backtracking line search procedure will terminate in finite steps,
that is,
\begin{equation}
\label{eq_th_ls_dec}
F(X_{k+1})-F(X_k)
\leq
-\frac{1}{2}\sigma\alpha_k\|V_k\|_{\mathbb{H}_k}^2
\leq
-\frac{1}{2}\sigma\gamma\widetilde{\alpha}\|V_k\|_{\mathbb{H}_k}^2,~\forall~k\geq0.
\end{equation}
\item[\rm(iii)] We have $\lim_{k\to\infty}\|V_k\|=0$ and all accumulation points of $\{X_k\}$ are stationary points of problem \eqref{eq:prob}.
\end{enumerate}
\end{theorem}

\begin{proof}
$\rm(i)$. By $\sigma\in(0,1/4)$ and \eqref{eq:tilde_alpha}, we can see that $\widetilde{\alpha}\in(0,1]$.
Combining \eqref{eq_th_phi1} and \eqref{eq_th_lm_sigma1} yields that for all $0<\alpha\leq1$, it holds
\begin{eqnarray}
\nonumber
F(\mathbf{R}_{X_k}(\alpha V_k))
&\leq&
F(X_k)
+c_1\alpha^2\| V_k\|^2
-\alpha(2L_hM_2\|V_k\|^2-\|V_k\|_{\mathbb{H}_k}^2)
\\
\label{eq:thm41_alpha1}
&\leq&
F(X_k)
+c_1\alpha^2\| V_k\|^2
-\frac{1}{3}\alpha\|V_k\|_{\mathbb{H}_k}^2,
\end{eqnarray}
where $c_1:=\varrho M_2+\frac{1}{2}LM_1^2$ is defined in \eqref{c2} and the second inequality follows from $\lambda_{\min}(\mathbb{H}_k)
>3L_hM_2+\sigma_k/2$ (see steps 3--6 of Algorithm \ref{algo2}).
If $0<\alpha\leq\widetilde{\alpha}$, by \eqref{eq:thm41_alpha1}, it holds that
\[
F(\mathbf{R}_{X_k}(\alpha V_k))
-F(X_k)
\leq
\frac{2-3\sigma}{6}\kappa_3\alpha \| V_k\|^2
-\frac{1}{3}\alpha\|V_k\|_{\mathbb{H}_k}^2
\leq
-\frac{1}{2}\sigma\alpha\|V_k\|_{\mathbb{H}_k}^2.
\]
Using the above inequality, by steps 11--13 of Algorithm \ref{algo2}, we have $\alpha_k\geq\gamma\widetilde{\alpha}$ for all $k\geq0$.

$\rm(ii)$. The assertion follows immediately from $\rm(i)$.

$\rm(iii)$.
Since $X_k$ is updated only when $\rho_k\geq\eta_1$,
by \eqref{eq:th_new_rhok}, \eqref{eq:tilde_alpha} and steps 3--6 of Algorithm \ref{algo2}, we have
\begin{eqnarray}
\nonumber
&&F(X_0)-F(X_{k+1})
=
\sum_{j=0}^k
(F(X_j)-F(\mathbf{R}_{X_j}(\alpha_jV_j))
\geq
\sum_{j=0}^k
\eta_1(\varphi_j(0)-\varphi_j(\alpha_jV_j))
\\
\label{eq:3_new_f_decrease}
&\geq&
\sum_{j=0}^k
\eta_1\frac{\alpha_i(2-\alpha_j)}{2}\langle\mathbb{H}_j[V_j],V_j\rangle
\geq
\frac{3}{2}\gamma\widetilde{\alpha}(2-\gamma\widetilde{\alpha})\eta_1L_hM_2
\sum_{j=0}^k\|V_j\|^2,
\end{eqnarray}
where the last inequality uses the fact $\gamma\widetilde{\alpha}\leq\alpha_k\leq1$.
By \eqref{eq_th_ls_dec}, $\{F(X_k)\}_k$ is a non-increasing sequence.
Since $\mathcal{M}$ is compact, $\{F(X_k)\}_k$ is bounded from below.
Thus, $\{F(X_k)\}_k$ is convergent, which together with \eqref{eq:3_new_f_decrease} implies $\lim_{k\to\infty}\|V_k\|=0$.
Let $X^*$ be an accumulation point of sequence $\{X_k\}$.
By \eqref{eq_th_sub_opt}, (B.3) of Assumption \ref{ass:th1} and $\lim_{k\to\infty}\|V_k\|=0$, we know that $X^*$ satisfies \eqref{eq:2_opt}. The proof is complete.
\end{proof}

\subsection{Local Superlinear Convergence for the Case of $H_k=\mathrm{Hess}f(X_k)$}

In this subsection, we establish superlinear convergence of Algorithm \ref{algo2} under Assumption \ref{ass:2}.
To establish our main convergence results, we do some preparatory work.

\vskip 2mm

\begin{lemma}
(\cite[Theorem 4.2]{manpqn2023})
\label{lm_th_accum}
Suppose Assumptions \ref{ass:2} and \ref{ass:th1} hold, and $X^*$ is the accumulation point
satisfying \eqref{eq:4_hess}.
Then, $X_k$ converges to $X^*$.
\end{lemma}

By \eqref{eq:4_hess} and Assumption \ref{ass:th1},
there exists a neighbourhood $\mathcal{U}_{X^*}$ of $X^*$ such that $\lambda_{\min}({\rm Hess} (f\circ{\bf R}_{X})(0_X))\geq4\delta/5$
for all $X\in\mathcal{U}_{X^*}$.
By Lemma \ref{lm_th_accum}, there exists $K_1\geq0$ such that $X_k\in \mathcal{U}_{X^*}$ for all $k\geq K_1$.
Since $\mathbf{R}_{X_k}(\xi)$ is a second-order retraction, we have ${\rm Hess} (f\circ{\bf R}_{X_k})(0_{X_k})=\mathrm{Hess}f(X_k)$. Thus,
\begin{equation}\label{eq:4_hess_neigh}
\lambda_{\min}({\rm Hess} (f\circ{\bf R}_{X_k})(0_{X_k}))=\lambda_{\min}(\mathrm{Hess}f(X_k))
\geq
\frac{4}{5}\delta,\quad\forall~k\geq K_1.
\end{equation}

The following result tells us that $\alpha_k=1$ can be accepted in steps 11--13 of Algorithm \ref{algo2} for all sufficiently large $k$.

\begin{lemma}
\label{lm_th_alpha}
Suppose Assumptions \ref{ass:2} and \ref{ass:th1} hold.
Then there exists $\widetilde{K}\geq 0$ such that
\begin{equation}
\label{eq_th_ls}
F(\mathbf{R}_{X_k}(V_k))
\leq
F(X_k)-\frac{1}{2}\sigma\|V_k\|^2,
~\forall~k\geq \widetilde{K},
\end{equation}
which implies $X_{k+1}=\mathbf{R}_{X_k}(V_k)$.
Moreover, $\rho_k\geq\eta_2$ for all $k\geq \widetilde{K}$, and therefore $\sigma_k\to0$.
\end{lemma}

\begin{proof}
For $k\geq 0$, let $R_k(V):=f(\mathbf{R}_{X_k}(V))-[f(X_k)+\langle g_k,V\rangle+\frac{1}{2}\langle V,\mathrm{Hess}f(X_k)[V]\rangle]$,
where $V\in\mathrm{T}_{X_k}\mathcal{M}$.
Then, we have
\begin{eqnarray}
\label{eq_th_ass_hess_lip}
R_k(V)\leq\frac{1}{2}\underbrace{\max_{0\leq t\leq1}\|\mathrm{Hess}(f\circ\mathbf{R}_{X_k})(tV)-\mathrm{Hess}f(X_k)\|}_{r_k(V)}\cdot\|V\|^2.
\end{eqnarray}
Pick any $\epsilon>0$.
Since $\mathrm{Hess}f(\mathbf{R}_{X}(V))$ is continuous with respect to $(X,V)$ (see Assumption \ref{ass:th1}),
taking into account $X_k$ converges to $X^*$ and $V_k\to0$,
we know that there exists $K_2\geq0$ such that $r_k(V_k)<\epsilon$ for all $k\geq K_2$.
Then, we can deduce that
\begin{eqnarray}
\nonumber
F(\mathbf{R}_{X_{k}}(V_k))-F(X_k)
&=&\langle g_k, V_k\rangle
+\frac{1}{2}\langle V_k, \mathrm{Hess}f(X_k)[V_k]\rangle
+R_k(V_k)+h(\mathbf{R}_{X_{k}}(V_k))-h(X_k)
\\
\label{wangwang}
&\leq&
\varphi_{k}(V_k)-\varphi_{k}(0)
-\frac{1}{2}(\sigma_k-r_k(V_k))\|V_k\|^2
\\
\label{eq_th_alpha1}
&\leq&
-\frac{1}{2}\|V_k\|_{\mathbb{H}_k}^2
+\frac{1}{2}(2L_hM_2+\epsilon-\sigma_k)\|V_k\|^2,
\end{eqnarray}
where \eqref{eq_th_alpha1} follows from \eqref{eq_th_phi1}
and the fact $r_k(V_k)<\epsilon$ for all $k\geq K_2$.
By \eqref{eq:4_hess_neigh}, we have
$\|V\|^2_{\mathbb{H}_k}
\geq
\frac{4}{5}\delta\|V\|^2
\geq
4L_hM_2\|V\|^2$
for all $V\in\mathrm{T}_{X_k}\mathcal{M}$.
Combining it with \eqref{eq_th_alpha1} gives
\begin{equation}
\nonumber
F(\mathbf{R}_{X_{k}}(V_k))-F(X_k)
\leq
-\frac{1}{6}\|V_k\|_{\mathbb{H}_k}^2
+\frac{1}{2}(\epsilon-\sigma_k)\|V_k\|^2.
\end{equation}
Without loss of generality, assume that $\epsilon<\min\{1/3,1-\eta_2\}\cdot L_hM_2$.
Using the above inequality and taking into account $\sigma\in(0,1/4)$, we can obtain that $F(\mathbf{R}_{X_{k}}(V_k))-F(X_k)\leq-\frac{1}{2}\sigma\|V_k\|_{\mathbb{H}_k}^2$
for all $k\geq K_2$.
Then, the line search condition \eqref{eq_th_ls_dec} can always be satisfied with $\alpha_k=1$ when $k\geq K_2$.

By \eqref{eq:4_hess_neigh}, we know that
\[
\lambda_{\min}(\mathbb{H}_k)-\sigma_k/2
=\lambda_{\min}(\mathrm{Hess}f(X_k))+\sigma_k/2
>\kappa_3
=3L_hM_2,\quad\forall k\geq K_1.
\]
Thus, $\sigma_k$ will not increase during the steps 4--6 of Algorithm \ref{algo2}.

Let $\widetilde{K}:=\max\{K_1,K_2\}$.
By \eqref{wangwang} and \eqref{eq_th_1}, taking into account $\alpha_k=1$ for all $k\geq K_2$, we have
\[
1-\rho_k
=
\frac{F(\mathbf{R}_{X_k}(V_k))-F(X_k)-\varphi_k(V_k)+\varphi_k(0)}{\varphi_k(0)-\varphi_k(V_k)}
\leq
\frac{-\sigma_k+\epsilon}{L_hM_2+\sigma_k}
\leq1-\eta_2,\quad\forall k\geq \widetilde{K},
\]
which implies $\rho_k\geq\eta_2$.
From the procedures of Algorithm \ref{algo2}, we can obtain that $\sigma_k\to0$.
The proof is complete.
\end{proof}

\vskip 2mm

In the following, we prove an important result which will be used in several places.
Given $X\in\mathcal{M}$, define a function $\varphi$ on $\mathrm{T}_{X}\mathcal{M}$ by $\varphi(V):=\langle g,V\rangle
+\frac{1}{2}\langle H[V],V\rangle+h(\mathbf{R}_{X}(V))$,
where $g\in\mathrm{T}_{X}\mathcal{M}$ and $H$ is a linear operator on $\mathrm{T}_{X}\mathcal{M}$.

\vskip 2mm

\begin{lemma}
\label{lm_th_unique}
If $\lambda_{\min}(H)\geq 4L_hM_2$, then the optimal solution of $\min_{V\in\mathrm{T}_X\mathcal{M}}\varphi(V)$ is unique.
\end{lemma}

\begin{proof}
Pick any $\widetilde{V}\in\arg\min_{V\in\mathrm{T}_{X}\mathcal{M}}\varphi(V)$.
Then there exists $\xi\in\partial h(\mathbf{R}_{X}(\widetilde{V}))$ such that
\begin{eqnarray}\label{diyi}
g+H[\widetilde{V}]+\mathrm{Proj}_{\mathrm{T}_{X}\mathcal{M}}\mathbb{D}\mathbf{R}_{X}(\widetilde{V})^*[\xi]=0.
\end{eqnarray}
For any $W\in\mathrm{T}_{X}\mathcal{M}$ and $W\neq\widetilde{V}$, it holds that
\begin{eqnarray}
\nonumber
&&
\varphi(W)-\varphi(\widetilde{V})
\\
\nonumber
&=&
\langle g,W-\widetilde{V}\rangle
+\frac{1}{2}\langle H[W],W\rangle
-\frac{1}{2}\langle H[\widetilde{V}],\widetilde{V}\rangle
+h(\mathbf{R}_{X}(W))-h(\mathbf{R}_{X}(\widetilde{V}))
\\
\nonumber
&\geq&
\langle g+H[\widetilde{V}],W-\widetilde{V}\rangle
+\frac{1}{2}\langle H[W-\widetilde{V}],W-\widetilde{V}\rangle
+\langle\xi,\mathbf{R}_{X}(W)-\mathbf{R}_{X}(\widetilde{V})\rangle,
\\
\nonumber
&\geq&
\langle g+H[\widetilde{V}],W-\widetilde{V}\rangle
+2L_hM_2\|W-\widetilde{V}\|^2
+\langle\xi,\mathbb{D}\mathbf{R}_{X}(\widetilde{V})[W-\widetilde{V}]\rangle
-L_hM_2\|W-\widetilde{V}\|^2
\\
\nonumber
&=&L_hM_2\|W-\widetilde{V}\|^2>0,
\end{eqnarray}
where the first inequality follows from the convexity of $h$;
the second inequality follows from \eqref{eq_th_ass3} and $\lambda_{\min}(H)\geq 4L_hM_2$;
the second equality uses \eqref{diyi}.
Thus, the optimal solution of $\min_{V\in\mathrm{T}_{X}\mathcal{M}}\varphi(V)$ is unique.
\end{proof}

\vskip 2mm

Let $K_1$ be the integer such that \eqref{eq:4_hess_neigh} holds.
Suppose Assumption \ref{ass:2}
hold.
Then $\lambda_{\min}(\mathbb{H}_k)\geq4L_hM_2$ for all $k\geq K_1$.
From Lemma \ref{lm_th_unique}, we know that for any $k\geq K_1$, the following problem has a unique solution:
\begin{equation}
\label{eq:th_def_prox}
\mathrm{prox}_{h\circ\mathbf{R}_{X_k}}^{\mathbb{H}_k}(V)
:=
\mathop{\arg\min}\limits_{Y\in\mathrm{T}_{X_k}\mathcal{M}}\{\frac{1}{2}\|Y-V\|_{\mathbb{H}_k}^2+h(\mathbf{R}_{X_k}(Y))\},
~\mathrm{where}~V\in\mathrm{T}_{X_k}\mathcal{M}.
\end{equation}
The solution of the above problem is called the proximal mapping of the function $h\circ\mathbf{R}_{X_k}$ on the tangent space $\mathrm{T}_{X_k}\mathcal{M}$.
From the definition of $\varphi_{k}(V)$ (see \eqref{sub_prob1}), we can know that $\min\varphi_{k}$ has a unique solution $V_k$ for all $k\geq K_1$,
and
$
V_k=\mathrm{prox}_{h\circ\mathbf{R}_{X_k}}^{\mathbb{H}_k}(-\mathbb{H}_k^{-1}g_k).
$

\vskip 2mm

\begin{lemma}
\label{lm_th_lip}
Suppose Assumptions \ref{ass:2} and \ref{ass:th1} hold.
Then there exist $\varsigma>0$ and $\epsilon\in(0,1)$ such that for all $k\geq \widetilde{K}$,
where $\widetilde{K}$ is given as in Lemma \ref{lm_th_alpha},
if $U,~W\in\mathrm{T}_{X_k}\mathcal{M}$ satisfy $\|U-W\|\leq\epsilon$ and
\begin{equation}
\label{eq_th_prox_eps}
\max\{\|\mathrm{prox}_{h\circ\mathbf{R}_{X_k}}^{\mathbb{H}_k}(U)\|,
\|\mathrm{prox}_{h\circ\mathbf{R}_{X_k}}^{\mathbb{H}_k}(W)\|\}
\leq
\epsilon/2,
\end{equation}
then
\begin{equation}
\label{eq_th_lip}
\|\mathrm{prox}_{h\circ\mathbf{R}_{X_k}}^{\mathbb{H}_k}(U)-\mathrm{prox}_{h\circ\mathbf{R}_{X_k}}^{\mathbb{H}_k}(W)\|_{\mathbb{H}_k}\leq\varsigma\|U-W\|_{\mathbb{H}_k}.
\end{equation}
\end{lemma}

\begin{proof}
By \eqref{eq:th_def_prox} and \cite[(2.11)]{pn2014}, we know that for all $U\in\mathrm{T}_{X_k}\mathcal{M}$,
\[
\mathbb{H}_k[U-\mathrm{prox}_{h\circ\mathbf{R}_{X_k}}^{\mathbb{H}_k}(U)]
\in
\mathrm{Proj}_{\mathrm{T}_{X_k}\mathcal{M}}
\mathbb{D}\mathbf{R}_{X_k}(\mathrm{prox}_{h\circ\mathbf{R}_{X_k}}^{\mathbb{H}_k}(U))^*
[\partial h(\mathbf{R}_{X_k}(\mathrm{prox}_{h\circ\mathbf{R}_{X_k}}^{\mathbb{H}_k}(U)))].
\]
Then, there exists $\zeta_U\in\mathbb{R}^{n\times r}$ satisfying $\zeta_U\perp\mathrm{T}_{X_k}\mathcal{M}$ such that for all $U\in\mathrm{T}_{X_k}\mathcal{M}$,
\begin{equation}
\label{eq_th_zeta_partial_h}
\mathbb{H}_k[U-\mathrm{prox}_{h\circ\mathbf{R}_{X_k}}^{\mathbb{H}_k}(U)]
+\zeta_U
\in
\mathbb{D}\mathbf{R}_{X_k}(\mathrm{prox}_{h\circ\mathbf{R}_{X_k}}^{\mathbb{H}_k}(U))^*
[\partial h(\mathbf{R}_{X_k}(\mathrm{prox}_{h\circ\mathbf{R}_{X_k}}^{\mathbb{H}_k}(U)))]
.
\end{equation}

Select $\kappa_R\in(1,\sqrt{\frac{2\delta}{5L_hM_2}-1})$ arbitrarily.
Let $\widetilde{U}\in\mathbb{R}^{n\times r}$.
By $\rm(B.3)$ of Assumption \ref{ass:th1}, taking into account $\{X_k\}_{k\geq\widetilde{K}}$ is bounded,
there exists $\epsilon>0$ such that for all $k\geq\widetilde{K}$,
if $\|\widetilde{U}\|\leq\epsilon/2$ then
\begin{eqnarray}\label{zeng}
\max\big\{\|\mathbb{D}\mathbf{R}_{X_k}(\widetilde{U})\|_{op},
~\|\mathbb{D}\mathbf{R}_{X_k}(\widetilde{U})^{-1}\|_{op}\big\}\leq\kappa_R,
\end{eqnarray}
where $\|\cdot\|_{op}$ is the operator norm.
In the rest of the proof, for ease of notation, we use the notations:
$X:=X_k$,
$\mathbb{H}:=\mathbb{H}_k$,
$p(\cdot):=\mathrm{prox}_{h\circ\mathbf{R}_{X_k}}^{\mathbb{H}_k}(\cdot)$
and $\Gamma_X(\cdot):=\mathbb{D}\mathbf{R}_{X_k}(p(\cdot))$, where $k\geq\widetilde{K}$.

For any $U\in\mathrm{T}_{X}\mathcal{M}$ satisfying
$\|p(U)\|\leq\epsilon/2$, by \eqref{zeng},
it holds that
\begin{equation}
\label{eq_th_kappa_r}
\max\big\{\|\Gamma_X(U)\|_{op},
~\|\Gamma_X(U)^{-1}\|_{op}\big\}\leq\kappa_R,
\end{equation}
which together with \eqref{eq_th_zeta_partial_h} implies $\|\zeta_U\|\leq\kappa_RL_h$.
For such $U$, all $W\in\mathrm{T}_{X}\mathcal{M}$ and all $\widetilde{W}\in\mathbb{R}^{n\times r}$, we have
\begin{eqnarray}
\Upsilon&:=&\langle\Gamma_X(U)^{-*}[\widetilde{W}],
\mathbf{R}_X(p(U))-\mathbf{R}_X(p(W))\rangle\nonumber\\
&=&\langle\Gamma_X(U)^{-*}[\widetilde{W}],
\mathbf{R}_X(p(U))-\mathbf{R}_X(p(W))+\Gamma_X(U)[p(W)-p(U)]\rangle-\langle\widetilde{W},p(W)-p(U)\rangle\nonumber\\
&\leq&\kappa_R\|\widetilde{W}\|\cdot M_2\|p(U)-p(W)\|^2+\|\widetilde{W}\|\cdot\|p(U)-p(W)\|,\label{rong}
\end{eqnarray}
where \eqref{rong} uses \eqref{eq_th_ass3} and \eqref{eq_th_kappa_r}.
If $\widetilde{W}\perp\mathrm{T}_{X}\mathcal{M}$, then $\langle\widetilde{W},p(W)-p(U)\rangle=0$, and therefore
\begin{eqnarray}\label{r2}
\Upsilon\leq\kappa_R\|\widetilde{W}\|\cdot M_2\|p(U)-p(W)\|^2.
\end{eqnarray}

Since $h$ is convex, $\partial h$ is a monotone mapping.
Then for any $U,~W\in\mathrm{T}_{X}\mathcal{M},
~\eta\in\partial h(\mathbf{R}_{X}(p(U)))$
and $\eta^{\prime}\in\partial h(\mathbf{R}_{X}(p(W)))$,
it holds that
\begin{equation}
\label{eq_th_h_mono}
\langle\eta-\eta^{\prime},\mathbf{R}_X(p(U))-\mathbf{R}_X(p(W))\rangle
\geq0.
\end{equation}
Let $\eta_U:=\Gamma_X(U)^{-*}[\mathbb{H}[U-p(U)]+\zeta_U]$ and
$\eta_W:=\Gamma_X(W)^{-*}[\mathbb{H}[W-p(W)]+\zeta_W]$,
where $\zeta_U$ and $\zeta_W$ are given as in \eqref{eq_th_zeta_partial_h}.
Then we have $\eta_U\in\partial h(\mathbf{R}_{X}(p(U)))$ and $\eta_W\in\partial h(\mathbf{R}_{X}(p(W)))$.
Substituting $\eta=\eta_U$ and $\eta^{\prime}=\eta_W$ into \eqref{eq_th_h_mono} yields
\[
\langle\Gamma_X(U)^{-*}[\mathbb{H}[U-p(U)]+\zeta_U]
- \Gamma_X(W)^{-*}[\mathbb{H}[W-p(W)]+\zeta_W],
\mathbf{R}_X(p(U))-\mathbf{R}_X(p(W))\rangle
\geq
0.
\]
Then the above inequality can be reformulated as
\begin{eqnarray}
\nonumber
0
&\leq&
\langle\Gamma_X(U)^{-*}\mathbb{H}[(U-p(U))-(W-p(W))],
\mathbf{R}_X(p(U))-\mathbf{R}_X(p(W))\rangle
\\
\nonumber
&&
+\underbrace{\langle\Gamma_X(U)^{-*}(\Gamma_X(W)^{*}-\Gamma_X(U)^{*})[\eta_W],
\mathbf{R}_X(p(U))-\mathbf{R}_X(p(W))\rangle}_{\Upsilon_1}
\\
\label{eq:th_upsilon_sum}
&&
+\underbrace{\langle\Gamma_X(U)^{-*}[\zeta_U-\zeta_W],
\mathbf{R}_X(p(U))-\mathbf{R}_X(p(W))\rangle}_{\Upsilon_2}.
\end{eqnarray}
By \eqref{eq_th_ass4}, $\|\Gamma_X(W)^{*}-\Gamma_X(U)^{*}\|\leq 2M_2\|p(U)-p(W)\|$.
From \eqref{eq_th_zeta_partial_h} and $\eta_W\in\partial h(\mathbf{R}_X(p(W)))$, it holds that $\|\eta_W\|\leq L_h$.
Combining these results with \eqref{rong} and \eqref{r2}, we can deduce that
\begin{eqnarray}
\Upsilon_1&\leq&2L_hM_2\|p(U)-p(W)\|^2(1+\kappa_RM_2\|p(U)-p(W)\|),\label{eq:th_upsilon1}\\
\Upsilon_2&\leq&2L_hM_2\kappa_R^2\|p(U)-p(W))\|^2,\label{eq:th_upsilon2}
\end{eqnarray}
where \eqref{eq:th_upsilon2} uses the fact that $\|\zeta_U\|\leq\kappa_RL_h$, $\|\zeta_W\|\leq\kappa_RL_h$ and $\zeta_U,~\zeta_W\perp\mathrm{T}_{X}\mathcal{M}$.

By \eqref{eq:4_hess_neigh} and Lemma \ref{lm_th_sigma_ub},
we can see that there exists $\widetilde{\kappa}_2>0$ such that
\begin{equation}
\label{eq_th_tilde_kappa2}
\frac{4}{5}\delta\|V\|^2
\leq
\langle V, \mathbb{H}[V]\rangle
=\langle V, (\mathrm{Hess}f(X)+\sigma I)[V]\rangle
\leq\widetilde{\kappa}_2\|V\|^2,
~\forall~V\in\mathrm{T}_{X}\mathcal{M}.
\end{equation}
Thus, by \eqref{eq_th_ass3}, \eqref{eq_th_kappa_r} and \eqref{eq_th_tilde_kappa2}, we can deduce that
\begin{small}
\begin{eqnarray}
\nonumber
&&
\|p(U)-p(W)\|_{\mathbb{H}}^2
-M_2\widetilde{\kappa}_2\kappa_R\|p(U)-p(W)\|^3
\\
\nonumber
&\leq&
\|p(U)-p(W)\|_{\mathbb{H}}^2+\langle\Gamma_X(U)^{-*}\mathbb{H}[p(U)-p(W)],
\mathbf{R}_X(p(U))-\mathbf{R}_X(p(W))-\Gamma_X(U)[p(U)-p(W)]\rangle
\\
\nonumber
&=&
\langle\Gamma_X(U)^{-*}\mathbb{H}[p(U)-p(W)],
\mathbf{R}_X(p(U))-\mathbf{R}_X(p(W))\rangle
\\
\nonumber
&\leq&
\langle\Gamma_X(U)^{-*}\mathbb{H}[U-W],
\mathbf{R}_X(p(U))-\mathbf{R}_X(p(W))\rangle
+\Upsilon_1+\Upsilon_2
\quad\quad\quad\quad\quad\quad\quad\quad\quad\quad\quad\quad\quad\quad\quad\quad\quad\quad
\mathrm{by}~\eqref{eq:th_upsilon_sum}
\\
\nonumber
&\leq&
\langle U-W,p(U)-p(W)\rangle_{\mathbb{H}}
+\langle\Gamma_X(U)^{-*}\mathbb{H}[U-W],
\mathbf{R}_X(p(U))-\Gamma_X(U)[p(U)-p(W)]-\mathbf{R}_X(p(W))\rangle
\\
\nonumber
&&
+2L_hM_2\|p(U)-p(W)\|^2
(1+\kappa_R^2+\kappa_RM_2\|p(U)-p(W)\|),
\quad\quad\quad\quad\quad\quad\quad\quad\quad\quad\quad\quad~~
\mathrm{by}~\eqref{eq:th_upsilon1}~\mathrm{and}~\eqref{eq:th_upsilon2}
\\
\nonumber
&\leq&
\langle U-W,p(U)-p(W)\rangle_{\mathbb{H}}
+M_2\widetilde{\kappa}_2\kappa_R\|U-W\|\cdot\|p(U)-p(W)\|^2
\\
\label{eq_th_super3}
&&
+2L_hM_2\|p(U)-p(W)\|^2
(1+\kappa_R^2+\kappa_RM_2\|p(U)-p(W)\|).
\end{eqnarray}
\end{small}
If $\|U-W\|\leq\epsilon$ and $\|p(U)-p(W)\|\leq\epsilon$, using \eqref{eq_th_tilde_kappa2} and \eqref{eq_th_super3}, we can obtain
\begin{eqnarray}
\nonumber
&&
\langle U-W,p(U)-p(W)\rangle_{\mathbb{H}}
\\
\nonumber
&\geq&
(1-\frac{5}{2\delta}L_hM_2(1+\kappa_R^2))\|p(U)-p(W)\|_{\mathbb{H}}^2
-2(\widetilde{\kappa}_2
+L_hM_2) M_2\kappa_R\epsilon\|p(U)-p(W)\|^2
\\
\label{eq_th_lip3}
&\geq&
\underbrace{\left(1-\frac{5}{2\delta}L_hM_2(1+\kappa_R^2)
-\frac{5}{2\delta}(\widetilde{\kappa}_2
+L_hM_2)M_2\kappa_R\epsilon\right)}_{c_3}
\|p(U)-p(W)\|_{\mathbb{H}}^2.
\end{eqnarray}
Since $\kappa_R\in(1,\sqrt{\frac{2\delta}{5L_hM_2}-1})$ and $\delta>5L_hM_2$,
we can see that if $\epsilon$ is sufficiently small then $c_3>0$.
Combining \eqref{eq_th_lip3} with the Cauchy-Schwarz inequality yields
\[
\|p(U)-p(W)\|_{\mathbb{H}}
\leq
\frac{1}{c_3}\|U-W\|_{\mathbb{H}},
~\forall~U,W\in\mathrm{T}_{X}\mathcal{M},~\mathrm{s.t.}~\|U-W\|\leq\epsilon,~\|p(U)-p(W)\|\leq\epsilon.
\]
Let $\varsigma:=1/c_3$. Then the assertion \eqref{eq_th_lip} holds.
\end{proof}

\vskip 2mm

Next we present our main results of this section.

\vskip 2mm

\begin{theorem}
\label{thm_superlinear}
Suppose Assumptions \ref{ass:2} and \ref{ass:th1} hold.
Then the sequence $\{X_k\}$ generated by Algorithm \ref{algo2} converges locally q-superlinearly to $X^*$.
\end{theorem}

\begin{proof}
Denote $\xi^*_k:=\mathbf{R}_{X_k}^{-1}(X^*)$.
By Lemma \ref{lm_th_alpha}, $X_{k+1}=\mathbf{R}_{X_k}(V_k)$ for sufficiently large $k$.
Since $X_k\to X^*$, we have $\|\xi^*_k\|\to0$ and $\|V_k\|\to0$.
From \cite[(4.28)]{manpqn2023}, we know that there exists $\varepsilon>0$ such that
\begin{equation}
\label{eq_th_super1}
(1-\varepsilon)\|V_k-\xi^*_k\|
\leq
\|X_{k+1}-X^*\|
=
\|\mathbf{R}_{X_k}(V_k)-\mathbf{R}_{X_k}(\xi^*_k)\|
\leq
(1+\varepsilon)\|V_k-\xi^*_k\|,
\end{equation}
for sufficiently large $k$.

Denote $g_k^*:=\mathrm{Proj}_{\mathrm{T}_{X_k}\mathcal{M}}\nabla(f\circ\mathbf{R}_{X_k})(\xi^*_k)=\mathrm{grad}(f\circ\mathbf{R}_{X_k})(\xi^*_k)$.
Then, it holds that
\begin{eqnarray}
\nonumber
&&\|g_k^*-g_k-\mathrm{Hess}f(X_k)[\xi^*_k]\|
\\
\nonumber
&=&
\|
\mathrm{grad}(f\circ\mathbf{R}_{X_k})(\xi^*_k)
-\mathrm{grad}(f\circ\mathbf{R}_{X_k})(0_{X_k})
-\mathrm{Hess}(f\circ\mathbf{R}_{X_k})(0_{X_k})[\xi^*_k]\|
\\
\label{eq_th_super_hess_lip}
&\leq&\underbrace{\max_{0\leq t\leq1}\|\mathrm{Hess}(f\circ\mathbf{R}_{X_k})(t\xi^*_k)-\mathrm{Hess}(f\circ\mathbf{R}_{X_k})(0_{X_k})\|}_{r_k}\cdot\|\xi^*_k\|.
\end{eqnarray}
By $X_k\to X^*$, $\|\xi^*_k\|\to0$ and Assumption \ref{ass:2}, we know that $r_k$ converges to $0$.
Let
\begin{equation}
\label{eq_th_new_sub_prob}
\varphi_k^*(V):=\langle g_k^*,V\rangle+\frac{1}{2}\langle V, \mathbb{H}_k[V]\rangle +h(\mathbf{R}_{X_k}(\xi^*_k+V)),
\quad\forall~V\in\mathrm{T}_{X_k}\mathcal{M}.
\end{equation}
From $X^*\in\arg\min_{X\in\mathcal{M}}f(X)+h(X)$, it holds that $\xi^*_k\in\arg\min_{\xi\in\mathrm{T}_{X_k}\mathcal{M}}f(\mathbf{R}_{X_k}(\xi))+h(\mathbf{R}_{X_k}(\xi))$,
which implies $0\in g_k^*+\mathrm{Proj}_{\mathrm{T}_{X_k}\mathcal{M}}(\mathbb{D}\mathbf{R}_{X_k}(\xi^*_k))^*[\partial h(\mathbf{R}_{X_k}(\xi^*_k))]$.
Using this relation, similar to the proof of Lemma \ref{lm_th_unique}, we can prove that $\varphi_k^*(V)-\varphi_k^*(0)\geq(\frac{2}{5}\delta-L_hM_2)\|V\|^2>0$ for all nonzero $V\in\mathrm{T}_{X_k}\mathcal{M}$.
Thus, $0=\arg\min_{V\in\mathrm{T}_{X_k}\mathcal{M}}\varphi_k^*(V)$, which together with \eqref{eq:th_def_prox} implies $\xi^*_k=\mathrm{prox}_{h\circ\mathbf{R}_{X_k}}^{\mathbb{H}_k}(-\mathbb{H}_k^{-1}[g_k^*]+\xi^*_k)$.

Let $\epsilon>0$ be given as in Lemma \ref{lm_th_lip}.
Let $U:=-\mathbb{H}_k^{-1}[g_k]$ and $W:=-\mathbb{H}_k^{-1}[g_k^*]+\xi^*_k$.
Then $V_k=\mathrm{prox}_{h\circ\mathbf{R}_{X_k}}^{\mathbb{H}_k}(U)$ and $\xi^*_k=\mathrm{prox}_{h\circ\mathbf{R}_{X_k}}^{\mathbb{H}_k}(W)$.
Since $X_k\to X^*$, for sufficiently large $k$, we know that $U$ and $W$ satisfy $\|U-W\|\leq\epsilon$ and \eqref{eq_th_prox_eps},
and therefore Lemma \ref{lm_th_lip} can be applied.
Then, by \eqref{eq_th_super1} and \eqref{eq_th_lip}, we have
\begin{eqnarray}
\nonumber
&&
\|X_{k+1}-X^*\|
\leq
(1+\varepsilon)\|V_k-\xi^*_k\|
\\
\nonumber
&\leq&
(1+\varepsilon)\|\mathrm{prox}_{h\circ\mathbf{R}_{X_k}}^{\mathbb{H}_k}(-\mathbb{H}_k^{-1}[g_k])
-\mathrm{prox}_{h\circ\mathbf{R}_{X_k}}^{\mathbb{H}_k}(-\mathbb{H}_k^{-1}[g_k^*]+\xi^*_k)\|
\\
\nonumber
&\leq&
\frac{(1+\varepsilon)\varsigma}{\sqrt{4\delta/5}}\|\xi^*_k+\mathbb{H}_k^{-1}[g_k-g_k^*]\|_{\mathbb{H}_k}
\leq
\frac{5(1+\varepsilon)\varsigma}{4\delta}\|g_k^*-g_k-\mathbb{H}_k[\xi^*_k]\|
\\
\label{eq_th_thm_super2}
&\leq&
\frac{5(1+\varepsilon)\varsigma}{4\delta}\big(r_k\|\xi^*_k\|+\sigma_k\|\xi^*_k\|\big)
\leq
\frac{5(1+\varepsilon)\varsigma}{4(1-\varepsilon)\delta}\big(r_k\|X_k-X^*\|+\sigma_k\|X_k-X^*\|\big)
\\
\label{eq_th_thm_super3}
&=&
o(\|X_k-X^*\|),
\end{eqnarray}
where the first inequality of \eqref{eq_th_thm_super2} follows from \eqref{eq_th_super_hess_lip} and the second inequality of \eqref{eq_th_thm_super2} follows from \cite[(4.28)]{manpqn2023};
\eqref{eq_th_thm_super3} follows from the facts that $r_k\to0$ and $\sigma_k\to0$.
\end{proof}

\subsection{Local Superlinear Convergence for the Case of Quasi-Newton Approximation $H_k$}

In this subsection, we consider the case that $H_k$ in \eqref{sub_prob1} is generated by the Quasi-Newton method.
Using the damped technique introduced in \cite{damp2017,damp1978} to update $H_k$, we can ensure that there exist $\kappa_1,~\kappa_2>0$ such that
\begin{eqnarray}\label{xi}
\kappa_1\|V\|^2
\leq\|V\|_{H_{k}}^2\leq \kappa_2\|V\|^2,\quad\forall~V\in {\rm T}_{X_k}\mathcal{M},~\forall~k\geq0.
\end{eqnarray}
Let $\varphi_k(V)$ be defined by \eqref{sub_prob1},
and we select a $V_k\in\arg\min_{V\in\mathrm{T}_{X_k}\mathcal{M}}\varphi_k(V)$ as the search direction of $F$ at $X_k$.
Analogous to the proof of Theorem \ref{lm_th_alpha_lb},
we can prove that $\lim_{k\to\infty}\|V_k\|=0$, and all accumulation points of $\{X_k\}$ are stationary points of problem \eqref{eq:prob}.
Let $X^*$ be an  accumulation point of $\{X_k\}$ such that \eqref{eq:4_hess} holds.
Then, using the same argument as that in Lemma \ref{lm_th_accum},
we can obtain that $X_k$ converges to $X^*$.

We further suppose that $H_k$ satisfies the following Dennis-Mor\'e condition (cf. \cite{nocedal2006})
\begin{equation}
\label{cond_dennis}
\lim_{k\to\infty}\frac{\|(H_k-\mathrm{Hess}f(X_k))[V_k]\|}{\|V_k\|}\to0.
\end{equation}
Then we can prove the local superlinear convergence of Algorithm \ref{algo2}.

\vskip 2mm

\begin{theorem}
Suppose Assumptions \ref{ass:2} and \ref{ass:th1} hold.
Further suppose $\{H_k\}$  satisfy \eqref{xi} and \eqref{cond_dennis}.
Then the following statements hold.
\begin{enumerate}
\item[\rm(i)]
There exists $\widehat{K}\geq0$ such that $\alpha_k=1$ can be accepted for all $k\geq \widehat{K}$, and $\lim_{k\to\infty}\sigma_k=0$;
\item[\rm(ii)]
The sequence $\{X_k\}$ generated by Algorithm \ref{algo2} converges locally q-superlinearly to $X^*$.
\end{enumerate}
\end{theorem}

\begin{proof}
The proof of $\rm(i)$ is analogous to Lemma \ref{lm_th_alpha}.
We only give a brief proof of $\rm(ii)$.

For ease of notation, we use the notation $\mathcal{H}_k:=\mathrm{Hess}f(X_k)+\sigma_kI$.
At the current iterate $X_k$, we select a $\widetilde{V}_k$ in the set
$\arg\min_{V\in\mathrm{T}_{X_k}\mathcal{M}}\widetilde{\varphi}_k(V):=\langle g_k, V\rangle
+\frac{1}{2}\langle V,\mathcal{H}_k[V]\rangle+h(\mathbf{R}_{X_{k}}(V))$.
Then $\{\widetilde{V}_k\}$ is bounded. Let $\widetilde{V}$ be an arbitrary accumulation point of $\{\widetilde{V}_k\}$.
Thus, we have $\widetilde{V}\in\arg\min_{V\in\mathrm{T}_{X^*}\mathcal{M}}\varphi^*(V)$,
where $\varphi^*(V):=\langle{\rm grad}f(X^*), V\rangle
+\frac{1}{2}\langle V,{\rm Hess}f(X^*)[V]\rangle+h(\mathbf{R}_{X^*}(V))$.
By \eqref{eq:2_opt}, Assumptions \ref{ass:2} and Lemma \ref{lm_th_unique}, we know that $\arg\min_{V\in\mathrm{T}_{X^*}\mathcal{M}}\varphi^*(V)=\{0_{X^*}\}$.
Thus, $\widetilde{V}=0_{X^*}$, and therefore $\widetilde{V}_k\to0_{X^*}$.

By $\rm(i)$, $X_{k+1}:=\mathbf{R}_{X_k}(V_k)$ for sufficiently large $k$.
Let $\widetilde{X}_{k+1}:=\mathbf{R}_{X_k}(\widetilde{V}_k)$ and $\xi^*_k:=\mathbf{R}_{X_k}^{-1}(X^*)$.
By \cite[(4.28)]{manpqn2023}, there exists $\varepsilon>0$ such that $\|X_{k+1}-X^*\|
\leq(1+\varepsilon)\|V_k-\xi^*_k\|$ and $\|\widetilde{X}_{k+1}-X^*\|\leq(1+\varepsilon)\|\widetilde{V}_k-\xi^*_k\|$ for sufficiently large $k$.
Using the same argument as proving \eqref{eq_th_thm_super3}, we can deduce that
\begin{eqnarray}\label{eq_w}
\|\widetilde{X}_{k+1}-X^*\|\leq(1+\varepsilon)\|\widetilde{V}_k-\xi^*_k\|=o(\|X_k-X^*\|).
\end{eqnarray}
Thus
\begin{eqnarray}
\nonumber
\|X_{k+1}-X^*\|
&\leq&
(1+\varepsilon)(\|V_k-\widetilde{V}_k\|
+\|\widetilde{V}_k-\xi^*_k\|)
\\
\label{eq_th_dennis_super1}
&\leq&
(1+\varepsilon)(\|V_k-\widetilde{V}_k\|
+o(\|X_k-X^*\|).
\end{eqnarray}
By $\widetilde{V}_k\in\arg\min_{V\in\mathrm{T}_{X_k}\mathcal{M}}\widetilde{\varphi}_k(V)$ and $V_k\in\arg\min_{V\in\mathrm{T}_{X_k}\mathcal{M}}\varphi_k(V)$, 
similar to the proof of Lemma \ref{lm_th_unique}, we can obtain
\begin{eqnarray}
\nonumber
\widetilde{V}_k
&\in&
\mathop{\arg\min}\limits_{V\in\mathrm{T}_{X_k}\mathcal{M}}
\langle g_k+\mathcal{H}_k[\widetilde{V}_k], V\rangle
+h(\mathbf{R}_{X_k}(V))
+L_hM_2\|V-\widetilde{V}_k\|^2,
\\
\nonumber
V_k
&\in&
\mathop{\arg\min}\limits_{V\in\mathrm{T}_{X_k}\mathcal{M}}
\langle g_k+(H_k+\sigma_kI)[V_k], V\rangle
+h(\mathbf{R}_{X_k}(V))
+L_hM_2\|V-V_k\|^2.
\end{eqnarray}
Thus, it holds that
\begin{eqnarray}
\nonumber
\langle g_k+\mathcal{H}_k[\widetilde{V}_k],  \widetilde{V}_k- V_k\rangle
+h(\mathbf{R}_{X_k}(\widetilde{V}_k))
&\leq&
h(\mathbf{R}_{X_k}(V_k))
+L_hM_2\|\widetilde{V}_k-V_k\|^2,
\\
\nonumber
\langle g_k+(H_k+\sigma_kI)[V_k], V_k- \widetilde{V}_k\rangle
+h(\mathbf{R}_{X_k}(V_k))
&\leq&
h(\mathbf{R}_{X_k}(\widetilde{V}_k))
+L_hM_2\|\widetilde{V}_k-V_k\|^2.
\end{eqnarray}
Summing the above two inequalities and rearranging, we can obtain
\begin{eqnarray}
\nonumber
&&
\langle \widetilde{V}_k-V_k, (H_k-\mathrm{Hess}f(X_k))[V_k]\rangle
\\
\label{eq:th_dennis1}
&\geq&
\langle \widetilde{V}_k-V_k, \mathcal{H}_k[\widetilde{V}_k-V_k]\rangle
-2L_hM_2\|\widetilde{V}_k-V_k\|^2\geq
L_hM_2\|\widetilde{V}_k-V_k\|^2,
\end{eqnarray}
where the second inequality follows from \eqref{eq:4_hess_neigh}.
Combining \eqref{eq:th_dennis1} and \eqref{cond_dennis} yields $\|V_k-\widetilde{V}_k\|=o(\|V_k\|)$,
which together with \eqref{eq_th_dennis_super1} implies
\begin{eqnarray}
\label{eq_th_dennis_super3}
\|X_{k+1}-X^*\|=o(\|V_k\|)+o(\|X_k-X^*\|).
\end{eqnarray}
Since $\|V_k-\widetilde{V}_k\|=o(\|V_k\|)$, there exists $\theta_1>0$ such that $\|V_k\|
\leq
\theta_1\|\widetilde{V}_k\|$ for all sufficiently large $k$.

By $X_k\to X^*$, we know that there exists $\theta_2>0$ such that $\|\widetilde{V}_k\|=\|\mathbf{R}_{X_k}^{-1}(\widetilde{X}_{k+1})\|
\leq
\theta_2\|\widetilde{X}_{k+1}-X_k\|$ for all sufficiently large $k$.
Thus, we have
\begin{eqnarray}
\nonumber
\|V_k\|
&\leq&
\theta_1\theta_2\|\widetilde{X}_{k+1}-X_k\|
\leq
\theta_1\theta_2(\|\widetilde{X}_{k+1}-X^*\|+\|X_k-X^*\|)
\\
\nonumber
&\leq&
\theta_1\theta_2(o(\|X_k-X^*\|)+\|X_k-X^*\|)
=
O(\|X_k-X^*\|),
\end{eqnarray}
where the last inequality uses \eqref{eq_w}.
Substituting $\|V_k\|=O(\|X_k-X^*\|)$ into \eqref{eq_th_dennis_super3} yields $\|X_{k+1}-X^*\|=o(\|X_k-X^*\|)$.
\end{proof}

\section{Numerical Experiments}
\label{sec:5}

In this section, we consider applying our algorithm to the compressed modes (CM) problems and the sparse principal component analysis (sparse PCA) problems, which will be introduced in details later.
Although ARPN shows local superlinear convergence rate, 
due to the high computational cost of solving the subproblem \eqref{sub_prob1},  it is difficult to apply ARPN to large-scale composite optimization problems.
In the following, we only present numerical results of ARPQN for \eqref{eq:prob}. 
To demonstrate the practical efficiency of ARPQN, 
we compare it with other existing numerical methods for composite optimization problems, 
including ManPG, ManPG-Ada (both methods are proposed in \cite{mashiqian2020}), ManPQN \cite{manpqn2023} and a semismooth Newton based augmented Lagrangian (named ALMSSN) method proposed in \cite{dingchao2022}.
The above algorithms are implemented in MATLAB R2018b and run on a PC with Intel Core i5 CPU (2.3GHz) and 8GB memory.

For the subproblems of ManPG, ManPG-Ada, ManPQN and ARPQN,
we use the adaptive regularized semismooth Newton (ASSN) method to solve them.
Similar with ManPG, ManPG-Ada and ManPQN, the stopping criterion of ARPQN is set as either $\|V_k\|^2\leq10^{-8}nr$ or the algorithm reaches the maximum iteration number $70000$. The maximum iteration number for solving the subproblem \eqref{eq:sub_prob} is set as $100$.
We set the parameter $\vartheta_k$ of ARPQN to $\max\{\frac{\mathrm{tr}(y_{k-1}^Ty_{k-1})}{\mathrm{tr}(s_{k-1}^Ty_{k-1})},\vartheta_0\}$, where $\vartheta_0>0$ is a given constant, $s_{k-1}$ and $y_{k-1}$ are defined in subsection \ref{subsec:algo1}.
The initial value of the regularization parameter is set as $\sigma_0=1$.
For ALMSSN, we choose the QR decomposition as the retraction mapping;
for ManPG, ManPG-Ada and ManPQN, the singular value decomposition (SVD) is used as the retraction mapping;
the retraction mapping for ARPQN will be discussed later.
The parameters used in ManPG, ManPG-Ada, ManPQN and ALMSSN are set to be the default values in \cite{mashiqian2020}, \cite{manpqn2023} and \cite{dingchao2022}, respectively.
We conduct numerical experiments utilizing Algorithm \ref{alg:ada_manpqn}, which is equipped with either monotone line search strategy or nonmonotone line search strategy in steps 10-12.
For simplicity, we denote the former algorithm as ARPQN and denote the latter as NLS-ARPQN.
In practical implementation, NLS-ARPQN demonstrates better numerical performance compared to ARPQN.

In the following, we firstly introduce how to use the ASSN method to solve the subproblem \eqref{eq:sub_prob} in practical.
Next, we do some numerical tests on different groups of parameters $(\eta_1,\eta_2,\gamma_1,\gamma_2)$ for NLS-ARPQN on the CM problem. Then we choose the parameter combination with the best performance and apply it to NLS-ARPQN for all test problems.
Moreover, in order to investigate the effect of retraction in the NLS-ARPQN method, we provide numerical experiments on CM problems encompassing different types of retraction, including SVD, QR decomposition and Cayley transformation.
It can be observed that NLS-ARPQN with the retraction based on SVD outperforms NLS-ARPQN with other retractions. Consequently, we apply the retraction based on SVD to NLS-ARPQN for all test problems.
Numerical results of the above mentioned algorithms are averaged on $50$ randomly generated instances with different random initial points. 
Figures and tables below report the averaged results of each algorithm, including running time in seconds, iteration number, sparsity of solution $X^*$, the total number of line search steps and the averaged number of iterations of the ASSN method.

\subsection{The ASSN Method for Solving \eqref{eq:sub_prob}}

In our implementation, $\mathcal{B}_k$ in \eqref{eq:sub_prob} is replaced by $\mathbf{B}_k$ which satisfies \eqref{eq:3_bkdiag}, i.e. 
${\rm tr}(V^T\mathbf{B}_k[V])={\rm tr}(V^T({\rm diag}B_k)V)$ for any $V\in{\rm T}_{X_k}\mathcal{M}$, where $B_k$ is updated by \eqref{eq3_eu_bk}.
We use $\mathbb{B}_k$ to denote ${\rm diag}B_k+\sigma_kI$.
Based on \eqref{eq:3_bkdiag}, the Lagrangian function for \eqref{eq:sub_prob} can be formulated as
\begin{eqnarray}
\nonumber
\mathcal{L}_k(V,\Lambda)
&=&
\langle\nabla f(X_k),V\rangle+\frac{1}{2}{\rm tr}(V^T(\mathbf{B}_k+\sigma_kI)[V])
+h(X_k+V)-\langle \Lambda,\mathcal{A}_k(V)\rangle
\\
\nonumber
&=&
\langle\nabla f(X_k)-\mathcal{A}_k^* (\Lambda),V\rangle+\frac{1}{2}{\rm tr}(V^T\mathbb{B}_kV)
+h(X_k+V),
\end{eqnarray}
where the symmetric matrix $\Lambda\in\mathbb{R}^{r\times r}$ is the Lagrange multiplier for the constraint $\mathcal{A}_k(V):=V^TX_k+X_k^TV=0$, which means $V\in{\rm T}_{X_k}\mathcal{M}$, and $\mathcal{A}_k^*(\cdot)$ is denoted the adjoint operator of $\mathcal{A}_k(\cdot)$.
Let $V(\Lambda):=\arg\min_{V}\mathcal{L}_k(V, \Lambda)$. 
Then $V(\Lambda)={\rm prox}_h^{\mathbb{B}_k}\big(X_k-\mathbb{B}_k^{-1}(\nabla f(X_k)-\mathcal{A}_k^*(\Lambda))\big)-X_k$ 
where ${\rm prox}_h^{\mathbb{B}_k}(x)=\arg\min_{y}\left\{h(y)+\frac{1}{2}\|y-x\|_{\mathbb{B}_k}\right\}$.
Substituting $V(\Lambda)$ into $\mathcal{A}_k(V)=0$ yields
\begin{equation}\label{eq3_e}
E(\Lambda):=\mathcal{A}_k(V(\Lambda)) = V(\Lambda)^TX_k+X_k^TV(\Lambda) = 0.
\end{equation}
It can be proved  that the operator $E(\cdot)$ is monotone and Lipschitz continuous, and thus the ASSN method can be applied to solving \eqref{eq3_e}. 

Before applying the ASSN method, we need to vectorize $E(\Lambda)$. 
Since $E(\Lambda)$ and $\Lambda$ are both symmetric, we only focus on the vectorization of their lower triangular part. 
Thus there exists a duplication matrix $U_r\in\mathbb{R}^{r^2\times\frac{1}{2}r(r+1)}$ such that
$\overline{\rm vec}(\Lambda)=U_r^+{\rm vec}(\Lambda)$,
where $\overline{\rm vec}(\Lambda)$ denotes the vectorization of the lower triangular part of $\Lambda$
and $U_r^+=(U_r^TU_r)^{-1}U_r$ denotes the Moore-Penrose inverse of $U_r$. 
Define
\[
\mathcal{G}(\overline{\rm vec}(\Lambda)) 
:= 4 U_r^+ (I_r\otimes X_k^T)
\mathcal{J}(y)\vert_{y={\rm vec}(X(\Lambda))}
(I_r\otimes (\mathbb{B}_k^{-1}X_k))U_r,
\]
where $\mathcal{J}(y)$ is the generalized Jacobian of ${\rm prox}_h^{\mathbb{B}_k}(y)$ and $X(\Lambda):=X_k-\mathbb{B}_k^{-1}(\nabla f(X_k)-\mathcal{A}_k^*(\Lambda))$. It can be proved that $\mathcal{G}(\overline{\rm vec}(\Lambda)) \in \partial \overline{\rm vec}(E(U_r\overline{\rm vec}(\Lambda)))$ by the procedure in \cite[section 4.2]{mashiqian2020}. In ASSN, the conjugate gradient method is used to compute the Newton step $d_{\ell}$ at the current iterate $\Lambda_{\ell}$ by solving
\[
(\mathcal{G}(\overline{\rm vec}(\Lambda_{\ell}))+\eta I)d_{\ell} = - \overline{\rm vec}(E(\Lambda_{\ell})),
\]
where $\eta>0$ is a regularization parameter. 
Then, we use the same strategy as that in \cite{assn2018} to obtain the next iterate $\Lambda_{\ell+1}$.
For more details about the ASSN method, we refer the reader to \cite{assn2018,mashiqian2020}.

\subsection{Compressed Modes Problem}\label{sec51}

In this subsection, we consider the compressed modes (CM) problem \cite{cm2013} which looks for spatially localized sparse solutions of the independent-particle Schr{\"o}dinger's equation. For the 1D free-electron case, the CM problem can be
\begin{equation}
\label{eq:5_cm}
\min_{X\in\mathcal{M}}tr(X^THX)+\mu\|X\|_1,
\end{equation}
where $H$ is a discretization of the Schr{\"o}dinger operator.

In this subsection, we firstly compare different combinations of parameters $(\eta_1,\eta_2,\gamma_1,\gamma_2)$ on the CM problems with different $n$ and $r$. 
We report the numerical results in Tables \ref{tb50_1} and \ref{tb50_2}, from which 
we can observe that the performance of NLS-ARPQN is not sensitive to different parameters. 
In particular, the parameter combination $(\eta_1,\eta_2,\gamma_1,\gamma_2)=(0.2,0.9,0.3,3)$ requires less iterations and running time than other parameter combinations in most cases of $n$ and $r$. Thus, we apply this parameter combination to our algorithm for all test problems.

\begin{table}[htbp]\centering
\caption{Comparison of different parameters $(\eta_1,\eta_2,\gamma_1,\gamma_2)$ for NLS-ARPQN on CM problems, different $n=\{64,128,256,512\}$ with $r= 4$ and $\mu=0.1$.
The best result is marked in bold.
}\label{tb50_1}
\begin{tabular}{ccccccc}
\midrule
$n=64$ & Iter & $F(X^*)$  & sparsity & CPU time & \# line-search &  SSN iters  \\
\midrule
$(0.1,0.9,0.3,3)$ &   109.04    & 1.425  &    0.81  &   0.0301   &   303.92   &   1.57  \\ 
$(0.1,0.7,0.3,3)$ &   107.96    & 1.425  &    0.80  &   {\bf 0.0291}   &   300.02   &   1.59  \\ 
$(0.2,0.9,0.3,3)$ &   110.20    & 1.425  &    0.80  &   0.0340   &   398.22   &   1.56  \\ 
$(0.2,0.7,0.3,3)$ &   108.90    & 1.425  &    0.80  &   0.0345   &   406.52   &   1.56  \\ 
$(0.2,0.5,0.3,3)$ &   {\bf 107.18}    & 1.425  &    0.80  &   0.0338   &   386.52   &   1.52  \\ 
$(0.2,0.5,0.1,2)$ &   111.76    & 1.425  &    0.80  &   0.0367   &   439.66   &   1.54  \\ 
$(0.2,0.5,0.1,3)$ &   109.08    & 1.425  &    0.80  &   0.0361   &   428.40   &   1.54  \\ 
$(0.2,0.5,0.1,5)$ &   107.40    & 1.425  &    0.80  &   0.0350   &   408.76   &   1.54  \\ 
\midrule
$n=128$ &   \\
\midrule
$(0.1,0.9,0.3,3)$ &   146.54    & 1.887  &    0.82  &   0.0480   &   474.72   &   1.38  \\ 
$(0.1,0.7,0.3,3)$ &   142.70    & 1.887  &    0.82  &   {\bf 0.0470}   &   453.10   &   1.40  \\ 
$(0.2,0.9,0.3,3)$ &   139.94    & 1.887  &    0.82  &   0.0512   &   554.24   &   1.40  \\ 
$(0.2,0.7,0.3,3)$ &   {\bf 139.06}    & 1.887  &    0.82  &   0.0524   &   575.96   &   1.37  \\ 
$(0.2,0.5,0.3,3)$ &   140.58    & 1.887  &    0.82  &   0.0538   &   600.98   &   1.37  \\ 
$(0.2,0.5,0.1,2)$ &   147.06    & 1.887  &    0.82  &   0.0572   &   647.66   &   1.34  \\ 
$(0.2,0.5,0.1,3)$ &   149.34    & 1.887  &    0.82  &   0.0602   &   688.34   &   1.35  \\ 
$(0.2,0.5,0.1,5)$ &   146.54    & 1.887  &    0.82  &   0.0575   &   642.56   &   1.36  \\ 
\midrule
$n=256$ &   \\
\midrule
$(0.1,0.9,0.3,3)$ &   253.22    & 2.495  &    0.84  &   0.1382   &   1222.70   &   1.25  \\ 
$(0.1,0.7,0.3,3)$ &   255.34    & 2.495  &    0.84  &   {\bf 0.1378}   &   1218.58   &   1.25  \\ 
$(0.2,0.9,0.3,3)$ &   {\bf 217.50}    & 2.495  &    0.84  &   0.1426   &   1410.02   &   1.26  \\ 
$(0.2,0.7,0.3,3)$ &   221.32    & 2.495  &    0.84  &   0.1514   &   1561.62   &   1.25  \\ 
$(0.2,0.5,0.3,3)$ &   221.02    & 2.495  &    0.84  &   0.1619   &   1720.42   &   1.22  \\ 
$(0.2,0.5,0.1,2)$ &   260.16    & 2.495 &    0.84  &   0.2216   &   2551.18   &   1.13  \\ 
$(0.2,0.5,0.1,3)$ &   250.30    & 2.495  &    0.84  &   0.2027   &   2334.66   &   1.14  \\ 
$(0.2,0.5,0.1,5)$ &   252.52    & 2.495  &    0.84  &   0.1947   &   2171.68   &   1.17  \\ 
\midrule
$n=512$ &   \\
\midrule
$(0.1,0.9,0.3,3)$ &   518.98    & 3.295 &    0.86  &   0.2273   &   550.48   &   0.90  \\ 
$(0.1,0.7,0.3,3)$ &   525.36    & 3.295  &    0.86  &   0.2380   &   563.62   &   0.89  \\ 
$(0.2,0.9,0.3,3)$ &   511.42    & 3.295  &    0.86  &   0.2206   &   561.14   &   0.90  \\ 
$(0.2,0.7,0.3,3)$ &   514.10    & 3.295  &    0.86  &   0.2250   &   568.80   &   0.89  \\ 
$(0.2,0.5,0.3,3)$ &   {\bf 509.80}    & 3.295  &    0.86  &   0.2211   &   565.90   &   0.89  \\ 
$(0.2,0.5,0.1,2)$ &   513.62    & 3.296  &    0.86  &   0.2296   &   550.42   &   0.90  \\ 
$(0.2,0.5,0.1,3)$ &   508.98    & 3.295  &    0.86  &   {\bf 0.2198}   &   552.12   &   0.90  \\ 
$(0.2,0.5,0.1,5)$ &   509.86    & 3.295  &    0.86  &   0.2220   &   556.82   &   0.90  \\
\midrule
\end{tabular}
\end{table}

\begin{table}[htbp]\centering
\caption{Comparison of different parameters $(\eta_1,\eta_2,\gamma_1,\gamma_2)$ for NLS-ARPQN on CM problems, different $r=\{2,4,6,8\}$ with $n= 256$ and $\mu=0.2$.
The best result is marked in bold.
}\label{tb50_2}
\begin{tabular}{ccccccc}
\midrule
$r=2$ & Iter & $F(X^*)$  & sparsity & CPU time & \# line-search &  SSN iters  \\
\midrule
$(0.1,0.9,0.3,3)$ &   143.56    & 2.168  &    0.89  &   0.0630   &   207.40   &   1.37  \\ 
$(0.1,0.7,0.3,3)$ &   151.68    & 2.168  &    0.89  &   0.0680   &   223.86   &   1.36  \\ 
$(0.2,0.9,0.3,3)$ &   {\bf 123.54}    & 2.168  &    0.89  &   {\bf 0.0560}   &   213.74   &   1.32  \\ 
$(0.2,0.7,0.3,3)$ &   134.64    & 2.168  &    0.89  &   0.0625   &   248.44   &   1.31  \\ 
$(0.2,0.5,0.3,3)$ &   140.02    & 2.168  &    0.89  &   0.0662   &   262.90   &   1.34  \\ 
$(0.2,0.5,0.1,2)$ &   179.36    & 2.168  &    0.89  &   0.0965   &   398.72   &   1.29  \\ 
$(0.2,0.5,0.1,3)$ &   172.62    & 2.168  &    0.89  &   0.0888   &   373.38   &   1.34  \\ 
$(0.2,0.5,0.1,5)$ &   171.30    & 2.168  &    0.89  &   0.0836   &   343.10   &   1.36  \\  
\midrule
$r=4$ &   \\
\midrule
$(0.1,0.9,0.3,3)$ &   270.20    & 4.336  &    0.88  &   {\bf 0.1078}   &   951.02   &   1.13  \\ 
$(0.1,0.7,0.3,3)$ &   275.92    & 4.336  &    0.88  &   0.1106   &   989.46   &   1.12  \\ 
$(0.2,0.9,0.3,3)$ &   {\bf 244.28}    & 4.336 &    0.88  &   0.1098   &   1083.74   &   1.14  \\ 
$(0.2,0.7,0.3,3)$ &   244.68    & 4.336  &    0.88  &   0.1173   &   1209.62   &   1.14  \\ 
$(0.2,0.5,0.3,3)$ &   253.36    & 4.336  &    0.88  &   0.1267   &   1326.68   &   1.13  \\ 
$(0.2,0.5,0.1,2)$ &   277.42    & 4.336  &    0.88  &   0.1612   &   1878.98   &   1.06  \\ 
$(0.2,0.5,0.1,3)$ &   274.40    & 4.336  &    0.88  &   0.1493   &   1675.16   &   1.08  \\ 
$(0.2,0.5,0.1,5)$ &   270.54    & 4.336  &    0.88  &   0.1432   &   1572.04   &   1.11  \\ 
\midrule
$r=6$ &   \\
\midrule
$(0.1,0.9,0.3,3)$ &   364.02    & 6.509  &    0.88  &   0.1954   &   1291.42   &   1.28  \\ 
$(0.1,0.7,0.3,3)$ &   375.76    & 6.509 &    0.88  &   0.2027   &   1357.06   &   1.28  \\ 
$(0.2,0.9,0.3,3)$ &   {\bf 307.46}    & 6.509  &    0.88  &   {\bf 0.1856}   &   1338.66   &   1.30  \\ 
$(0.2,0.7,0.3,3)$ &   333.32    & 6.509  &    0.88  &   0.2159   &   1659.54   &   1.29  \\ 
$(0.2,0.5,0.3,3)$ &   335.14    & 6.509  &    0.88  &   0.2244   &   1812.86   &   1.24  \\ 
$(0.2,0.5,0.1,2)$ &   347.88    & 6.509  &    0.88  &   0.2683   &   2360.14   &   1.22  \\ 
$(0.2,0.5,0.1,3)$ &   340.80    & 6.509  &    0.88  &   0.2457   &   2085.14   &   1.22  \\ 
$(0.2,0.5,0.1,5)$ &   342.80    & 6.509  &    0.88  &   0.2371   &   1922.68   &   1.25  \\ 
\midrule
$r=8$ &  \\
\midrule
$(0.1,0.9,0.3,3)$ &   816.40    & 8.690  &    0.87  &   0.5414   &   3043.28   &   1.33  \\ 
$(0.1,0.7,0.3,3)$ &   830.00    & 8.690  &    0.87  &   0.5641   &   3157.40   &   1.35  \\ 
$(0.2,0.9,0.3,3)$ &   {\bf 674.56}    & 8.690  &    0.87  &   {\bf 0.5208}   &   3160.52   &   1.41  \\ 
$(0.2,0.7,0.3,3)$ &   693.60    & 8.690  &    0.87  &   0.5543   &   3496.16   &   1.36  \\ 
$(0.2,0.5,0.3,3)$ &   718.12    & 8.690  &    0.87  &   0.6095   &   4024.88   &   1.36  \\ 
$(0.2,0.5,0.1,2)$ &   812.22    & 8.690  &    0.87  &   0.8188   &   6362.10   &   1.24  \\ 
$(0.2,0.5,0.1,3)$ &   768.32    & 8.690  &    0.87  &   0.7161   &   5283.62   &   1.30  \\ 
$(0.2,0.5,0.1,5)$ &   740.62    & 8.690  &    0.87  &   0.6512   &   4553.80   &   1.32  \\ 
\midrule
\end{tabular}
\end{table}

In order to assess the impact of retractions on the efficiency and accuracy of NLS-ARPQN,
we conduct numerical experiments on the CM problems comparing different types of retractions, including SVD, QR decomposition and Cayley transformation.
Before showing numerical results, we introduce the above retractions.
By \cite[Proposition 7]{AbsMal2012}, the retraction based on SVD, denoted by $\mathbf{R}_X^{\mathrm{SVD}}(\xi)$, is just the $W$ of the polar decomposition $X+\xi=WS$ where $W\in\mathbb{R}^{n\times p}$ is orthogonal and $S\in\mathbb{R}^{p\times p}$ is symmetric positive definite (see Remark \ref{rm_retr});
the retraction based on the QR decomposition \cite{absil2008,mashiqian2020} can be written as
\[
\mathbf{R}_{X}^{\mathrm{QR}}(\xi)=\mathrm{qf}(X+\xi),
\]
where $\mathrm{qf}(A)$ denotes the $Q$ factor of the QR decomposition of $A$;
the retraction based on the Cayley transformation \cite{mashiqian2020,wzwywt} is given by
\[
\mathbf{R}_{X}^{\mathrm{Cayley}}(\xi)=\big(I_n-\frac{1}{2}W(\xi)\big)^{-1}\big(I_n+\frac{1}{2}W(\xi)\big)X,
\]
where $W(\xi)=(I_n-\frac{1}{2}XX^T)\xi X^T-X\xi^T(I_n-\frac{1}{2}XX^T)$.
Then we report numerical results of NLS-ARPQN with above retractions in Tables \ref{tb_cm_retr1}. It can be observed that NLS-ARPQN with the retraction using SVD outperforms NLS-ARPQN with other retractions in terms of iteration number and CPU time.
Thus we apply the SVD retraction to our algorithm for all test problems.

\begin{table}[htbp]\centering
\caption{Comparison of different retractions, including Cayley transformation, QR decomposition and SVD, for NLS-ARPQN on CM problems with different $(n,r,\mu)$.
}\label{tb_cm_retr1}
\begin{tabular}{ccccccc}
\midrule
$(n,r,\mu)=(256,4,0.1)$ & Iter & $F(X^*)$  & sparsity & CPU time & \# line-search &  SSN iters  \\
\midrule
ARPQN-SVD &   367.84    & 2.493  &    0.84  &   0.3349   &   2343.24   &   1.24  \\ 
ARPQN-QR &   862.08    & 2.493  &    0.84  &   0.4502   &   80.44   &   0.98  \\ 
ARPQN-Cayley &   259.80    & 2.493  &    0.84  &   3.1123   &   903.96   &   1.76  \\ 
\midrule
$(n,r,\mu)=(256,8,0.1)$         &  \\
\midrule
ARPQN-SVD &   755.20    & 5.021  &    0.80  &   1.1603   &   4866.40   &   1.52  \\ 
ARPQN-QR &   1433.36    & 5.021  &    0.80  &   3.0891   &   1876.56   &   1.93  \\ 
ARPQN-Cayley &   394.80    & 5.021  &    0.80  &   5.8411   &   1318.64   &   1.90  \\ 
\midrule
$(n,r,\mu)=(256,12,0.1)$         &  \\
\midrule
ARPQN-SVD &   984.16    & 7.759  &    0.77  &   3.7850   &   5792.16   &   1.90  \\ 
ARPQN-QR &   1740.08    & 7.761  &    0.77  &   8.6853   &   1315.92   &   2.28  \\ 
ARPQN-Cayley &   460.88    & 7.760  &    0.77  &   7.5592   &   1473.72   &   1.99  \\ 
\midrule
$(n,r,\mu)=(512,4,0.1)$         &  \\
\midrule
ARPQN-SVD &   453.08    & 3.297  &    0.86  &   0.5356   &   3405.08   &   1.12  \\ 
ARPQN-QR &   1189.84    & 3.297  &    0.86  &   0.5247   &   49.92   &   0.92  \\ 
ARPQN-Cayley &   240.48    & 3.297  &    0.86  &   10.3448   &   891.64   &   1.55  \\ 
\midrule
$(n,r,\mu)=(512,8,0.1)$         &  \\
\midrule
ARPQN-SVD &   696.28    & 6.659  &    0.83  &   1.7551   &   5196.16   &   1.34  \\ 
ARPQN-QR &   1340.64    & 6.660  &    0.83  &   2.9043   &   197.88   &   1.80  \\ 
ARPQN-Cayley &   245.88    & 6.659  &    0.83  &   10.3580   &   839.40   &   1.75  \\ 
\midrule
$(n,r,\mu)=(512,12,0.1)$         &  \\
\midrule
ARPQN-SVD &   839.12    & 10.167  &    0.81  &   3.2831   &   6198.68   &   1.65  \\ 
ARPQN-QR &   1709.68    & 10.169  &    0.81  &   6.9031   &   987.64   &   2.25  \\ 
ARPQN-Cayley &   287.16    & 10.167  &    0.81  &   12.7652   &   991.56   &   1.86  \\ 
\midrule
\end{tabular}
\end{table}

We compare ManPG, ManPG-Ada, ALMSSN, ManPQN, ARPQN and NLS-ARPQN on the CM problem and report their numerical results in Figures \ref{fig51_1}-\ref{fig51_3} and Tables \ref{tb51_1}-\ref{tb51_3}. It can be observed that NLS-ARPQN outperforms other algorithms in most cases. 
Compared to proximal gradient type methods, NLS-ARPQN requires less iterations and CPU time to converge, especially when $n$ and $r$ are large. 
Comparing ARPQN with NLS-ARPQN, we can observe that NLS-ARPQN shows better performance in terms of iteration number and running time than ARPQN due to the nonmonotone line search strategy.
The ALMSSN method usually outperforms ManPG, ManPG-Ada and ManPQN, especially when $n$ and $r$ become larger.
In most cases, ALMSSN needs less iterations and more running time to converge than ARPQN and NLS-ARPQN.

For ManPG, ManPG-Ada, ManPQN, ARPQN and NLS-ARPQN, the total number of line search steps and the averaged iteration number of the ASSN method for solving the subproblem are reported in the last two columns of Tables \ref{tb51_1}-\ref{tb51_3}.
Comparing NLS-ARPQN with ManPQN, we can see that NLS-ARPQN needs less line search steps, which indicates the role of the adaptive parameter $\sigma_k$ for accelerating the convergence of our algorithm. 
In particular, when $n$ and $r$ are large, NLS-ARPQN needs less averaged iterations of ASSN method than ManPQN in most cases due to the fact that the regularization parameter $\sigma_k$ makes the condition number of subproblem \eqref{eq:sub_prob} better.

\begin{figure}[H]
	\centering  
	\subfigbottomskip=2pt 
	\subfigcapskip=-5pt 
	\subfigure[CPU]{
		\includegraphics[width=0.45\linewidth]{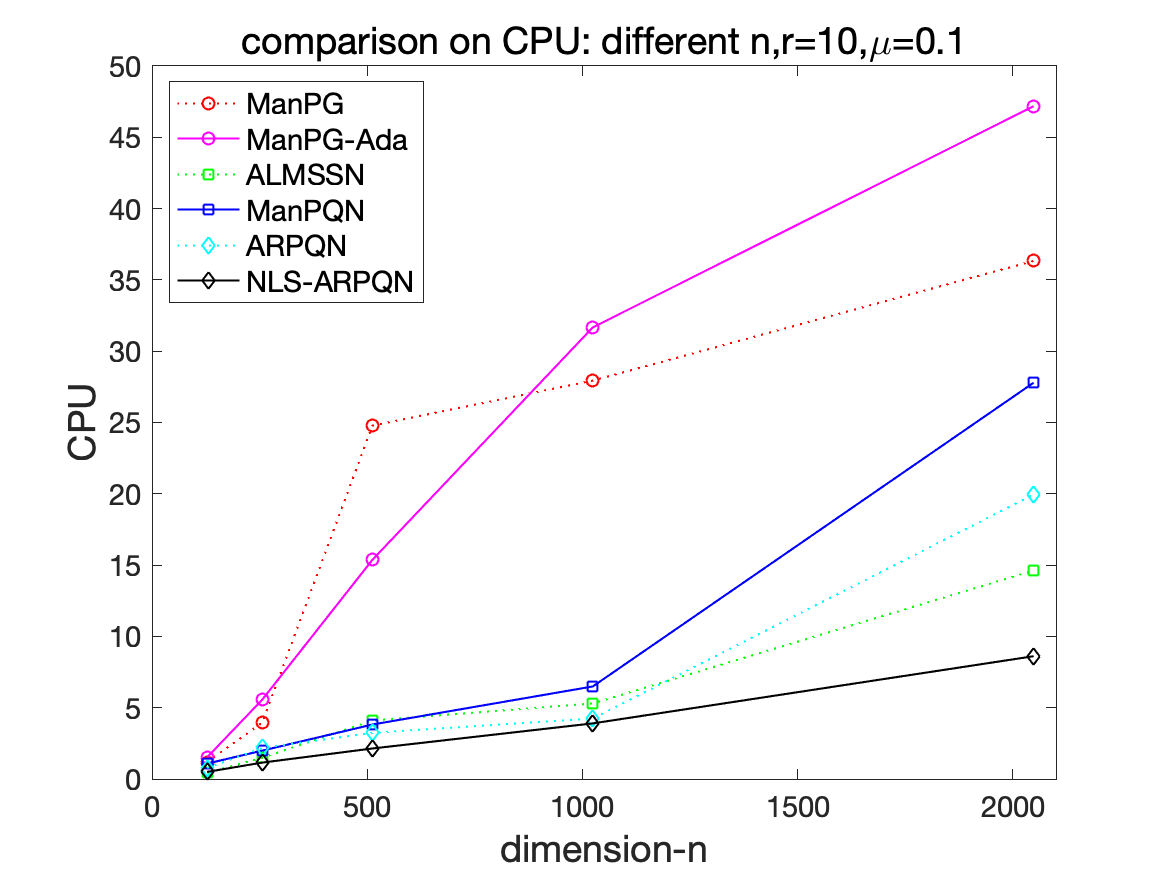}}
		\hspace{-5mm}
	\subfigure[Iter]{
		\includegraphics[width=0.45\linewidth]{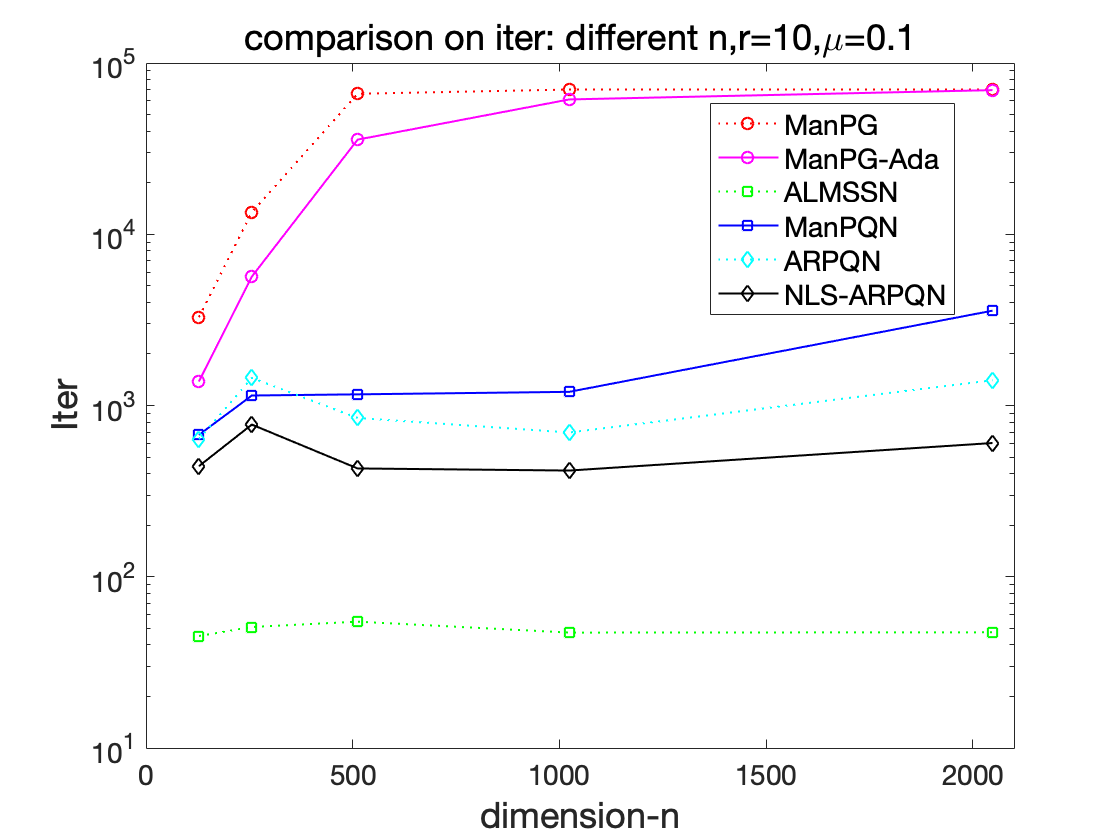}}
	\caption{Comparison on CM problem, different $n=\{128,256,512,1024,2048\}$ with $r= 10$ and $\mu=0.1$.}\label{fig51_1}
\end{figure}

\begin{figure}[H]
	\centering  
	\subfigbottomskip=2pt 
	\subfigcapskip=-5pt 
	\subfigure[CPU]{
		\includegraphics[width=0.45\linewidth]{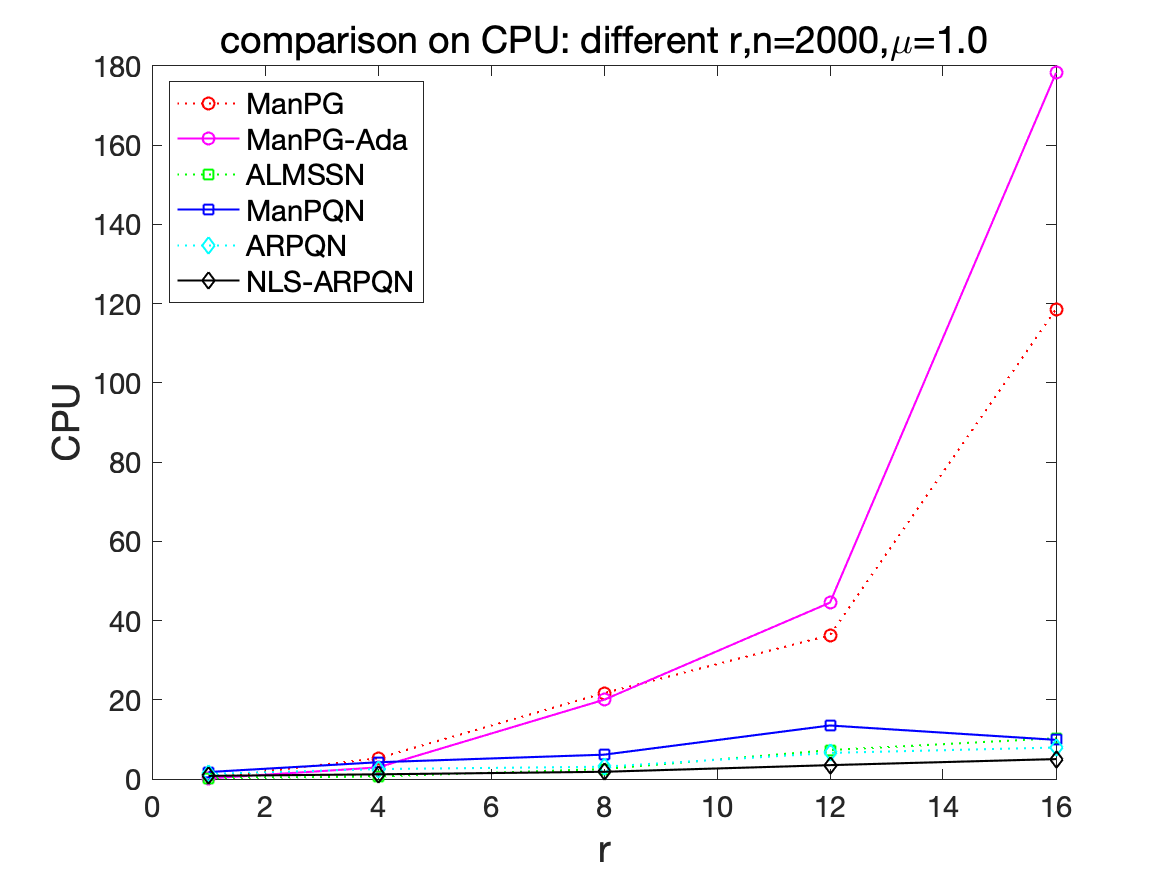}}
		\hspace{-5mm}
	\subfigure[Iter]{
		\includegraphics[width=0.45\linewidth]{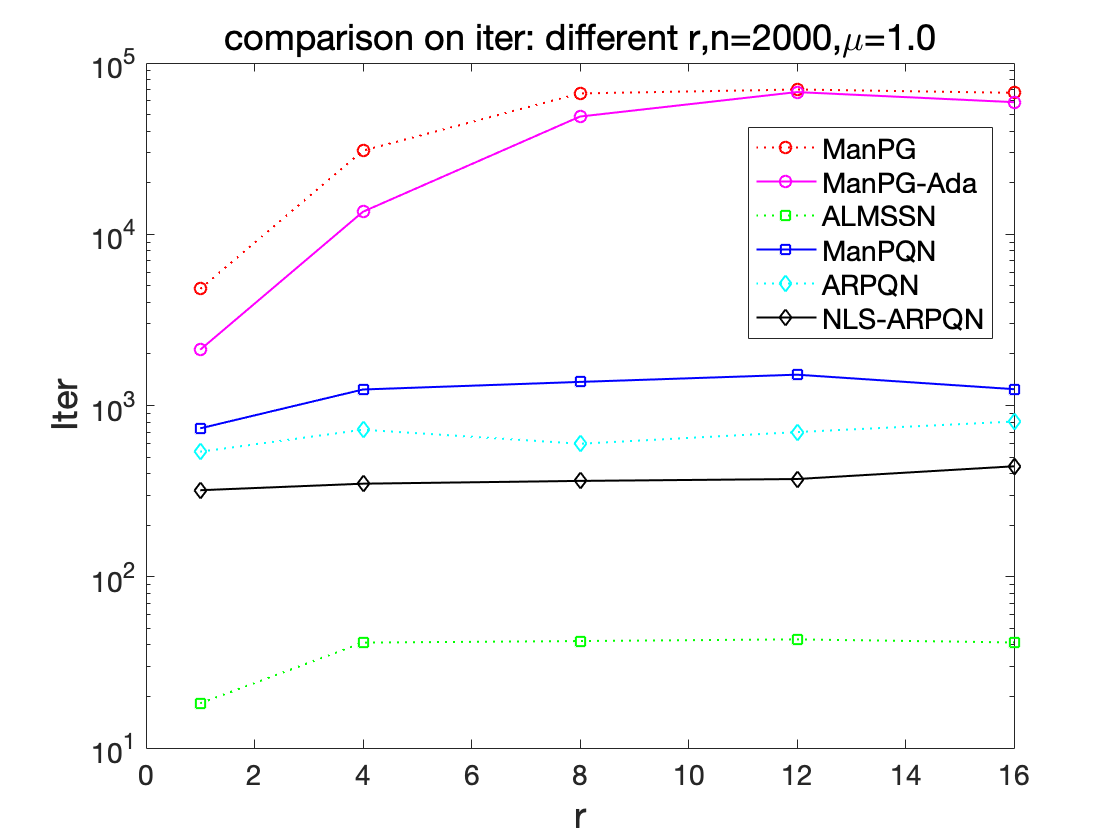}}
	\caption{Comparison on CM problem, different $r=\{1,4,8,12,16\}$ with $n=1024$ and $\mu=0.1$.}\label{fig51_22}
\end{figure}

\begin{figure}[H]
	\centering  
	\subfigbottomskip=2pt 
	\subfigcapskip=-5pt 
	\subfigure[CPU]{
		\includegraphics[width=0.45\linewidth]{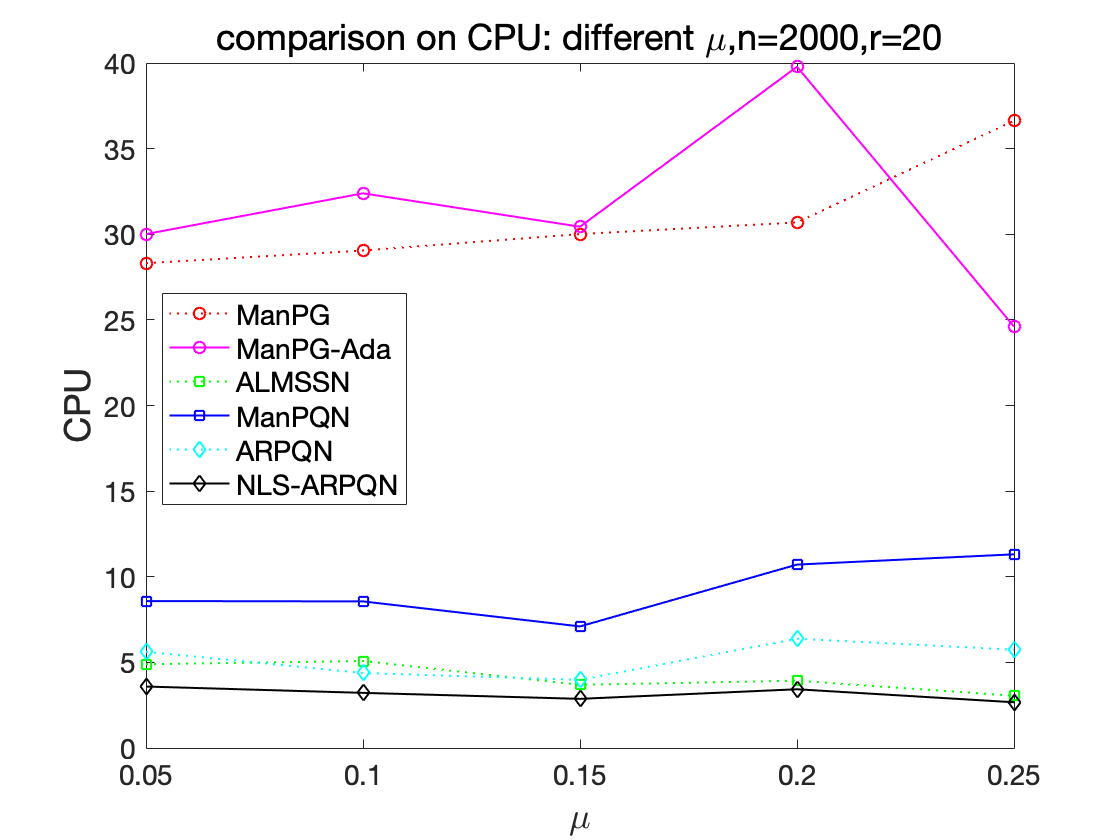}}
		\hspace{-5mm}
	\subfigure[Iter]{
		\includegraphics[width=0.45\linewidth]{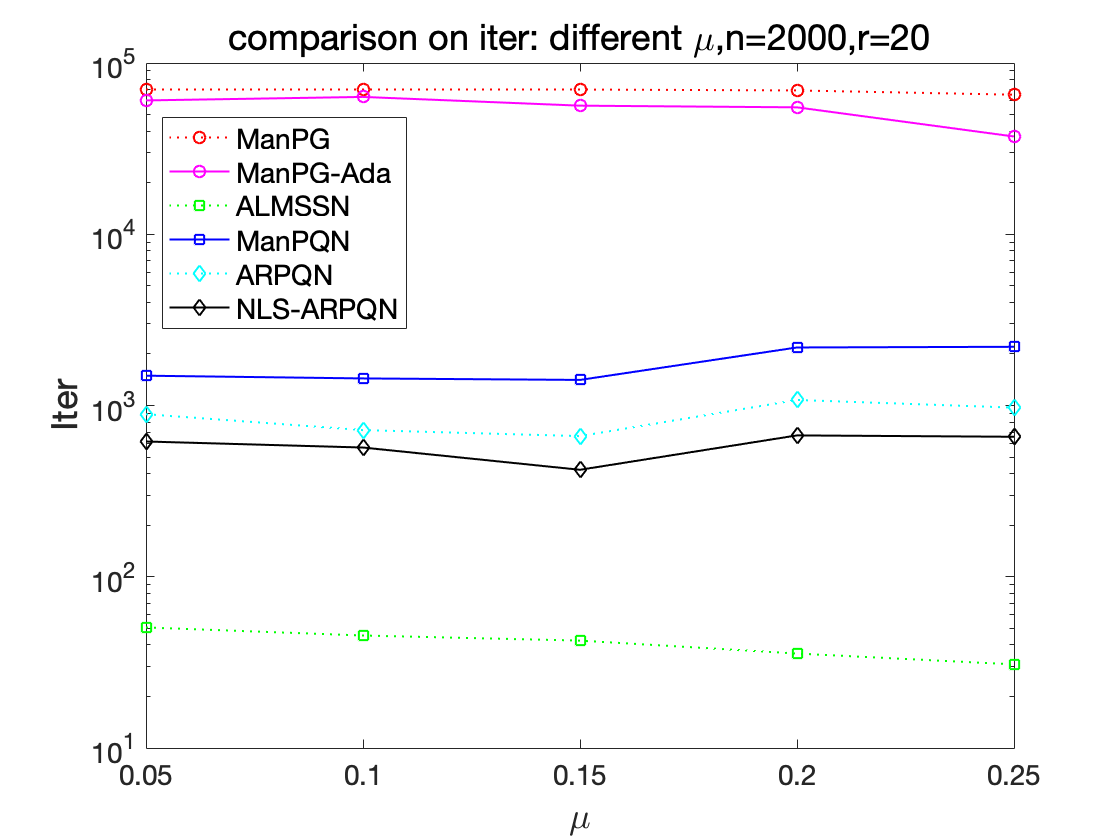}}
	\caption{Comparison on CM problem, different $\mu=\{0.05,0.10,0.15,0.20,0.25\}$ with $n= 1024$ and $r=10$.}\label{fig51_3}
\end{figure}

\begin{table}[htbp]\centering
\caption{Comparison on CM problem, different $n=\{128,256,512,1024,2048\}$ with $r= 10$ and $\mu=0.1$.}\label{tb51_1}
\begin{tabular}{ccccccc}
\midrule
$n=128$ & Iter & $F(X^*)$  & sparsity & CPU time & \# line-search &  SSN iters  \\
\midrule
ManPG     &   3276.68  & 4.815  &    0.73  &   1.1902   &   18.04   &   1.63  \\ 
ManPG-Ada &   1380.96  & 4.815  &    0.73  &   1.5263   &   795.96   &   5.19  \\ 
ALMSSN    &   45.08    & 4.815 &    0.49  &   0.4323   &   -    &    -  \\ 
ManPQN    &   673.44    & 4.825  &    0.73  &   1.0860   &   2240.92   &   7.88  \\ 
ARPQN     &   636.28    & 4.824  &    0.74  &   0.7717   &   2752.68   &   2.04  \\ 
NLS-ARPQN &   442.44    & 4.824  &    0.74  &   0.5016   &   2023.08   &   2.06  \\
\midrule
$n=256$ &   \\
\midrule
ManPG     &   13430.92  & 6.273  &    0.78  &   3.9509   &   260.96   &   0.77  \\ 
ManPG-Ada &   5655.52  & 6.273  &    0.78  &   5.5643   &   3940.16   &   2.94  \\ 
ALMSSN    &   50.88    & 6.272  &    0.58  &   1.4946   &   -    &    -  \\ 
ManPQN    &   1143.76    & 6.316  &    0.76  &   1.9971   &   5873.20   &   4.82  \\ 
ARPQN     &   1448.72    & 6.316  &    0.78  &   2.1861   &   9904.64   &   1.80  \\ 
NLS-ARPQN &   773.68    & 6.315  &    0.78  &   1.1588   &   5026.80   &   1.96  \\ 
\midrule
$n=512$ &   \\
\midrule
ManPG     &   66097.00  & 8.234  &    0.83  &   24.7815   &   2226.66   &   0.40  \\ 
ManPG-Ada &   35692.02  & 8.229  &    0.83  &   15.3766   &   23305.04  &   0.46  \\ 
ALMSSN    &   54.68    & 8.226  &    0.81  &   4.1034   &   -    &    -  \\ 
ManPQN    &   1161.82    & 8.343  &    0.82  &   3.8298   &   6656.54   &   2.37  \\ 
ARPQN     &   845.68    & 8.342  &    0.82  &   3.2429   &   6373.08   &   1.50  \\ 
NLS-ARPQN &   428.96    & 8.341  &    0.82  &   2.1450   &   2943.94   &   1.75  \\ 
\midrule
$n=1024$ &   \\
\midrule
ManPG     &   70001.00  & 10.867  &    0.86  &   27.9484   &   237.54   &   0.21  \\ 
ManPG-Ada &   61260.08  & 10.854  &    0.86  &   31.6759   &   37374.84   &   0.22  \\ 
ALMSSN    &   47.20    & 10.841  &    0.76  &   5.3273   &   -    &    -  \\ 
ManPQN    &   1201.68    & 10.967  &    0.85  &   6.4910   &   7561.10   &   1.96  \\ 
ARPQN     &   695.94    & 11.966  &    0.85  &   4.2407   &   5608.70   &   1.43  \\ 
NLS-ARPQN &   416.98    & 11.964  &    0.85  &   3.9036   &   2217.06   &   1.70  \\ 
\midrule
$n=2048$ &   \\
\midrule
ManPG     &   70001.00  & 14.360  &    0.88  &   36.3397   &   108.26   &   0.14  \\ 
ManPG-Ada &   69439.76  & 14.333  &    0.88  &   47.1851   &   41707.26   &   0.16  \\ 
ALMSSN    &   47.38    & 14.305  &    0.87  &   14.6289   &   -    &    -  \\ 
ManPQN    &   3578.12    & 14.603  &    0.88  &   27.7925   &   24124.88   &   1.11  \\ 
ARPQN     &   1404.62    & 14.600  &    0.88  &   19.9565   &   12084.12   &   1.17  \\ 
NLS-ARPQN &   603.76    & 14.595  &    0.88  &   8.6183   &   4893.88   &   1.62  \\ 
\midrule
\end{tabular}
\end{table}

\begin{table}[htbp]\centering
\caption{Comparison on CM problem, different $r=\{1,4,8,12,16\}$ with $n= 1024$ and $\mu=0.1$.}\label{tb51_22}
\begin{tabular}{ccccccc}
\midrule
$r=1$ & Iter & $F(X^*)$  & sparsity & CPU time & \# line-search &  SSN iters  \\
\midrule
ManPG     &   4795.54  & 1.084  &    0.89  &   0.2246   &   0.00   &   0.19  \\ 
ManPG-Ada &   2117.38  & 1.084  &    0.89  &   0.1412   &   1237.16   &   0.42  \\ 
ALMSSN    &   18.24    & 2.549  &    0.27  &   0.2035   &   -    &    -  \\ 
ManPQN    &   735.46    & 1.084  &    0.89  &   1.7698   &   4663.86   &   1.42  \\ 
ARPQN     &   537.24    & 1.084 &    0.89  &   1.2787   &   4019.86   &   1.13  \\ 
NLS-ARPQN &   320.06    & 1.084  &    0.89  &   0.7938   &   2111.00   &   1.16  \\ 
\midrule
$r=4$ &   \\
\midrule
ManPG     &   30793.08  & 4.337 &    0.88  &   5.2531   &   533.38   &   0.13  \\ 
ManPG-Ada &   13539.54  & 4.337  &    0.88  &   2.9336   &   7939.62   &   0.23  \\ 
ALMSSN    &   41.26    & 4.336  &    0.86  &   0.6555   &   -    &    -  \\ 
ManPQN    &   1240.04    & 4.353  &    0.88  &   4.2306   &   8540.00   &   1.12  \\ 
ARPQN     &   722.38    & 4.352  &    0.88  &   2.5177   &   5998.50   &   1.14  \\ 
NLS-ARPQN &   349.50    & 4.352  &    0.88  &   1.1812   &   2648.12   &   1.31  \\ 
\midrule
$r=8$ &   \\
\midrule
ManPG     &   66426.08  & 8.684  &    0.87  &   21.6938   &   215.64   &   0.26  \\ 
ManPG-Ada &   48678.56  & 8.677  &    0.87  &   20.0030   &   30866.64   &   0.23  \\ 
ALMSSN    &   42.12    & 8.672  &    0.84  &   2.5105   &   -    &    -  \\ 
ManPQN    &   1373.32    & 8.793  &    0.86  &   6.1420   &   8839.28   &   1.53  \\ 
ARPQN     &   597.76    & 8.791  &    0.86  &   3.0916   &   4686.96   &   1.33  \\ 
NLS-ARPQN &   362.80    & 8.790 &    0.86  &   1.8246   &   2656.08   &   1.56  \\ 
\midrule
$r=12$ &   \\
\midrule
ManPG     &   70001.00  & 13.103  &    0.84  &   36.3001   &   145.26   &   0.24  \\ 
ManPG-Ada &   67650.42  & 13.064  &    0.84  &   44.5306   &   39923.42   &   0.24  \\ 
ALMSSN    &   43.08    & 13.045  &    0.83  &   7.3484   &   -    &    -  \\ 
ManPQN    &   1513.42    & 13.391  &    0.85  &   13.5135   &   9749.76   &   2.17  \\ 
ARPQN     &   699.42    & 13.389  &    0.84  &   6.5995   &   5364.58   &   1.57  \\ 
NLS-ARPQN &   372.08    & 13.388  &    0.84  &   3.5000   &   2674.58   &   1.82  \\
\midrule
$r=16$ &   \\
\midrule
ManPG     &   66990.08  & 17.931  &    0.81  &   118.6215   &   16.50   &   2.28  \\ 
ManPG-Ada &   59054.76  & 17.894  &    0.81  &   178.2681   &   34562.08   &   7.34  \\ 
ALMSSN    &   41.34    & 17.889  &    0.80  &   10.2926   &   -    &    -  \\ 
ManPQN    &   1245.68    & 18.231  &    0.83  &   9.8806   &   8011.34   &   2.49  \\ 
ARPQN     &   806.34    & 18.229  &    0.83  &   7.9545   &   6591.18   &   1.58  \\ 
NLS-ARPQN &   441.58    & 18.228  &    0.83  &   5.0373   &   3308.00   &   1.93  \\  
\midrule
\end{tabular}
\end{table}

\begin{table}[htbp]\centering
\caption{Comparison on CM problem, different $\mu=\{0.05,0.10,0.15,0.20,0.25\}$ with $n=1024$ and $r=10$.}\label{tb51_3}
\begin{tabular}{ccccccc}
\midrule
$\mu=0.05$ & Iter & $F(X^*)$  & sparsity & CPU time & \# line-search &  SSN iters  \\
\midrule
ManPG     &   70001.00  & 6.308  &    0.77  &   28.3131   &   101.62   &   0.19  \\ 
ManPG-Ada &   60460.12  & 6.285  &    0.77  &   30.0166   &   35715.62   &   0.20  \\ 
ALMSSN    &   50.58    & 6.276  &    0.62  &   4.8953   &   -    &    -  \\ 
ManPQN    &   1493.02    & 6.467  &    0.80  &   8.5866   &   10019.88   &   2.49  \\ 
ARPQN     &   889.52    & 6.466  &    0.80  &   5.6202   &   7702.38   &   1.56  \\ 
NLS-ARPQN &   616.28    & 6.465  &    0.80  &   3.5975   &   4200.12   &   1.82  \\
\midrule
$\mu=0.10$ &   \\
\midrule
ManPG     &   70001.00  & 10.877  &    0.85  &   29.0577   &   186.50   &   0.22  \\ 
ManPG-Ada &   63374.00  & 10.857  &    0.86  &   32.3954   &   38428.62   &   0.21  \\ 
ALMSSN    &   45.38    & 10.841  &    0.78  &   5.0832   &   -    &    -  \\ 
ManPQN    &   1439.00    & 11.035  &    0.85  &   8.5663   &   9328.62   &   1.94  \\ 
ARPQN     &   718.62    & 11.034  &    0.85  &   4.3811   &   5876.00   &   1.44  \\ 
NLS-ARPQN &   568.26    & 11.033  &    0.85  &   3.2208   &   3612.12   &   1.71  \\
\midrule
$\mu=0.15$ &   \\
\midrule
ManPG     &   70001.00  & 15.013  &    0.88  &   30.0102   &   232.88   &   0.28  \\ 
ManPG-Ada &   56343.26  & 15.001  &    0.89  &   30.4436   &   35345.62   &   0.26  \\ 
ALMSSN    &   42.38    & 14.994  &    0.87  &   3.6974   &   -    &    -  \\ 
ManPQN    &   1410.50    & 15.161  &    0.88  &   7.1015   &   8928.38   &   1.58  \\ 
ARPQN     &   662.76    & 15.159  &    0.88  &   3.9839   &   5187.00   &   1.45  \\ 
NLS-ARPQN &   421.38    & 15.158  &    0.88  &   2.8780   &   2872.50   &   1.72  \\ 
\midrule
$\mu=0.20$ &   \\
\midrule
ManPG     &   69149.76  & 18.898  &    0.90  &   30.6943   &   5361.76   &   0.29  \\ 
ManPG-Ada &   55095.12  & 18.884  &    0.90  &   39.7875   &   97631.38   &   0.27  \\ 
ALMSSN    &   35.62    & 18.873  &    0.88  &   3.9380   &   -    &    -  \\ 
ManPQN    &   2182.00    & 19.101  &    0.89  &   10.7175   &   14111.26   &   1.36  \\ 
ARPQN     &   1077.38    & 19.100  &    0.89  &   6.3806   &   8865.24   &   1.34  \\ 
NLS-ARPQN &   668.88    & 19.098  &    0.89  &   3.4334   &   4450.78   &   1.64  \\  
\midrule
$\mu=0.25$ &   \\
\midrule
ManPG     &   65293.26  & 22.571  &    0.91  &   36.6351   &   67515.76   &   0.26  \\ 
ManPG-Ada &   37221.22  & 22.565  &    0.91  &   24.6111   &   58067.88   &   0.24  \\ 
ALMSSN    &   30.78    & 22.562  &    0.90  &   3.0534   &   -    &    -  \\ 
ManPQN    &   2202.00    & 22.698  &    0.91  &   11.3181   &   14126.38   &   1.23  \\ 
ARPQN     &   971.88    & 22.697  &    0.91  &   5.7411   &   8093.88   &   1.18  \\ 
NLS-ARPQN &   656.74    & 22.696  &    0.91  &   2.6786   &   3343.72   &   1.46  \\ 
\midrule
\end{tabular}
\end{table}

\subsection{Sparse PCA}\label{sec52}

The sparse principal component analysis (sparse PCA) \cite{spca2006} is a statistical problem, which aims to find $r$ ($r<\min\{m,n\}$) principal components with sparse loadings for a given data $A\in\mathbb{R}^{m\times n}$. Sparse PCA problem can be written as 
\begin{equation}
\label{eq:5_spca}
\min_{X\in\mathcal{M}}-tr(X^TA^TAX)+\mu\|X\|_1.
\end{equation}

In the following subsection, we compare the performance of ManPG, ManPG-Ada, ALMSSN, ManPQN, ARPQN and NLS-ARPQN for solving \eqref{eq:5_spca}. The termination conditions and values of parameters are set as those for the CM problems in subsection \ref{sec51}. The matrix $A\in\mathbb{R}^{m\times n}$ is generated by normal distribution with $m=50$. For each case with different $(n,r,\mu)$, experiments are repeated for $50$ times with randomly generated matrices $A$ and initial points, and averaged numerical results for each algorithm are presented below.

\begin{figure}[H]
	\centering  
	\subfigbottomskip=2pt 
	\subfigcapskip=-5pt 
	\subfigure[CPU]{
		\includegraphics[width=0.33\linewidth]{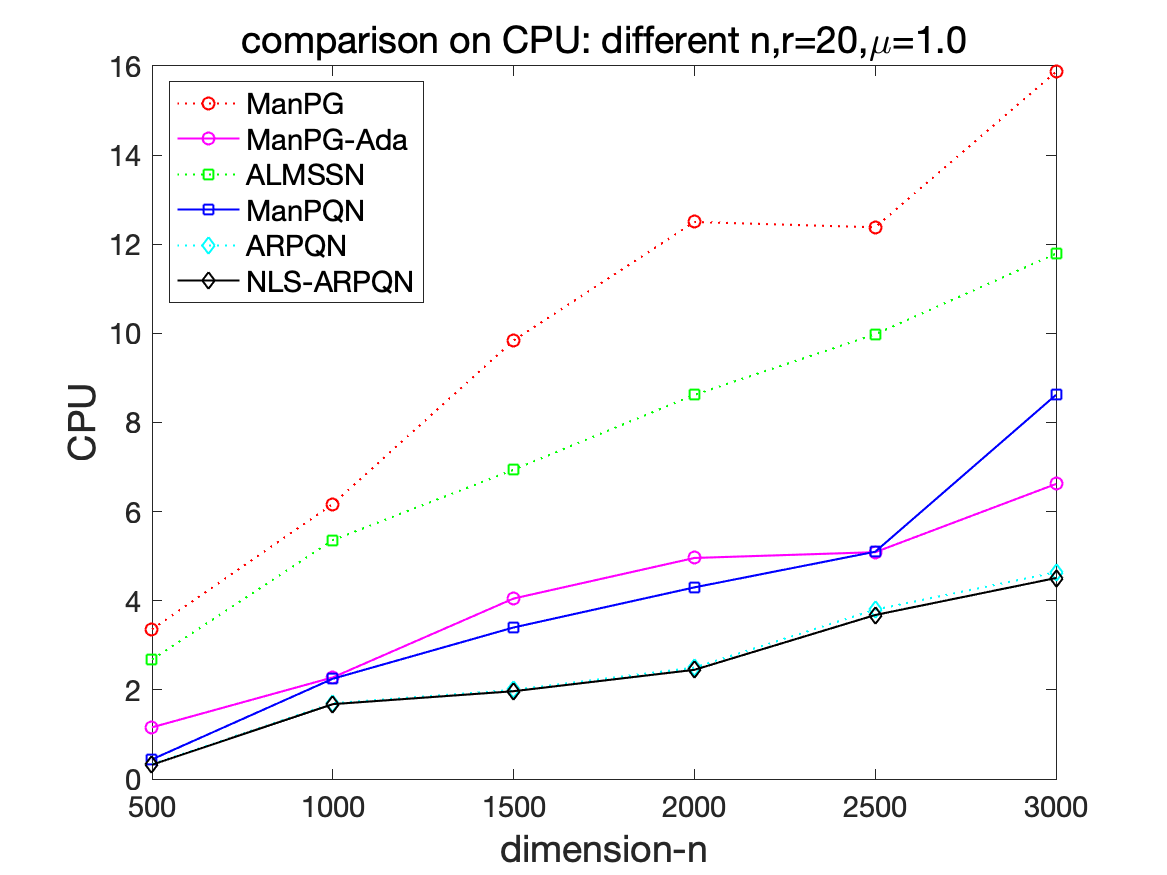}}
		\hspace{-5mm}
	\subfigure[Iter]{
		\includegraphics[width=0.33\linewidth]{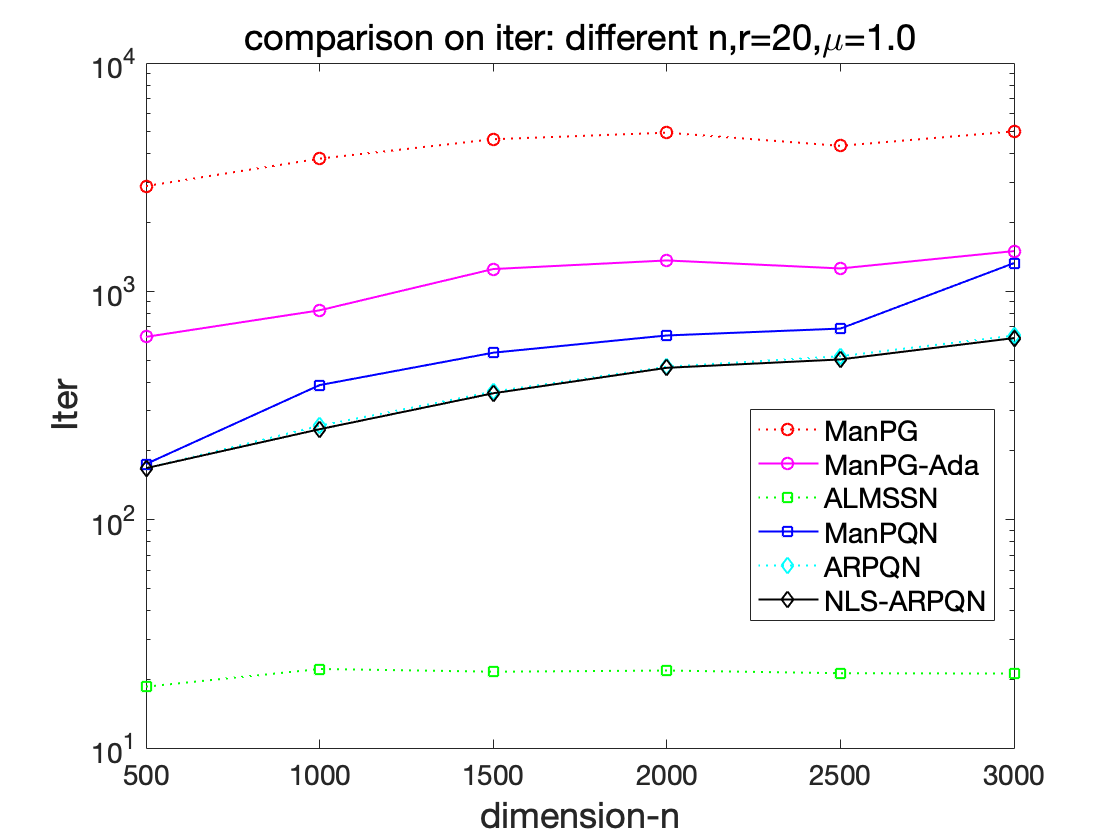}}
		\hspace{-5mm}
	\subfigure[Sparsity]{
		\includegraphics[width=0.33\linewidth]{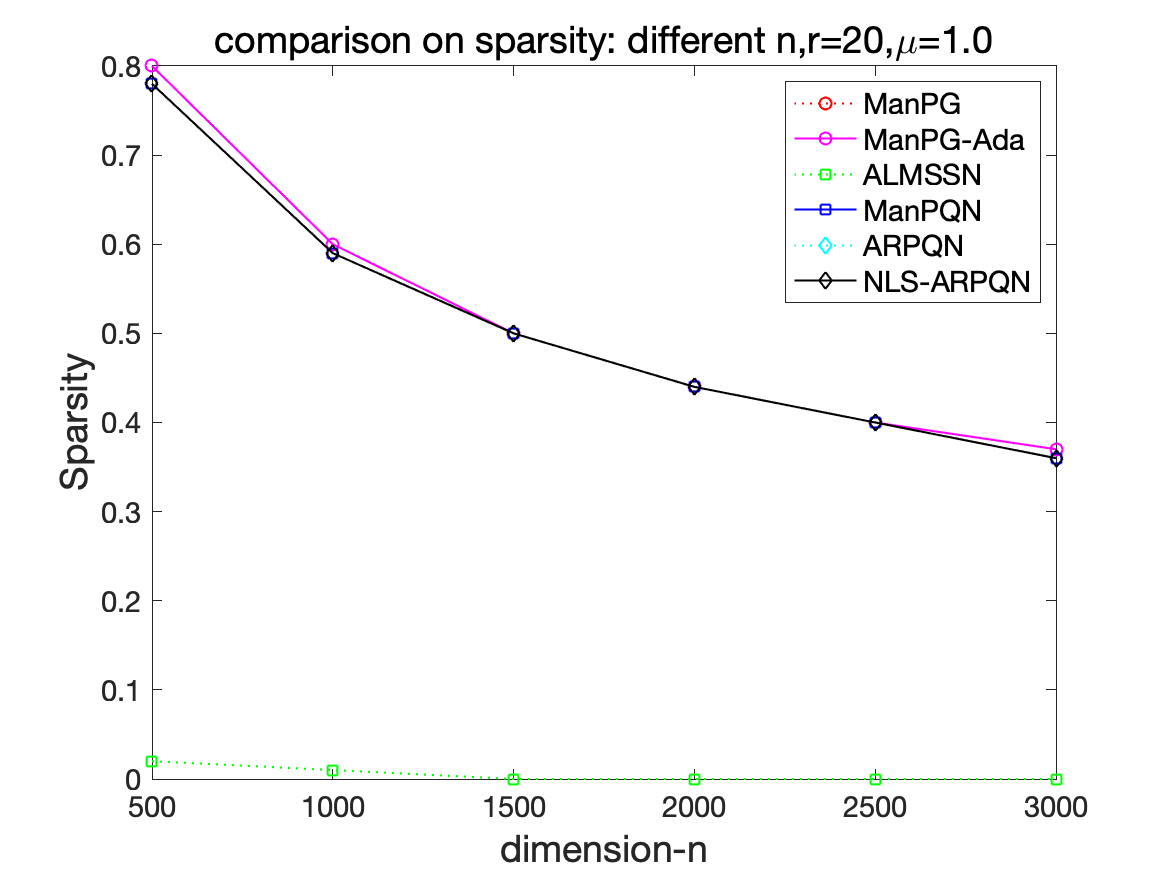}}
	\caption{Comparison on Sparse PCA problem, different $n=\{500, 1000, 1500, 2000, 2500, 3000\}$ with $r= 20$ and $\mu=1.0$.}\label{fig52_1}
\end{figure}

\begin{figure}[H]
	\centering  
	\subfigbottomskip=2pt 
	\subfigcapskip=-5pt 
	\subfigure[CPU]{
		\includegraphics[width=0.33\linewidth]{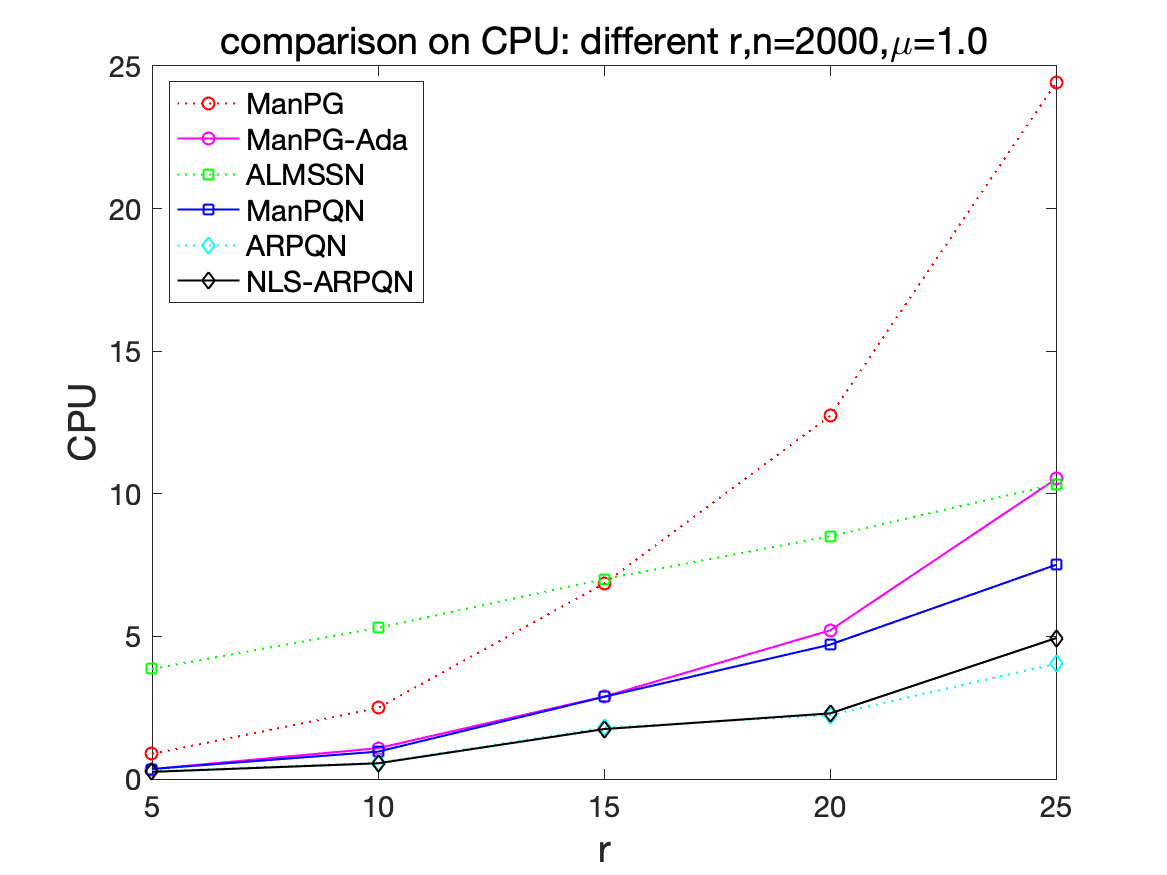}}
		\hspace{-5mm}
	\subfigure[Iter]{
		\includegraphics[width=0.33\linewidth]{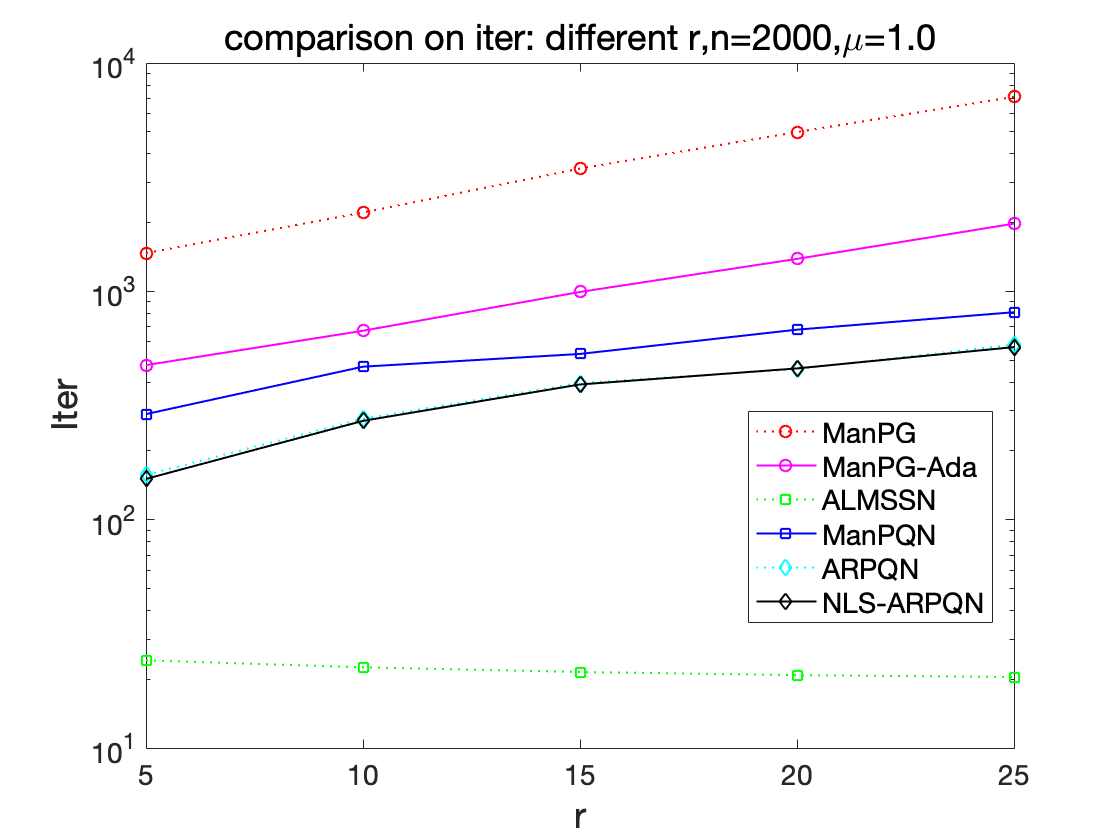}}
		\hspace{-5mm}
	\subfigure[Sparsity]{
		\includegraphics[width=0.33\linewidth]{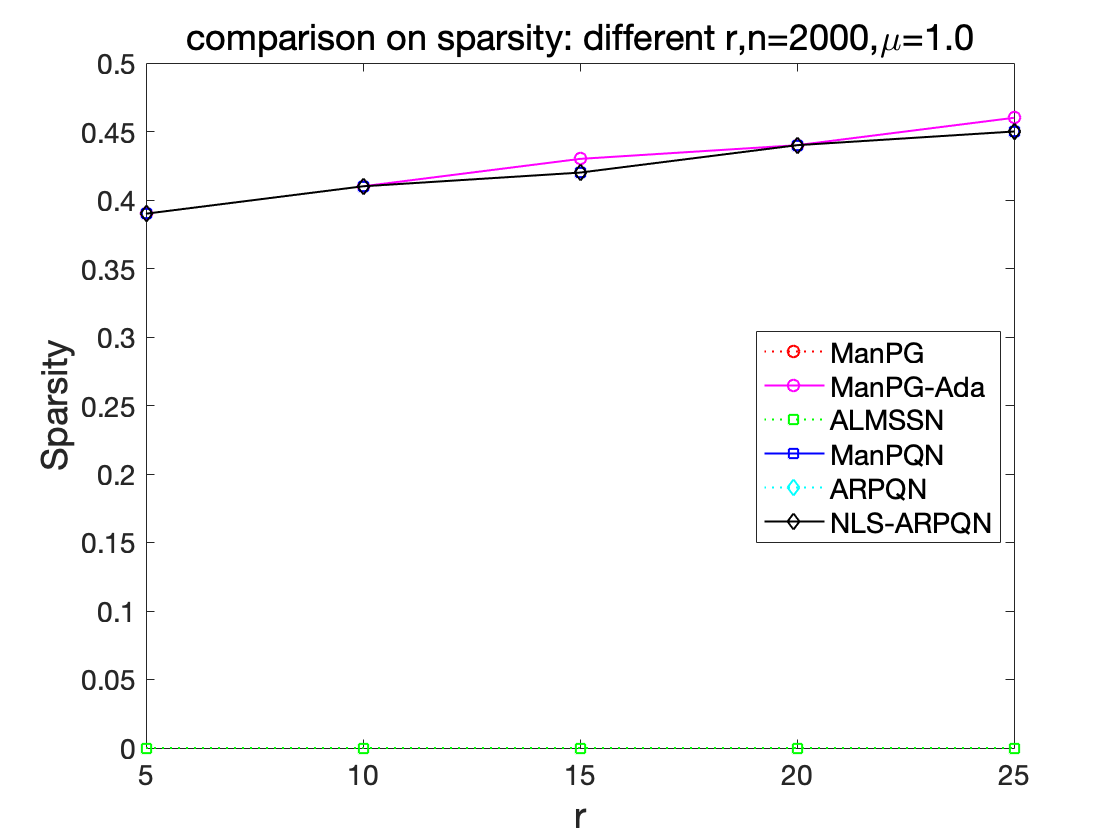}}
	\caption{Comparison on Sparse PCA problem, different $r=\{5,10,15,20,25\}$ with $n= 2000$ and $\mu=1.0$.}\label{fig52_2}
\end{figure}

\begin{figure}[H]
	\centering  
	\subfigbottomskip=2pt 
	\subfigcapskip=-5pt 
	\subfigure[CPU]{
		\includegraphics[width=0.33\linewidth]{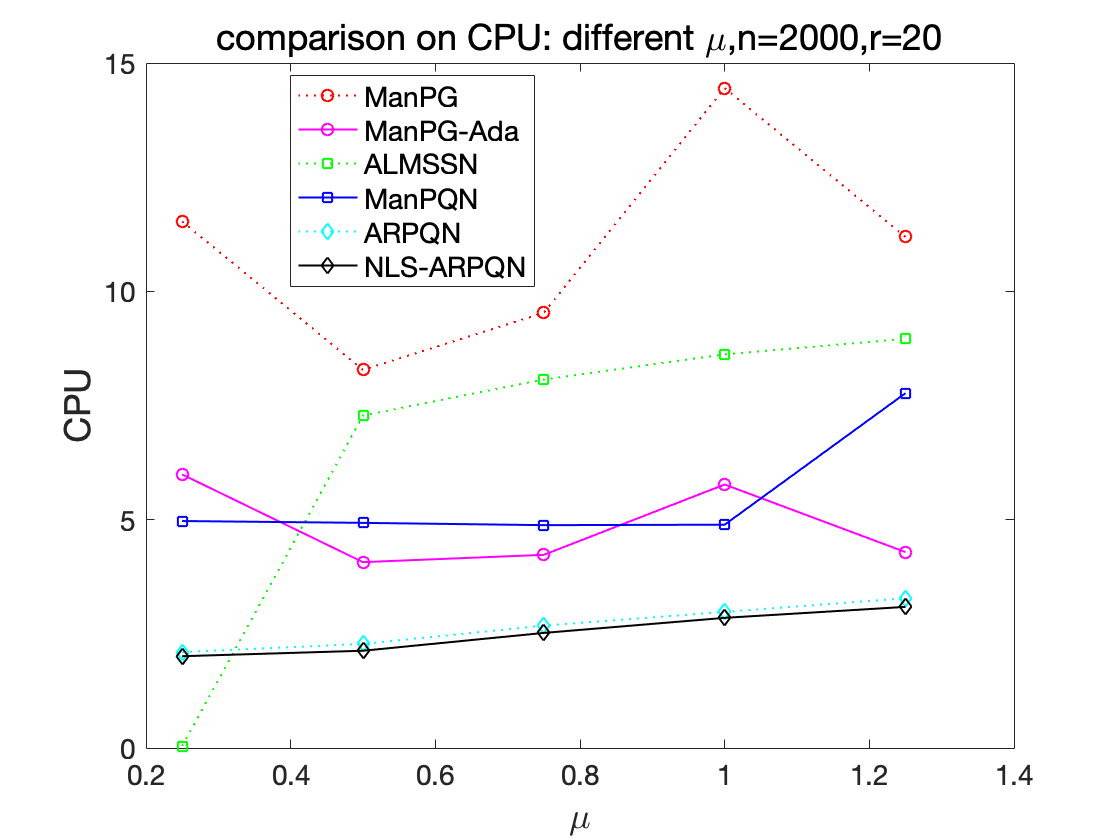}}
		\hspace{-5mm}
	\subfigure[Iter]{
		\includegraphics[width=0.33\linewidth]{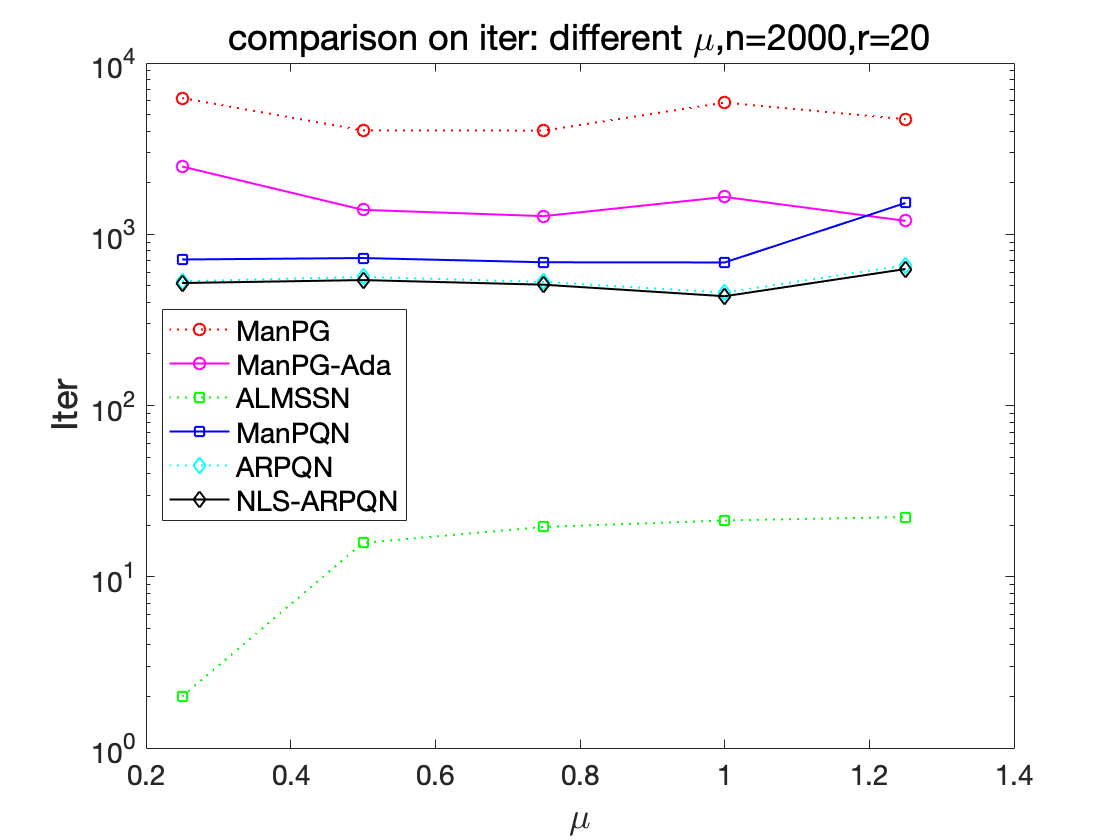}}
		\hspace{-5mm}
	\subfigure[Sparsity]{
		\includegraphics[width=0.33\linewidth]{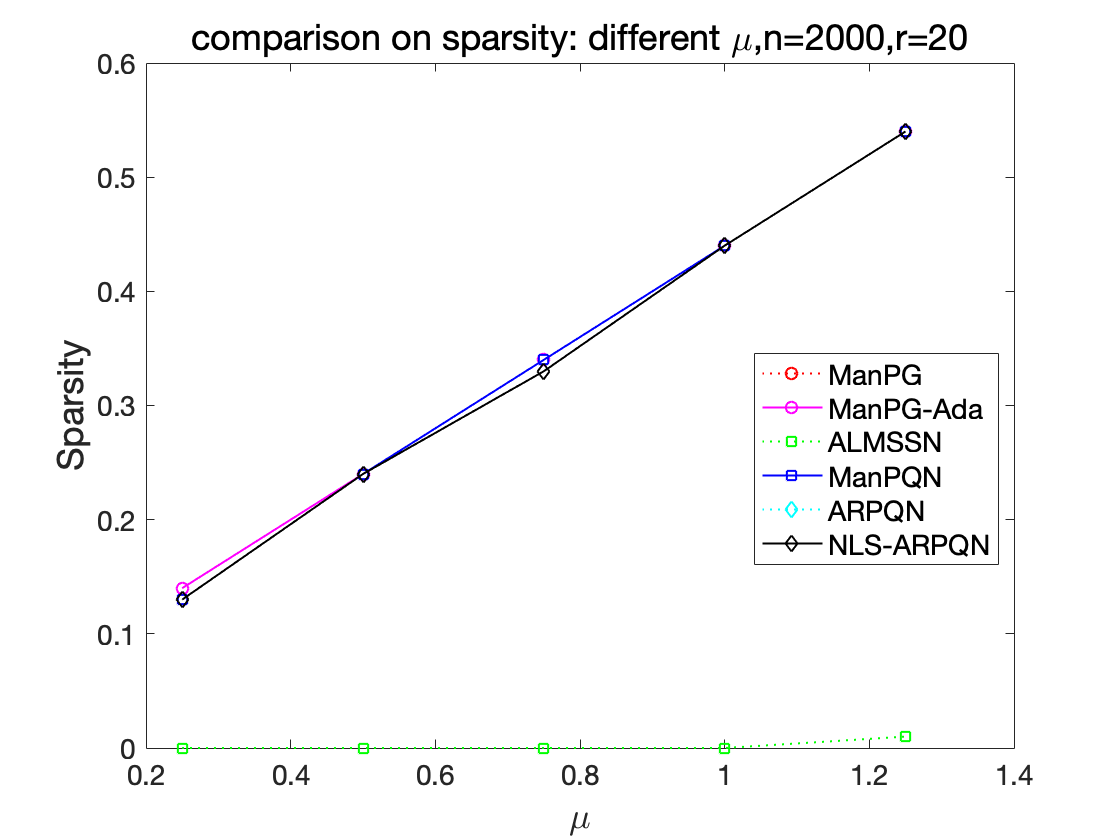}}
	\caption{Comparison on Sparse PCA problem, different $\mu=\{0.25,0.50,0.75,1.0,1.25\}$ with $n= 2000$ and $r=20$.}\label{fig52_32}
\end{figure}

Figures \ref{fig52_1}-\ref{fig52_32} and Tables \ref{tb52_1}-\ref{tb52_32} show the numerical performance of ManPG, ManPG-Ada, ALMSSN, ManPQN, ARPQN and NLS-ARPQN on the sparse PCA problems with different $n$, $r$ and $\mu$.
We can observe that NLS-ARPQN outperforms other algorithms in terms of iteration number and running time, especially when $n$ and $r$ are large. 
In some cases, ManPG and ManPG-Ada can achieve an optimal solution with slightly better sparsity than NLS-ARPQN and other algorithms.
Moreover, we report the total number of line search and the averaged iteration number of ASSN in the last two columns of Tables \ref{tb52_1}-\ref{tb52_32}. 
It is observed that NLS-ARPQN requires less averaged line search steps and ASSN iterations than ManPQN, which shows that the regularization technique can accelerate the convergence of our method and reduce the computational cost of solving the subproblem.
In most cases, NLS-ARPQN runs faster than ARPQN slightly.

\begin{table}[htbp]\centering
\caption{Comparison on Sparse PCA problem, different $n= \{500,1000,1500,2000,2500,3000\}$ with $r= 20$ and $\mu=1.0$.}\label{tb52_1}
\begin{tabular}{ccccccc}
\midrule
$n=500$ & Iter & $F(X^*)$  & sparsity & CPU time & \# line-search &  SSN iters  \\
\midrule
ManPG     &   2890.90  & -18.906  &    0.80  &   3.3681   &   8.90   &   1.17  \\ 
ManPG-Ada &   631.70  & -18.906  &    0.80  &   1.1635   &   268.28   &   1.99  \\ 
ALMSSN    &   18.56    & -19.030  &    0.02  &   2.6851   &   -    &    -  \\ 
ManPQN   &   175.28  & -18.449  &    0.78  &   0.4435   &   614.78   &   2.55  \\ 
ARPQN     &   167.80  & -18.431  &    0.78  &   0.3361   &   301.73   &   2.17  \\ 
NLS-ARPQN &   168.06  & -18.434  &    0.78  &   0.3232   &   329.08   &   2.17  \\ 
\midrule
$n=1000$ &   \\
\midrule
ManPG     &   3812.10  & -101.422  &    0.60  &   6.1787   &   0.06   &   1.09  \\ 
ManPG-Ada &   925.78  & -101.422  &    0.60  &   2.2805   &   528.32   &   1.81  \\ 
ALMSSN    &   22.18    & -100.600  &    0.01  &   5.3643   &   -    &    -  \\ 
ManPQN   &   388.42  & -100.620  &    0.59  &   2.2520   &   1424.52   &   2.39  \\ 
ARPQN     &   257.74  & -100.622  &    0.59  &   1.7096   &   440.42   &   2.17  \\ 
NLS-ARPQN &   248.96  & -100.626  &    0.59  &   1.6824   &   436.80   &   2.18  \\ 
\midrule
$n=1500$ &   \\
\midrule
ManPG     &   4635.64  & -212.686  &    0.50  &   9.8426   &   116.68   &   1.08  \\ 
ManPG-Ada &   1252.00  & -212.686  &    0.50  &   4.0527   &   748.94   &   1.75  \\ 
ALMSSN    &   21.58    & -211.680  &    0.00  &   6.9429   &   -    &    -  \\ 
ManPQN   &   538.56  & -211.531  &    0.50  &   3.4023   &   1260.08   &   2.85  \\ 
ARPQN     &   363.58  & -211.534  &    0.50  &   2.0068   &   543.44   &   2.17  \\ 
NLS-ARPQN &   357.66  & -211.536  &    0.50  &   1.9726   &   539.30   &   2.18  \\ 
\midrule
$n=2000$ &   \\
\midrule
ManPG     &   4953.10  & -336.927  &    0.44  &   12.5064   &   0.00   &   1.07  \\ 
ManPG-Ada &   1365.54  & -336.887  &    0.44  &   4.9687   &   10708.18   &   1.63  \\ 
ALMSSN    &   21.84    & -335.940  &    0.00  &   8.6223   &   -    &    -  \\ 
ManPQN   &   640.80  & -335.563  &    0.44  &   4.3081   &   1390.92   &   2.48  \\ 
ARPQN     &   467.62  & -335.566  &    0.44  &   2.5867   &   747.86   &   2.17  \\ 
NLS-ARPQN &   462.50  & -335.570  &    0.44  &   2.4501   &   746.02   &   2.17  \\ 
\midrule
$n=2500$ &   \\
\midrule
ManPG     &   4336.80  & -471.822  &    0.40  &   12.3885   &   0.00   &   1.07  \\ 
ManPG-Ada &   1259.64  & -471.822  &    0.40  &   5.0965   &   727.08   &   1.66  \\ 
ALMSSN    &   21.26    & -470.634  &    0.00  &   9.9808   &   -    &    -  \\ 
ManPQN   &   686.84  & -470.770  &    0.40  &   5.1039   &   1237.54   &   2.21  \\ 
ARPQN     &    519.78  & -470.772  &    0.40  &   3.8117   &   810.44   &   2.12  \\ 
NLS-ARPQN &  503.22  & -470.774  &    0.40  &   3.6804   &   810.14   &   2.13  \\ 
\midrule
$n=3000$ &   \\
\midrule
ManPG     &   5017.68  & -612.863  &    0.37  &   15.8727   &   0.00   &   1.05  \\ 
ManPG-Ada &   1500.76  & -612.822  &    0.37  &   6.6238   &   5928.24   &   1.59  \\ 
ALMSSN    &   21.18    & -611.560  &    0.00  &   11.7954   &   -    &    -  \\ 
ManPQN   &   1329.38  & -611.847  &    0.36  &   8.6239   &   1582.10   &   1.32  \\ 
ARPQN     &    640.04  & -611.849  &    0.36  &   4.6484   &   886.12   &   2.10  \\ 
NLS-ARPQN &  624.00  & -611.851  &    0.36  &   4.5124   &   885.96   &   2.11  \\  
\midrule
\end{tabular}
\end{table}

\begin{table}[htbp]\centering
\caption{Comparison on Sparse PCA problem, different $r=\{5,10,15,20,25\}$ with $n= 2000$ and $\mu=1.0$.}\label{tb52_2}
\begin{tabular}{ccccccc}
\midrule
$r=5$ & Iter & $F(X^*)$  & sparsity & CPU time & \# line-search &  SSN iters  \\
\midrule
ManPG     &   1470.40  & -100.375  &    0.39  &   0.8873   &   46.24   &   1.01  \\ 
ManPG-Ada &   475.20  & -100.375  &    0.39  &   0.3569   &   192.44   &   1.18  \\ 
ALMSSN    &   24.20    & -100.030  &    0.00  &   3.8681   &   -    &    -  \\ 
ManPQN   &   289.96  & -100.259  &    0.39  &   0.3530   &   567.60   &   1.79  \\ 
ARPQN     &   157.00  & -100.260  &    0.39  &   0.2783   &   434.36   &   1.73  \\ 
NLS-ARPQN &   150.84  & -100.260  &    0.39  &   0.2555   &   435.28   &   1.73  \\ 
\midrule
$r=10$ &   \\
\midrule
ManPG     &   2214.00  & -188.350  &    0.41  &   2.4909   &   0.00   &   1.04  \\ 
ManPG-Ada &   672.08  & -188.350  &    0.41  &   1.0854   &   325.08   &   1.56  \\ 
ALMSSN    &   22.56    & -188.040  &    0.00  &   5.3061   &   -    &    -  \\ 
ManPQN   &   467.76  & -187.900  &    0.41  &   0.9634   &   802.20   &   2.53  \\ 
ARPQN     &   276.64  & -187.901  &    0.41  &   0.5753   &   568.40   &   2.06  \\  
NLS-ARPQN &   270.88  & -187.901  &    0.41  &   0.5537   &   565.68   &   2.06  \\
\midrule
$r=15$ &   \\
\midrule
ManPG     &   3448.88  & -267.409  &    0.43  &   6.8556   &   0.00   &   1.05  \\ 
ManPG-Ada &   993.96  & -267.409  &    0.43  &   2.8994   &   541.72   &   1.61  \\ 
ALMSSN    &   21.52    & -266.50  &    0.00  &   7.0158   &   -    &    -  \\ 
ManPQN   &   531.88  & -266.610  &    0.42  &   2.8864   &   938.04   &   2.84  \\ 
ARPQN     &   396.68  & -266.612  &    0.42  &   1.8022   &   638.60   &   2.17  \\
NLS-ARPQN &   390.96  & -266.612  &    0.42  &   1.7595   &   636.60   &   2.17  \\ 
\midrule
$r=20$ &   \\
\midrule
ManPG     &   4983.54  & -337.312  &    0.44  &   12.7328   &   0.00   &   1.05  \\ 
ManPG-Ada &   1387.32  & -337.312 &    0.44  &   5.2116   &   815.38   &   1.68  \\ 
ALMSSN    &   20.86    & -336.220  &    0.00  &   8.5104   &   -    &    -  \\ 
ManPQN   &   679.46  & -335.902  &    0.44  &   4.7155   &   1125.00   &   2.48  \\ 
ARPQN     &   456.54  & -335.904  &    0.44  &   2.2499   &   740.00   &   2.18  \\ 
NLS-ARPQN &   459.62  & -335.907  &    0.44  &   2.3021   &   744.32   &   2.17  \\ 
\midrule
$r=25$ &   \\
\midrule
ManPG     &   7111.38  & -400.625  &    0.46  &   24.4102   &   0.00   &   1.08  \\ 
ManPG-Ada &   1979.08  & -400.625  &    0.46  &   10.5233   &   1308.16   &   1.73  \\ 
ALMSSN    &   20.46    & -398.470  &    0.00  &   10.3405   &   -    &    -  \\ 
ManPQN   &   810.16  & -398.337  &    0.45  &   7.5133   &   1488.78   &   2.39  \\ 
ARPQN     &   584.92  & -398.340  &    0.45  &   4.0411   &   957.70   &   2.16  \\ 
NLS-ARPQN &   570.00  & -398.345  &    0.45  &   4.9429   &   952.78  &   2.17  \\ 
\midrule
\end{tabular}
\end{table}

\begin{table}[htbp]\centering
\caption{Comparison on Sparse PCA problem, different $\mu=\{0.25,0.50,0.75,1.0,1.25\}$ with $n= 2000$ and $r=20$.}\label{tb52_32}
\begin{tabular}{ccccccc}
\midrule
$\mu=0.25$ & Iter & $F(X^*)$  & sparsity & CPU time & \# line-search &  SSN iters  \\
\midrule
ManPG     &   6243.92  & -777.077  &    0.14  &   11.5347   &   0.00   &   1.01  \\ 
ManPG-Ada &   2485.24  & -777.077  &    0.14  &   5.9953   &   1513.24   &   1.34  \\ 
ALMSSN    &   2.00    & -775.770  &    0.00  &   0.0468   &   -    &    -  \\ 
ManPQN   &   712.16  & -776.669  &    0.13  &   4.9737   &   1391.62   &   1.92  \\ 
ARPQN     &  530.68  & -776.669  &    0.13  &   2.1068   &   1197.62   &   2.23  \\ 
NLS-ARPQN &   519.00  & -776.670  &    0.13  &   2.0125   &   1178.62   &   2.20  \\ 
\midrule
$\mu=0.50$ & \\
\midrule
ManPG     &   4039.54  & -618.884  &    0.24  &   8.2814   &   0.00   &   1.03  \\ 
ManPG-Ada &   1388.08  & -619.001  &    0.24  &   4.0700   &   19141.54   &   1.47  \\ 
ALMSSN    &   15.78    & -618.200  &    0.00  &   7.2838   &   -    &    -  \\ 
ManPQN   &   726.32  & -618.395  &    0.24  &   4.9301   &   1484.68   &   2.28  \\ 
ARPQN     &   563.54  & -618.395  &    0.24  &   2.2898   &   1080.78   &   2.22  \\ 
NLS-ARPQN &   540.08  & -618.396  &    0.24  &   2.1372   &   1061.86   &   2.24  \\  
\midrule
$\mu=0.75$ & \\
\midrule
ManPG     &   4029.62  & -472.161  &    0.34  &   9.5323   &   0.00   &   1.05  \\ 
ManPG-Ada &   1275.92  & -472.087  &    0.34  &   4.2344   &   20258.46   &   1.47  \\ 
ALMSSN    &   19.54    & -471.790  &    0.00  &   8.0744   &   -    &    -  \\ 
ManPQN   &   686.00  & -471.759  &    0.34  &   4.8830   &   1333.46   &   2.52  \\ 
ARPQN     &   526.46  & -471.759  &    0.33  &   2.6849   &   983.46   &   2.28  \\
NLS-ARPQN &    507.38  & -471.760  &    0.33  &   2.5220   &   977.00   &   2.30  \\ 
\midrule
$\mu=1.0$ & \\
\midrule
ManPG     &   5873.16  & -336.964  &    0.44  &   14.4531   &   0.00   &   1.06  \\ 
ManPG-Ada &   1651.46  & -336.867  &    0.44  &   5.7722   &   39999.70   &   1.37  \\ 
ALMSSN    &   21.32    & -335.770  &    0.00  &   8.6273   &   -    &    -  \\ 
ManPQN   &   684.24  & -336.097  &    0.44  &   4.8974   &   1268.78   &   2.48  \\ 
ARPQN     &   456.70  & -336.099  &    0.44  &   2.9879   &   722.78   &   2.16  \\ 
NLS-ARPQN &   434.54  & -336.100  &    0.44  &   2.8547   &   720.16   &   2.18  \\ 
\midrule
$\mu=1.25$ & \\
\midrule
ManPG     &   4696.16  & -214.463  &    0.54  &   11.1935   &   0.16   &   1.08  \\ 
ManPG-Ada &   1197.70  & -214.463  &    0.54  &   4.2976   &   727.24   &   1.76  \\ 
ALMSSN    &   22.32    & -212.050  &    0.01  &   8.9605   &   -    &    -  \\ 
ManPQN   &   1522.92  & -213.368  &    0.54  &   7.7783   &   1745.92   &   1.48  \\ 
ARPQN     &    658.42  & -213.369  &    0.54  &   3.2862   &   1213.92   &   2.52  \\ 
NLS-ARPQN &   626.58  & -213.380  &    0.54  &   3.0959   &   1054.68   &   2.59  \\ 
\midrule
\end{tabular}
\end{table}

\section{Conclusion}
\label{sec:6}

In this paper, we propose two adaptive regularized proximal Newton-type methods, ARPQN and ARPN, for the composite optimization problem \eqref{eq:prob} over the Stiefel manifold. The ARPQN method can be regarded as a variant of the ManPQN algorithm proposed in \cite{manpqn2023}. Specifically, at each iterate, the quadratic model, used in the proximal mapping of ARPQN, is formed by adding a regularization term to that used in ManPQN.
This adaptive regularization strategy can be used to reduce the overall computational cost of solving \eqref{eq:prob}.
Analysis of the global convergence and the iteration complexity of ARPQN is established, and the local linear convergence rate is proved under the strong convexity assumption on the objective function. Numerical results demonstrate that the adaptive regularization strategy can be used to accelerate the proximal quasi-Newton method.
The subproblem of ARPN is formed by replacing the term $h(X_k+V)$ by $h(\mathbf{R}_{X_k}(V))$ in \eqref{eq:sub_prob}.
We establish the global convergence and the local superlinear convergence of ARPN.
We only present the numerical results of ARPQN
since solving the subproblem of ARPN is so expensive that the total computational cost of ARPN is considerably high.

As shown in the numerical experiments, the computational cost of ARPQN mainly lies in solving the subproblem \eqref{eq:sub_prob} by the ASSN method, which grows rapidly as the dimensions $n$ and $r$ of the problem increase. One topic of our future work is to design a first-order method to solve the subproblem. This paper only focuses on composite optimization over the Stiefel manifold. Naturally, another topic of our future work is to extend ARPQN and ARPN to general Riemannian manifolds.
Additionally, it would be valuable to further investigate the practical implementation of ARPN in solving large-scale composite optimization problems.

\bmhead{Acknowledgments}

The work of Wei Hong Yang was supported by the National Natural Science Foundation of China grant 72394365.
%
%

\bmhead{Data Availability}

The data that support the findings of this study are available from the corresponding author upon request.

\bmhead{Conflict of Interest}

The authors have no relevant financial or non-financial interests to disclose.





\bibliographystyle{plain}
\nocite{proxmanifold2005}

\bibliography{ref_ARPQN}


%

\end{document}